\documentclass[10pt]{amsart}
\usepackage{amsmath}
\usepackage{amsxtra}
\usepackage{amscd}
\usepackage{amsthm}
\usepackage{amsfonts}
\usepackage{amssymb}
\usepackage{eucal}
\usepackage[all]{xy}
\usepackage{graphicx}
\usepackage{stmaryrd}

\usepackage[usenames]{color}

\newtheorem{cor}[subsubsection]{Corollary}
\newtheorem{lem}[subsubsection]{Lemma}
\newtheorem{prop}[subsubsection]{Proposition}

\newtheorem{thm}[subsubsection]{Theorem}

\theoremstyle{definition}
\newtheorem{defn}[subsubsection]{Definition}
\newtheorem{example}[subsubsection]{Example}

\newcommand{\iso}{\buildrel{\sim}\over{\longrightarrow}}

\theoremstyle{definition}

\theoremstyle{remark}
\newtheorem{rem}[subsubsection]{Remark}

\newcommand{\secref}[1]{Sect.~\ref{#1}}
\newcommand{\lemref}[1]{Lemma~\ref{#1}}
\newcommand{\propref}[1]{Proposition~\ref{#1}}
\newcommand{\corref}[1]{Corollary~\ref{#1}}

\newcommand{\remref}[1]{Remark~\ref{#1}}

\numberwithin{equation}{section}

\newcommand{\nc}{\newcommand}
\nc{\renc}{\renewcommand}
\nc{\ssec}{\subsection}
\nc{\sssec}{\subsubsection}
\nc{\on}{\operatorname}

\nc{\Kdot}{{\bullet}}

\nc\ol{\overline}
\nc\wt{\widetilde}
\nc\tboxtimes{\wt{\boxtimes}}
\nc{\alp}{\alpha}

\newcommand{\limto}{{\displaystyle\lim_{\longrightarrow}}}
\newcommand{\rightlim}{\mathop{\limto}}

\newcommand{\leftlim}{\mathop{\displaystyle\lim_{\longleftarrow}}}
\newcommand{\limfromn}{\leftlim\limits_{\raise3pt\hbox{$n$}}}
\newcommand{\limton}{\rightlim\limits_{\raise3pt\hbox{$n$}}}

\nc{\subscheme}{{``sub"\-scheme }}
\nc{\subschemes}{{``sub"schemes }}
\nc{\subspaces}{{``sub"spaces }}
\nc{\subspace}{{``sub"space }}

\newcommand{\mono}{\hookrightarrow}
\newcommand{\epi}{\twoheadrightarrow}

\nc{\Zp}{{Z_{\infty}^+}}
\nc{\Rn}{{R_n}}

\nc{\Corr}{\on{Corr}}
\nc{\Sets}{\on{Sets}}
\nc{\Gp}{\on{Groupoids}}
\nc{\Alg}{\on{AlgSp}}
\nc{\AAlg}{\on{{\bf AlgSp}}}
\nc{\St}{\on{AlgSt}}
\nc{\SSt}{\on{{\bf AlgSt}}}
\nc{\Sp}{\on{{\bf Sp}}}
\nc{\CORR}{\on{{\bf Corr}}}
\nc{\TTw}{\on{{\bf Tw}}}

\nc{\DGCat}{\on{DGCat}}

\nc{\BIG}{{\Phi}}
\nc{\ZZ}{{\mathbb Z}}
\nc{\NN}{{\mathbb N}}
\nc{\OO}{{\mathbb O}}
\renc{\SS}{{\mathbb S}}
\nc{\DD}{{\mathbb D}}
\nc{\GG}{{\mathbb G}}

\nc{\Fq}{{\mathbb F}_q}
\nc{\Fqb}{\ol{{\mathbb F}_q}}
\nc{\Ql}{\ol{{\mathbb Q}_\ell}}
\nc{\id}{\text{id}}
\nc\X{\mathcal X}

\nc{\bl}{\on{bl}}
\nc{\st}{\on{st}}
\nc{\der}{\on{der}}
\nc{\Eq}{\on{Eq}}
\nc{\ET}{\on{et}}
\nc{\Et}{\on{Et}}
\nc{\ev}{\on{ev}}
\nc{\fin}{\on{fin}}
\nc{\Hom}{\on{Hom}}
\nc{\Ob}{\on{Ob}}
\nc{\obs}{\underline{obs}}
\nc{\QC}{\on{QC}}
\nc{\Lie}{\on{Lie}}
\nc{\Loc}{\on{Loc}}
\nc{\Pic}{\on{Pic}}
\nc{\Bun}{\on{Bun}}
\nc{\IC}{\on{IC}}
\nc{\Aut}{\on{Aut}}
\nc{\pr}{\on{pr}}
\nc{\rk}{\on{rk}}
\nc{\Sh}{\on{Sh}}
\nc{\Perv}{\on{Perv}}
\nc{\pos}{{\on{pos}}}
\nc{\Conv}{\on{Conv}}
\nc{\sep}{\on{sep}}
\nc{\Sph}{\on{Sph}}
\nc{\Sym}{\on{Sym}}
\nc{\Tw}{\on{Tw}}
\nc{\BunBb}{\overline{\Bun}_B}
\nc{\BunBbm}{\overline{\Bun}_{B^-}}
\nc{\BunBbel}{\overline{\Bun}_{B,el}}
\nc{\BunBbmel}{\overline{\Bun}_{B^-,el}}
\nc{\Buno}{\overset{o}{\Bun}}
\nc{\BunPb}{{\overline{\Bun}_P}}
\nc{\BunBM}{\Bun_{B(M)}}
\nc{\BunBMb}{\overline{\Bun}_{B(M)}}
\nc{\BunPbw}{{\widetilde{\Bun}_P}}
\nc{\BunBP}{\widetilde{\Bun}_{B,P}}
\nc{\GUb}{\overline{G/U}}
\nc{\GUPb}{\overline{G/U(P)}}

\nc{\Hhom}{\underline{\on{Hom}}}
\nc\syminfty{\on{Sym}^{\infty}}
\nc\lal{\ol{\lambda}}
\nc\xl{\ol{x}}
\nc\thl{\ol{\theta}}
\nc\nul{\ol{\nu}}
\nc\mul{\ol{\mu}}
\nc\Sum\Sigma
\nc{\oX}{\overset{o}{X}{}}
\nc{\hl}{\overset{\leftarrow}h{}}
\nc{\hr}{\overset{\rightarrow}h{}}
\nc{\M}{{\mathcal M}}
\nc{\N}{{\mathcal N}}
\nc{\F}{{\mathcal F}}
\nc{\D}{{\mathcal D}}
\nc{\Q}{{\mathcal Q}}
\nc{\Y}{{\mathcal Y}}
\nc{\G}{{\mathcal G}}
\nc{\E}{{\mathcal E}}
\nc{\CalC}{{\mathcal C}}
\nc\Dh{\widehat{\D}}

\nc{\C}{{\mathcal C}}
\nc{\K}{{\mathcal K}}
\renewcommand{\H}{{\mathcal H}}

\nc{\T}{{\mathcal T}}
\nc{\V}{{\mathcal V}}
\renc{\P}{{\mathcal P}}
\nc{\A}{{\mathcal A}}
\nc{\B}{{\mathcal B}}
\nc{\U}{{\mathcal U}}

\nc{\Gr}{{\on{Gr}}}

\nc{\frn}{{\check{\mathfrak u}(P)}}
\nc{\p}{\mathfrak p}
\nc{\q}{\mathfrak q}
\nc\f{{\mathfrak f}}

\nc{\qo}{{\mathfrak q}}
\nc{\po}{{\mathfrak p}}
\nc{\s}{{\mathfrak s}}
\nc\w{\text{w}}

\nc\mathi\iota
\nc\Spec{\on{Spec}}
\nc\Spf{\on{Spf}}
\nc\Mod{\on{Mod}}
\nc{\tw}{\widetilde{\mathfrak t}}
\nc{\pw}{\widetilde{\mathfrak p}}
\nc{\qw}{\widetilde{\mathfrak q}}
\nc{\jw}{\widetilde j}

\nc{\grb}{\overline{\Gr}}
\nc{\I}{\mathcal I}

\nc{\lambdach}{{\check\lambda}}
\nc{\Lambdach}{{\check\Lambda}{}}
\nc{\much}{{\check\mu}}
\nc{\omegach}{{\check\omega}}
\nc{\nuch}{{\check\nu}}
\nc{\etach}{{\check\eta}}
\nc{\alphach}{{\check\alpha}}
\nc{\betach}{{\check\beta}}
\nc{\rhoch}{{\check\rho}}
\nc{\ch}{{\check h}}

\nc{\Hb}{\overline{\H}}

\nc{\BA}{{\mathbb{A}}}
\nc{\BB}{{\mathbb{B}}}
\nc{\BC}{{\mathbb{C}}}
\nc{\BG}{{\mathbb{G}}}
\nc{\BM}{{\mathbb{M}}}
\nc{\BD}{{\mathbb{D}}}
\nc{\BN}{{\mathbb{N}}}
\nc{\BP}{{\mathbb{P}}}
\nc{\BQ}{{\mathbb{Q}}}
\nc{\BR}{{\mathbb{R}}}
\nc{\BX}{{\mathbb{X}}}
\nc{\BXp}{\mathbb{X}^+}
\nc{\BXm}{\mathbb{X}^-}
\nc{\BXpm}{\mathbb{X}^{\pm}}

\nc{\BY}{{\mathbb{Y}}}
\nc{\BZ}{{\mathbb{Z}}}
\nc{\BS}{{\mathbb{S}}}

\nc{\CA}{{\mathcal{A}}}
\nc{\CB}{{\mathcal{B}}}

\nc{\CE}{{\mathcal{E}}}
\nc{\CF}{{\mathcal{F}}}
\nc{\CG}{{\mathcal{G}}}

\nc{\CL}{{\mathcal{L}}}
\nc{\CC}{{\mathcal{C}}}
\nc{\CCD}{{\mathcal{D}}}
\nc{\CM}{{\mathcal{M}}}
\nc{\CN}{{\mathcal{N}}}
\nc{\CK}{{\mathcal{K}}}
\nc{\CO}{{\mathcal{O}}}
\nc{\CP}{{\mathcal{P}}}
\nc{\CQ}{{\mathcal{Q}}}
\nc{\CR}{{\mathcal{R}}}
\nc{\CS}{{\mathcal{S}}}
\nc{\CT}{{\mathcal{T}}}
\nc{\CU}{{\mathcal{U}}}
\nc{\CV}{{\mathcal{V}}}
\nc{\CW}{{\mathcal{W}}}
\nc{\CX}{{\mathcal{X}}}
\nc{\CY}{{\mathcal{Y}}}
\nc{\CZ}{{\mathcal{Z}}}
\nc{\CI}{{\mathcal{I}}}
\nc{\CJ}{{\mathcal{J}}}

\nc{\csM}{{\check{\mathcal A}}{}}
\nc{\oM}{{\overset{\circ}{\mathcal M}}{}}
\nc{\obM}{{\overset{\circ}{\mathbf M}}{}}
\nc{\oCA}{{\overset{\circ}{\mathcal A}}{}}
\nc{\obA}{{\overset{\circ}{\mathbf A}}{}}
\nc{\ooM}{{\overset{\circ}{M}}{}}
\nc{\osM}{{\overset{\circ}{\mathsf M}}{}}
\nc{\vM}{{\overset{\bullet}{\mathcal M}}{}}
\nc{\nM}{{\underset{\bullet}{\mathcal M}}{}}
\nc{\oD}{{\overset{\circ}{\mathcal D}}{}}
\nc{\obD}{{\overset{\circ}{\mathbf D}}{}}
\nc{\oA}{{\overset{\circ}{\mathbb A}}{}}
\nc{\op}{{\overset{\bullet}{\mathbf p}}{}}
\nc{\cp}{{\overset{\circ}{\mathbf p}}{}}
\nc{\oU}{{\overset{\bullet}{\mathcal U}}{}}
\nc{\oZ}{{\overset{\circ}{\mathcal Z}}{}}
\nc{\ofZ}{{\overset{\circ}{\mathfrak Z}}{}}
\nc{\oF}{{\overset{\circ}{\fF}}}

\nc{\fa}{{\mathfrak{a}}}
\nc{\fb}{{\mathfrak{b}}}
\nc{\fd}{{\mathfrak{d}}}
\nc{\fg}{{\mathfrak{g}}}
\nc{\fgl}{{\mathfrak{gl}}}
\nc{\fh}{{\mathfrak{h}}}
\nc{\fj}{{\mathfrak{j}}}
\nc{\fl}{{\mathfrak{l}}}
\nc{\fm}{{\mathfrak{m}}}
\nc{\fn}{{\mathfrak{n}}}
\nc{\fu}{{\mathfrak{u}}}
\nc{\fp}{{\mathfrak{p}}}
\nc{\fr}{{\mathfrak{r}}}
\nc{\fs}{{\mathfrak{s}}}
\nc{\ft}{{\mathfrak{t}}}
\nc{\fsl}{{\mathfrak{sl}}}
\nc{\hsl}{{\widehat{\mathfrak{sl}}}}
\nc{\hgl}{{\widehat{\mathfrak{gl}}}}
\nc{\hg}{{\widehat{\mathfrak{g}}}}
\nc{\chg}{{\widehat{\mathfrak{g}}}{}^\vee}
\nc{\hn}{{\widehat{\mathfrak{n}}}}
\nc{\chn}{{\widehat{\mathfrak{n}}}{}^\vee}

\nc{\fA}{{\mathfrak{A}}}
\nc{\fB}{{\mathfrak{B}}}
\nc{\fD}{{\mathfrak{D}}}
\nc{\fE}{{\mathfrak{E}}}
\nc{\fF}{{\mathfrak{F}}}
\nc{\fG}{{\mathfrak{G}}}
\nc{\fK}{{\mathfrak{K}}}
\nc{\fL}{{\mathfrak{L}}}
\nc{\fM}{{\mathfrak{M}}}
\nc{\fN}{{\mathfrak{N}}}
\nc{\fP}{{\mathfrak{P}}}
\nc{\fU}{{\mathfrak{U}}}
\nc{\fV}{{\mathfrak{V}}}
\nc{\fZ}{{\mathfrak{Z}}}

\nc{\bb}{{\mathbf{b}}}
\nc{\bc}{{\mathbf{c}}}
\nc{\bd}{{\mathbf{d}}}
\nc{\be}{{\mathbf{e}}}
\nc{\bj}{{\mathbf{j}}}
\nc{\bn}{{\mathbf{n}}}
\nc{\bp}{{\mathbf{p}}}
\nc{\bq}{{\mathbf{q}}}
\nc{\bu}{{\mathbf{u}}}
\nc{\bv}{{\mathbf{v}}}
\nc{\bx}{{\mathbf{x}}}
\nc{\bs}{{\mathbf{s}}}
\nc{\by}{{\mathbf{y}}}
\nc{\bw}{{\mathbf{w}}}
\nc{\bA}{{\mathbf{A}}}
\nc{\bK}{{\mathbf{K}}}
\nc{\bB}{{\mathbf{B}}}
\nc{\bC}{{\mathbf{C}}}
\nc{\bG}{{\mathbf{G}}}
\nc{\bD}{{\mathbf{D}}}
\nc{\bH}{{\mathbf{H}}}
\nc{\bM}{{\mathbf{M}}}
\nc{\bN}{{\mathbf{N}}}
\nc{\bV}{{\mathbf{V}}}
\nc{\bW}{{\mathbf{W}}}
\nc{\bX}{{\mathbf{X}}}
\nc{\bZ}{{\mathbf{Z}}}
\nc{\bS}{{\mathbf{S}}}

\nc{\sA}{{\mathsf{A}}}
\nc{\sB}{{\mathsf{B}}}
\nc{\sC}{{\mathsf{C}}}
\nc{\sD}{{\mathsf{D}}}
\nc{\sG}{{\mathsf{G}}}
\nc{\sF}{{\mathsf{F}}}
\nc{\sK}{{\mathsf{K}}}
\nc{\sM}{{\mathsf{M}}}
\nc{\sO}{{\mathsf{O}}}
\nc{\sW}{{\mathsf{W}}}
\nc{\sQ}{{\mathsf{Q}}}
\nc{\sP}{{\mathsf{P}}}
\nc{\sZ}{{\mathsf{Z}}}
\nc{\sfp}{{\mathsf{p}}}
\nc{\sfq}{{\mathsf{q}}}
\nc{\sr}{{\mathsf{r}}}
\nc{\sfi}{{\mathsf{i}}}
\nc{\sk}{{\mathsf{k}}}
\nc{\sg}{{\mathsf{g}}}
\nc{\sff}{{\mathsf{f}}}
\nc{\sfb}{{\mathsf{b}}}
\nc{\sfc}{{\mathsf{c}}}
\nc{\sd}{{\mathsf{d}}}

\nc{\BK}{{\bar{K}}}

\nc{\tA}{{\widetilde{\mathbf{A}}}}
\nc{\tB}{{\widetilde{\mathcal{B}}}}
\nc{\tg}{{\widetilde{\mathfrak{g}}}}
\nc{\tG}{{\widetilde{G}}}
\nc{\TM}{{\widetilde{\mathbb{M}}}{}}
\nc{\tO}{{\widetilde{\mathsf{O}}}{}}
\nc{\tU}{{\widetilde{\mathfrak{U}}}{}}
\nc{\TZ}{{\tilde{Z}}}
\nc{\tx}{{\tilde{x}}}
\nc{\tbv}{{\tilde{\bv}}}
\nc{\tfP}{{\widetilde{\mathfrak{P}}}{}}
\nc{\tz}{{\tilde{\zeta}}}
\nc{\tmu}{{\tilde{\mu}}}

\nc{\urho}{\underline{\rho}}
\nc{\uB}{\underline{B}}
\nc{\uC}{{\underline{\mathbb{C}}}}
\nc{\ui}{\underline{i}}
\nc{\uj}{\underline{j}}
\nc{\ofP}{{\overline{\mathfrak{P}}}}
\nc{\oB}{{\overline{\mathcal{B}}}}
\nc{\og}{{\overline{\mathfrak{g}}}}
\nc{\oI}{{\overline{I}}}

\nc{\eps}{\varepsilon}
\nc{\hrho}{{\hat{\rho}}}

\nc{\one}{{\mathbf{1}}}
\nc{\two}{{\mathbf{t}}}

\nc{\Rep}{{\mathop{\operatorname{\rm Rep}}}}
\nc{\Tot}{{\mathop{\operatorname{\rm Tot}}}}
\nc{\Ker}{{\mathop{\operatorname{\rm Ker}}}}
\nc{\Coker}{{\mathop{\operatorname{\rm Coker}}}}
\nc{\im}{{\mathop{\operatorname{\rm Im}}}}
\nc{\Hilb}{{\mathop{\operatorname{\rm Hilb}}}}
\nc{\End}{{\mathop{\operatorname{\rm End}}}}
\nc{\Ext}{{\mathop{\operatorname{\rm Ext}}}}
\nc{\CHom}{{\mathop{\operatorname{{\mathcal{H}}\it om}}}}
\nc{\GL}{{\mathop{\operatorname{\rm GL}}}}
\nc{\gr}{{\mathop{\operatorname{\rm gr}}}}
\nc{\Id}{{\mathop{\operatorname{\rm Id}}}}
\nc{\de}{{\mathop{\operatorname{\rm def}}}}
\nc{\length}{{\mathop{\operatorname{\rm length}}}}
\nc{\supp}{{\mathop{\operatorname{\rm supp}}}}

\nc{\Mor}{{\mathop{\operatorname{{\mathcal{M}}\it or}}}}

\nc{\Cliff}{{\mathsf{Cliff}}}
\nc{\Fl}{\on{Fl}}
\nc{\Fib}{{\mathsf{Fib}}}
\nc{\Coh}{{\mathsf{Coh}}}
\nc{\FCoh}{{\mathsf{FCoh}}}

\nc{\reg}{{\text{\rm reg}}}

\nc{\cplus}{{\mathbf{C}_+}}
\nc{\cminus}{{\mathbf{C}_-}}
\nc{\cthree}{{\mathbf{C}_\bullet}}
\nc{\Qbar}{{\bar{Q}}}
\nc\Eis{\on{Eis}}
\nc\Eisb{\ol\Eis{}}
\nc\wh{\widehat}
\nc{\Def}{\on{Def_{\check{\fb}}(E)}}
\nc{\barZ}{\overline{Z}{}}
\nc{\barbarZ}{\overline{\barZ}{}}
\nc{\barpi}{\overline\pi}
\nc{\barbarpi}{\overline\barpi}
\nc{\barpip}{\overline\pi{}^+}
\nc{\barpim}{\overline\pi{}^-}

\nc{\fq}{\mathfrak q}

\nc{\sfqb}{\ol{\sfq}{}}
\nc{\sfpb}{\ol{\sfp}{}}

\nc{\hattimes}{\wh\otimes}

\nc{\bh}{{\bar{h}}}
\nc{\bOmega}{{\overline{\Omega(\check \fn)}}}

\nc{\seq}[1]{\stackrel{#1}{\sim}}

\nc{\cT}{{\check{T}}}
\nc{\cG}{{\check{G}}}
\nc{\cM}{{\check{M}}}
\nc{\cB}{{\check{B}}}

\nc{\ct}{{\check{\mathfrak t}}}
\nc{\cg}{{\check{\fg}}}
\nc{\cb}{{\check{\fb}}}
\nc{\cn}{{\check{\fn}}}

\nc{\cLambda}{{\check\Lambda}}

\nc{\cla}{{\check\lambda}}
\nc{\cmu}{{\check\mu}}
\nc{\cnu}{{\check\nu}}
\nc{\ceta}{{\check\eta}}

\nc{\DefbE}{{\on{Def}_{\cB}(E_\cT)}}

\nc{\imathb}{{\ol{\imath}}}
\nc{\Dmod}{\on{D-mod}}
\nc{\Maps}{\on{Maps}}
\nc{\MMaps}{\on{\mathbf{Maps}}}
\nc{\GMaps}{{\MMaps^{\BG_m}}}
\nc{\gMaps}{{\Maps^{\BG_m}}}
\nc{\Vect}{\on{Vect}}
\nc{\CMaps}{\mathcal Maps}
\nc{\sotimes}{\overset{!}\otimes}
\nc{\dr}{\on{dR}}

\nc{\red}{\on{red}}

\begin{document}

\title{On algebraic spaces with an action of $\BG_m$}

\author{V.~Drinfeld}

\date{\today}

\begin{abstract}
Let $Z$ be an algebraic space of finite type over a field, equipped with an action of the multiplicative group $\BG_m$. In this situation we define and study a certain algebraic space equipped with an unramified morphism to $\BA^1\times Z\times Z$ (if $Z$ is affine and smooth this is just the closure of the graph of the action map $\BG_m\times Z\to Z$). In particular, we prove the results on $\wt{Z}$ announced in \cite{DrGa1}.

In articles joint with D.~Gaitsgory we use this set-up to prove a new result in the geometric theory of automorphic forms and to give a new proof of a very important theorem of T.~Braden.
\end{abstract}

\subjclass{Primary 14L30; Secondary 14D24, 20G05, 22E57}

\keywords{Toric action, algebraic space, hyperbolic, fixed points, attractor, repeller}

\maketitle

\tableofcontents


\section*{Introduction}

\ssec{The goals of this article}

Algebraic varieties equipped with an action of the multiplicative group $\BG_m$ have been studied for quite a while (especially by A.~Bia{\l}ynicki-Birula and his school); see, e.g., the works \cite{Bia},  \cite{BS}, \cite{Ju1},  \cite{Ju2},  \cite{Kon},  \cite{Som}.

\medskip

This article has two goals.

The first one is to define the notion of attractor for an \emph{arbitrary algebraic space}\footnote{We do not require $Z$ to be either separated or normal. We include quasi-separatedness in the definition of algebraic space, but this is a very weak property (which holds automatically for \emph{schemes} of finite type over $k$).} of finite type over a field $k$ equipped with a $\BG_m$-action and to prove the basic properties of attractors in this generality. The main difficulty is that if $(Z\otimes_k\bar k)_{\red}$ is not assumed to be a normal scheme then Sumihiro's theorem is not applicable, so the $\BG_m$-action on $Z\otimes_k\bar k$ is not necessarily locally 
linear\footnote{An action of $\BG_m$ on a scheme $Z$ is said to be \emph{locally linear} if
$Z$ can be covered by open affine subschemes preserved by the $\BG_m$-action.}.

The second (and more important) goal is to provide the geometric background for the articles 
\cite{DrGa1, DrGa2}. Namely, for any algebraic $k$-space $Z$ of finite type acted on by $\BG_m\,$, we define a certain algebraic space $\wt{Z}$ of finite type over $\BA^1\times Z\times Z$ and study its properties. The space 
$\wt{Z}$ seems to be new even if $Z$ is a separated smooth scheme (although it is somewhat similar to the space from \cite[Theorem 0.1.2]{BS} denoted there by $Z$). The space $\wt{Z}$  plays a crucial role in \cite{DrGa2}, where it is used to prove a new result in the geometric theory of automorphic forms. It also allows to give a new proof of a very important theorem of T.~Braden, see \cite{DrGa1}. In each of the articles \cite{DrGa1,DrGa2} the space $\wt{Z}$ is used to construct the unit of a certain adjunction.

Let us note that most of the results of this article were announced in \cite{DrGa1}.

\medskip

Now let us explain more details.

\ssec{Attractors and repellers}  \label{sss:Z+}
Let $k$ be any field, and let $Z$ be a  an algebraic $k$-space of finite type  acted on by $\BG_m\,$.
According to Theorem~\ref{t:attractors} and the easy \propref{p:Z^0closed}, there exist algebraic spaces $Z^0$, $Z^+$, and $Z^-$  of finite type over $k$ representing the following functors:
$$\Maps(S,Z^0)=\Maps^{\BG_m}(S,Z),$$
$$\Maps(S,Z^+)=\Maps^{\BG_m}(\BA^1\times S,Z),$$
$$\Maps(S,Z^-)=\Maps^{\BG_m}(\BA^1_-\times S,Z),$$
where $S$ is a $k$-scheme, $\BA^1_-:=\BP^1-\{\infty\}$,  and the $\BG_m$-actions on $\BA^1$ and 
$\BA^1_-$ are the usual ones.\footnote{Using the map $t\mapsto t^{-1}$, one can identify $\BA^1_-$ with the scheme $\BA^1$ equipped with the $\BG_m$-action opposite to the usual one.}
The space $Z^0$ (resp.~$Z^+$ and $Z^-$) is called the \emph{space of $\BG_m$-fixed points}
(resp.~the \emph{attractor} and \emph{repeller}).


Let $p^+:Z^+\to Z$ and $q^+:Z^+\to Z^0$ denote the maps corresponding to evaluating a $\BG_m$-equivariant morphism $\BA^1\times S\to Z$ at $1\in \BA^1$  and $0\in \BA^1$, respectively. One defines
$p^-:Z^-\to Z$ and $q^-:Z^-\to Z^0$ similarly. Let $i^+:Z^0\to Z^+$ (resp. $i^-:Z^0\to Z^-$) denote the morphism induced by the projection $\BA^1\times S\to S$ (resp.~$\BA^1_-\times S\to S$).


\medskip

According to \propref{p: p^+}, the morphisms $p^{\pm}:Z^{\pm}\to Z$ are always unramified\footnote{The definition of ``unramified" is recalled in Subsect.~\ref{sss:unramified} below.}, and if $Z$ is separated they are monomorphisms (because $\BG_m\subset\BA^1$ is schematically dense). Of course, if $Z$ is affine then so are $Z^0$ and $Z^\pm$; moreover, in this case the morphisms $p^{\pm}:Z^{\pm}\to Z$ are closed embeddings.

Let us also mention \propref{p:Cartesian}, which says that the morphism
\[
j:=(i^+,i^-):Z^0\to Z^+\underset{Z}\times Z^-
\]
is an open embedding (and also a closed one). This fact is used in \cite{DrGa1} to construct the co-unit of the adjunction  in Braden's theorem.

Of course, in the case where $Z$ is a scheme equipped with a locally linear $\BG_m$-action all above-mentioned results are well known (in a slightly different language).

\ssec{The space $\wt{Z}$}  

\sssec{Hyperbolas}    \label{sss:hyperbolas}

We now consider the following family of curves over $\BA^1$, denoted by $\BX$: as a scheme, 
$\BX=\BA^2=\Spec k[\tau_1,\tau_2]$, and the map $\BX\to \BA^1$ is $(\tau_1,\tau_2)\mapsto \tau_1\tau_2\,$.The fibers of this map are hyperbolas; the zero fiber is the coordinate cross, i.e., a degenerate hyperbola.

\medskip

We let $\BG_m$ act on $\BX$ hyperbolically: 
\[
\lambda\cdot (\tau_1,\tau_2):=(\lambda\cdot \tau_1,\; \lambda{}^{-1}\cdot \tau_2).
\]

\sssec{The space $\wt{Z}$}  \label{sss:tilde Z}

According to Theorem~\ref{t:tildeZ}, there exists an algebraic space  $\wt{Z}$ of finite type over $\BA^1$ representing the following functor on the category of schemes over $\BA^1$:
$$\Maps_{\BA^1}(S,\wt{Z})=\Maps^{\BG_m}(\BX\underset{\BA^1}\times S,Z).$$
If $Z$ is a scheme equipped with a locally linear $\BG_m$-action then the existence of $\wt{Z}$ (as a scheme) is easy to prove, see Subsect.~\ref{ss:locally linear}.

 In general, we prove representability of the above functor using M.~Artin's technique
(see Section~\ref{s:tildeZ}). \emph{It would be nice if somebody finds a simpler and more constructive proof of representability.}

\sssec{The canonical morphism $\wt{p}:\wt{Z}\to \BA^1\times Z\times Z$} 
Note that any section $\sigma :\BA^1\to\BX$ of the morphism $\BX\to\BA^1$ defines a map 
$$\sigma^*:\Maps^{\BG_m} (\BX\underset{\BA^1}\times S\, ,Z)\to\Maps (S,Z)$$
and therefore a morphism $\wt{Z}\to Z$. Let $\pi_1:\wt{Z}\to Z$ and $\pi_2:\wt{Z}\to Z$ denote the morphisms corresponding to the sections
\[
t\mapsto (1,t)\in \BX \quad \text{ and }\quad t\mapsto (t,1)\in \BX\, ,
\]
respectively. Now define 
\begin{equation} \label{e:from tilde}
\wt{p}:\wt{Z}\to \BA^1\times Z\times Z
\end{equation}
to be the morphism whose first component is the tautological projection $\wt{Z}\to \BA^1$, and the second and the third components are $\pi_1$ and $\pi_2$, respectively. 

The morphism $\wt{p}:\wt{Z}\to \BA^1\times Z\times Z$ is always unramified, and if $Z$ is separated then 
$\wt{p}$ is a monomorphism (see \propref{p:props tilde p}). Moreover, if $Z$ is affine then $\wt{p}$ is a closed embedding  (see \propref{p:2new tilde}), so $\wt{p}$ identifies $\wt{Z}$ with a closed subscheme of 
$\BA^1\times Z\times Z$.


\sssec{The fibers of the morphism $\wt{Z}\to \BA^1$} 

Let $\wt{Z}_t$ denote the preimage of $t\in \BA^1$ under the projection $\wt{Z}\to \BA^1$. Let $\wt{p}_t$ denote the corresponding map $\wt{Z}_t\to Z\times Z$.

\medskip

By definition, $(\wt{Z}_1,\wt{p}_1)$ identifies with $(Z,\Delta_Z)$. For any $t\in \BA^1-\{0\}$, the pair 
$(\wt{Z}_t,\wt{p}_t)$ is the graph of the action of $t\in \BG_m$ on $Z$. Moreover, the morphism $\wt{p}$ induces an isomorphism
\begin{equation}  \label{e:graph}
\BG_m\underset{\BA^1}\times\wt{Z}\iso\Gamma , \quad\quad \Gamma:=\{ (t,z_1, z_2)\,|\,t\cdot z_1=z_2 \}.
\end{equation}

The space $\wt{Z}_0$ identifies with $Z^+\underset{Z^0}\times Z^-$ so that the morphism
$\wt{p}_0:\wt{Z}_0\to Z\times Z$ identifies with the composition 
$$Z^+\underset{Z^0}\times Z^-\hookrightarrow Z^+\times Z^- \overset{p^+\times p^-}\longrightarrow Z\times Z.$$ The above-mentioned identification comes from the fact that the degenerate hyperbola $\BX_0$ is the union of 
the coordinate axes, one of which identifies with $\BA^1$ and the other one with $\BA^1_-\,$.

Thus $\wt{Z}$ provides an ``interpolation" between the spaces $\wt{Z}_1=Z$ and 
$\wt{Z}_0=Z^+\underset{Z^0}\times Z^-$.

\sssec{Smoothness}  \label{sss:smoothness}
If $Z$ is smooth then so is the morphism $\wt{Z}\to\BA^1$, see \propref{p:2smoothness}. 

If $Z$ is smooth and affine then by \propref{p:affine and smooth}, the morphism $\wt{p}:\wt{Z}\to \BA^1\times Z\times Z$ induces an isomorphism 
$$\wt{Z}\iso\overline{\Gamma},$$ where $\Gamma$ is as in formula~\eqref{e:graph} and $\overline{\Gamma}$ is the scheme-theoretic closure of $\Gamma$ in $\BA^1\times Z\times Z\,$.

\sssec{Remark} 
We prove that if the algebraic space  $Z$ is separated (resp. is a scheme) then
so are all the algebraic spaces $Z^0$, $Z^\pm$, and $\wt{Z}$ (see \propref{p:Z^0closed}, \corref{c:attractors}, and \propref{p:props tilde p}).


\ssec{Another proof of the representability theorems}
The most difficult results of this article are Theorem~\ref{t:attractors} (representability of $Z^+$ in the category of algebraic spaces) and Theorem~\ref{t:tildeZ} (representability of $\wt{Z}$). D.~Halpern-Leistner drew my attention to the fact that if the diagonal morphism $\Delta :Z\to Z\times Z$ is affine (e.g., if $Z$ is separated) then these results follow from Theorem 2.1 of \cite{HP}. To prove our Theorem~\ref{t:attractors}, one uses the fact that the stack $\BA^1/\BG_m$ is \emph{cohomologically proper} over $\Spec k$ in the sense of \cite{HP}; by \cite[Theorem 2.1]{HP}, this implies that the stack parametrizing morphisms $\BA^1/\BG_m\to Z/\BG_m$ is algebraic. Similarly, our Theorem~\ref{t:tildeZ} follows from the cohomological properness of the morphism $\BX/\BG_m\to\BA^1$ considered in Subsect.~\ref{sss:hyperbolas}.

Currently this approach requires the assumption that $Z$ has affine diagonal (because of a similar assumption in 
\cite[Theorem 3.4.2]{Lur2}, which is used in \cite{HP}). Hopefully this assumption will be removed in a forthcoming work by B.~Bhatt and D.~Halpern-Leistner.



\ssec{Organization of the paper}
In \secref{s:actions} we define and study the space of $\BG_m$-fixed points $Z^0$, the attractor $Z^+$, and the repeller $Z^-$ corresponding to an algebraic $k$-space $Z$ of finite type acted on by $\BG_m\,$.

In \secref{ss:deg} we define and study the space $\wt{Z}$. A more detailed description of Section~\ref{ss:deg} can be found at the beginning of the section. 

In \secref{s:openness} we prove some openness results. One of them is used in \cite{DrGa1}.

In Sections \ref{s:attractors}-\ref{s:tildeZ} we prove Theorems~\ref{t:attractors} and
Theorems~\ref{t:tildeZ} (the proofs are too long to be given in Sections \ref{s:actions} and \ref{ss:deg}).

In Appendix \ref{s:very general} we prove a very general \lemref{l:postponed}.

In Appendix \ref{s:Polish} we briefly recall some results on attractors due to 
A.~Bia{\l}ynicki-Birula, J.~Konarski, and A.~J.~Sommese.

In Appendix \ref{s: 2-categorical framework} we describe the categorical structure formed by the correspondences 
considered in Sections~\ref{s:actions} and \ref{ss:deg}. We also explain why this structure implies the main result of \cite{DrGa1}.

\ssec{Some conventions and recollections}  \label{ss:conventions}
\sssec{Maps and morphisms as synonyms} We often use the word ``map" as a synonym of ``morphism".
The space of morphisms between objects $X,Y$ of a category will usually be denoted by $\Maps (X ,Y)$.
\sssec{General notion of $k$-space}  \label{sss:spaces}
Once and for all, we fix a field $k$ (of any characteristic).
By a \emph{$k$-space} (or simply  \emph{space}) we mean a contravariant functor $F$ from the category of 
$k$-schemes to that of sets which is a sheaf for the fpqc topology. Instead of considering all $k$-schemes as ``test schemes", it suffices to consider affine ones (any fpqc sheaf on the category of affine 
$k$-schemes uniquely extends to an fpqc sheaf on the category of all $k$-schemes). Instead of $F(\Spec R)$ we write simply $F(R)$; in other words, we consider $F$ as a covariant functor on the category of $k$-algebras.

Note that for any $k$-scheme $S$ we have $F(S)=\Maps (S,F)$, where $\Maps$ stands for the set of morphisms between $k$-spaces. Usually we prefer to write $\Maps (S,F)$ rather than $F(S)$.

\sssec{Algebraic $k$-spaces}   \label{sss:alg spaces}
General $k$-spaces will appear only as ``intermediate" objects.
For us, the really geometric objects are \emph{algebraic spaces}. We will be using the definition of algebraic space from \cite{LM} (which goes back to  M.~Artin).
\footnote{In particular, quasi-separatedness is included into the definition of algebraic space. Thus
the quotient $\BA^1/\BZ$ (where the discrete group $\BZ$ acts by translations)
is \emph{not} an algebraic space.}

Any quasi-separated $k$-scheme (in particular, any $k$-scheme of finite type) is an algebraic space.
The reader may prefer to restrict his attention to schemes.

\sssec{Monomorphisms}
A morphism of $k$-spaces $f:X_1\to X_2$ is said to be a \emph{monomorphism} if 
the corresponding map $$\Maps (S,X_1)\to\Maps (S,X_2)$$
is injective for any $k$-scheme $S$. In particular, this applies if $X_1$ and $X_2$ are algebraic spaces
(e.g., schemes). It is known that a morphism \emph{of finite type} between schemes (or algebraic spaces) is a monomorphism if and only if each of its geometric fibers is a reduced scheme with at most one point.
It follows that a finite monomorphism is a closed embedding.

\sssec{Unramified morphisms}  \label{sss:unramified}
According to Definition 17.3.1 from EGA IV-4, a morphism of schemes is said to be  \emph{unramified} if it is formally unramified and locally of finite presentation. The definition in~\cite{St}  is slightly different: ``locally of finite presentation" is replaced by ``locally of finite type" . The difference is irrelevant for us because we will be dealing between morphisms between Noetherian schemes (or algebraic spaces).

Recall that a morphism $f$ is formally unramified if and only if the corresponding sheaf of relative differentials is zero. If $f$ has finite type this is equivalent to the geometric fibers of $f$ being finite and reduced.

\ssec{Acknowledgements}

I thank A.~Beilinson, J.~Konarski, A.~J.~Sommese,  and especially B.~Conrad  and N.~Rozenblyum for helpful discussions and suggestions. 

I am especially grateful to D.~Gaitsgory. In fact, this article appeared as a part of a project joint with him (see \cite{DrGa1,DrGa2}). Moreover, a part of the work on this article was done jointly with him (e.g., the formulation of Propositions~\ref{p:Cartesian} and \ref{p:2open embeddings} is due to D.~Gaitsgory). 

\medskip 

The author's research was partially supported by NSF grants DMS-1001660 and DMS-1303100.

\section{Fixed points, attractors, and repellers}  \label{s:actions}
The main results of this section are \propref{p:Z^0closed}, Theorem~\ref{t:attractors},
\propref{p: p^+}, and \propref{p:Cartesian} (the latter is used in \cite{DrGa1} to construct the co-unit of the adjunction  in Braden's theorem). In the case of a scheme equipped with a locally linear $\BG_m$-action these results are well known (in a slightly different language).

\medskip

We will be using the conventions of Subsect.~\ref{ss:conventions} and especially those regarding the notions of $k$-space and algbraic $k$-space (see \ref{sss:spaces}-\ref{sss:alg spaces}).

\ssec{The space of $\BG_m$-equivariant maps} \label{ss:GMaps}
Let $Y,Z$ be $k$-spaces equipped with an action of $\BG_m$. Then we define a $k$-space
$\GMaps(Y,Z)$ as follows: for any $k$-scheme $S$,
\begin{equation}
\Maps (S,\GMaps(Y,Z)):=\gMaps (Y\times S,Z)
\end{equation}
(the r.h.s. is clearly an fpqc sheaf with respect to $S$). The action of $\BG_m$ on $Z$ induces a $\BG_m$-action on $\GMaps(Y,Z)$.

Note that even if $Y$ and $Z$ are schemes, the space $\GMaps(Y,Z)$ does not have to be a scheme (or an algebraic space), in general. 

\ssec{The space of fixed points} \label{ss:fixed_points}
Let $Z$ be a $k$-space equipped with an action of $\BG_m$. Then we set
\begin{equation}
Z^0:=\GMaps(\Spec k,Z).
\end{equation}

Note that $Z^0$ is a subspace of $Z$ because $\Maps (S,Z^0)=\gMaps (S,Z)$ is a subset of
$\Maps (S,Z)$.

\begin{defn}
$Z^0$ is called \emph{the subspace of fixed points} of $Z$.
\end{defn}



\begin{prop}   \label{p:Z^0closed}
If $Z$ is an algebraic $k$-space (resp. scheme) of finite type then so is $Z^0$. Moreover, the morphism
$Z^0\to Z$ is a closed embedding.
\end{prop}

This proposition is easy. 
The only surprise is that $Z^0\subset Z$ is closed even if $Z$ is not separated. Idea of the proof: since $\BG_m$ is connected $Z^0=Z_0\,$, where  $Z_0$ is the space of fixed points of the \emph{formal} multiplicative group acting on $Z$; on the other hand, $Z_0$ is a closed subspace of $Z$ (e.g., in the characteristic zero case $Z_0$ is just the space of zeros of the vector field on $Z$ corresponding to the 
$\BG_m$-action). The detailed proof is below.

\begin{proof}
It suffices to show that the morphism $Z^0\to Z$ is a closed embedding.

Let $\CG$ be the space of stabilizers, i.e., an $S$-point of $\CG$  is a pair $(z,g)$, where $z\in Z(S)$ and 
$g\in\BG_m (S)$ stabilizes $z$. We have a monomorphism of group schemes\footnote{This is a slight abuse of language: if $Z$ itself is not a scheme then $\CG$ is a scheme only in the \emph{relative} sense (i.e., 
$\CG\times_ZS$ is a scheme for any scheme $S$ over $Z$).} over $Z$
\begin{equation}   \label{e:groupmonomorphism}
\varphi :\CG\to\BG_m\times Z\, .
\end{equation}

Let $(\BG_m)_n\subset\BG_m$ denote the $n$-th infinitesimal neighborhood of $1\in\BG_m\,$, so
$(\BG_m)_n=\Spec k[\lambda ]/(\lambda -1)^{n+1}$ and
\[
(\BG_m)_n\times Z=\Spec\CA_n\, , \quad\quad \CA_n:=\CO_Z [\lambda ]/(\lambda -1)^n.
\]
Base-changing the monomorphism \eqref{e:groupmonomorphism} with respect to 
$(\BG_m)_n\times Z\mono \BG_m\times Z$, one gets a morphism
 $$\varphi_n :\CG_n\to(\BG_m)_n\times Z=\Spec\CA_n\, ,$$
 where $\CG_n\subset\CG$ is the $n$-th infinitesimal neighborhood in $\CG$ of the unit section $Z\mono\CG$.
 The morphism $\varphi_n$ is a closed embedding (because it is a finite monomorphism). Let 
 $\CI_n\subset \CA_n$ be the corresponding sheaf of ideals. Note that the image of $\CI_n$ in 
$ \CA_{n-1}$ equals $\CI_{n-1}\,$. 
%

Now set
\[
\CJ_n;=\im (\CI_n\otimes\CHom_{\CO_Z}(\CA_n,\CO_Z)\to \CO_Z).
\]
Each $\CJ_n$ is an ideal in $\CO_Z\,$, and $\CJ_n\subset\CJ_{n+1}\,$. Let $Z_0\subset Z$ be the closed subspace corresponding to the union of the ideals $\CJ_n\,$.

Let us prove that $Z^0=Z_0\,$. It is clear that $Z^0\subset Z_0\,$. It remains to show that the morphism
\begin{equation} \label{e:phi_0}
\CG\times_ZZ_0\to\BG_m\times Z_0
\end{equation}
 induced by the homomorphism \eqref{e:groupmonomorphism} is an isomorphism.  Clearly it is a monomorphism (because $\varphi$ is). Let $U\subset\CG\times_ZZ_0$ denote the locus where the morphism \eqref{e:phi_0} is etale. Then $U$ is open in both $\CG\times_ZZ_0$ and $\BG_m\times Z_0\,$. Note that $U$ contains the unit section (by the definition of $Z_0$). Moreover, it is easy to see that $U^{-1}=U$ and that $U$ is stable under multiplication by any local section of $U$.  So $U$ is an open subgroup scheme of $\BG_m\times Z_0\,$. Since $\BG_m$ is connected, we see that the morphism
 $U\to\BG_m\times Z_0$ is an isomorphism. Since the map \eqref{e:phi_0} is a monomorphism this means that
$U= \CG\times_ZZ_0$ and the map \eqref{e:phi_0} is an isomorphism. 
\end{proof}

\begin{example}  \label{ex:fixed-affine}
Suppose that $Z$ is an affine scheme $\Spec A$. A $\BG_m$-action on $Z$ is the same as a 
$\BZ$-grading on $A$ (namely, the $n$-th component of $A$ consists of functions 
$f\in H^0(Z,\CO_Z)$ such that $f(\lambda\cdot z)=\lambda^n\cdot f(z)\,$). It is easy to see that $Z^0=\Spec A^0$, where $A^0$ is the maximal graded quotient algebra of $A$ concentrated in degree 0 (in other words, $A^0$ is the quotient of $A$ by the ideal generated by homogeneous elements of non-zero degree).
\end{example}

\begin{lem}   \label{l:TZ0}
For any $z\in Z^0$ the tangent space\footnote{We define the tangent space by $T_zZ:=(T_z^*Z)^*$, where $T_z^*Z$ is the fiber of $\Omega^1_{Z/k}$ at $z$. (The equality $T_z^*Z=m_z/m_z^2$ holds \emph{if the residue field of $z$ is finite and separable} over $k$.)} $T_zZ^0\subset T_zZ$ equals $(T_zZ)^{\BG_m}$.
\end{lem}

\begin{proof}
We can assume that the residue field of $z$ equals $k$ (otherwise do base change). Then compute $T_zZ^0$ in terms of morphisms 
$\Spec k[\varepsilon ]/(\varepsilon^2 )\to Z^0$.
\end{proof}

\ssec{Attractors} \label{ss:attr}
\sssec{The definition} 
Let $Z$ be a $k$-space equipped with an action of $\BG_m$. Then we set
\begin{equation}  \label{e:attr}
Z^+:=\GMaps(\BA^1,Z),
\end{equation}
where $\BG_m$ acts on $\BA^1$ by dilations.

\begin{defn}
$Z^+$ is called the \emph{attractor} of $Z$.
\end{defn}

Later we will prove (see Theorem~\ref{t:attractors}) that if $Z$ is an algebraic space of finite type then so is 
$Z^+$.

\sssec{Structures on $Z^+$} \label{sss:structures}

(i) $\BA^1$ is a monoid with respect to multiplication. The action of $\BA^1$ on itself induces an $\BA^1$-action on $Z^+$, which extends the $\BG_m$-action defined in \secref{ss:GMaps}.

\medskip

(ii) Restricting a morphism $\BA^1\times S\to Z$ to $\{1\}\times S$ one gets a morphism $S\to Z$. Thus we get a 
$\BG_m$-equivariant morphism $p^+:Z^+\to Z$. 

Note that if $Z$ is \emph{separated} then $p^+:Z^+\to Z$ is a \emph{mono}morphism. To see this, it suffices to interpret
$p^+$ as the composition
\[
\GMaps (\BA^1,Z)\to \GMaps (\BG_m,Z)=Z.
\]

Thus if $Z$ is separated then $p^+$ identifies $Z^+(k)$ with a subset of $Z(k)$. It consists of those points $z\in Z(k)$ for which the map $\BG_m\to Z$ defined by $t\mapsto t\cdot z$ extends to
a map $f:\BA^1\to Z$; informally, the limit
\begin{equation}   \label{e:limit}
\underset{t\to 0}{lim}\,\, t\cdot z
\end{equation}
should exist.

\medskip

(iii) Recall that $Z^0=\GMaps (\Spec k, Z)$. We equip $\Spec k$ and $Z^0$ with the trivial action of the multiplicative monoid $\BA^1$.

The $\BA^1$-equivariant maps $0:\Spec k\to\BA^1$ and $\BA^1\to \Spec k$ induce $\BA^1$-equivariant maps
$q^+:Z^+\to Z^0$ and $i^+:Z^0\to Z^+$ such that $q^+\circ i^+=\id_{Z^0}$ and the composition $p^+\circ i^+$ is equal to the canonical embedding $Z^0\mono Z$.

Note that if $Z$ is separated then for $z\in Z^+(k)\subset Z(k)$ the point $q^+(z)$ is the limit \eqref{e:limit}.

\sssec{Affine case}   \label{sss:attractors-affine}
Suppose that $Z$ is affine, i.e., $Z=\Spec A$, where $A$ is a $\BZ$-graded commutative algebra.
It is easy to see that in this case $Z^+$ is represented by the affine scheme $\Spec A^+$, where $A^+$ is the 
maximal $\BZ_+$-graded quotient algebra of $A$ (in other words, the quotient of $A$ by the ideal generated by by all homogeneous elements of $A$ of strictly negative degrees). 

By Example~\ref{ex:fixed-affine}, $Z^0=\Spec A^0$, where $A^0$ is the maximal graded quotient algebra of 
$A$ (or equivalently, of $A^+$) concentrated in degree 0. Since $A^+$ is $\BZ_+$-graded, the algebra $A^0$ identifies with the $0$-th graded component of $A^+$. Thus we get homomorphisms $A^0\mono A^+\epi A^0$. They correspond to the morphisms $$Z^0\overset{\;\;q^+}\longleftarrow Z^+\overset{\;\;i^+}\longleftarrow Z^0.$$

\ssec{Results on attractors}   \label{ss:Results_attractors}

\sssec{Representability of $Z^+$}

\begin{thm}   \label{t:attractors}
Let $Z$ be an algebraic $k$-space of finite type equipped with a $\BG_m$-action. Then

(i) $Z^+$ is an algebraic $k$-space of finite type;

(ii) the morphism $q^+:Z^+\to Z^0$ is affine.
\end{thm}

The proof is given in Section~\ref{s:attractors}. It yields a rather explicit description of the pair
 $(Z^+, q^+)$ in terms of the formal neighborhood of $Z^0\subset Z$, see Theorem~\ref{t:2attractors}(ii) and 
 Subsect.~\ref{ss:A+}.
 
Note that if $Z$ is affine Theorem~\ref{t:attractors} is clear from Subsect.~\ref{sss:attractors-affine}, and this immediately implies the theorem 
in the case of schemes equipped with a \emph{locally linear} $\BG_m$-action, 
see Subsect,~\ref{sss: locally linear} below. This case is enough for most practical purposes, see Remark~\ref{r:locally linear} below.

\
\begin{cor}  \label{c:attractors}
(i) If $Z$ is a separated algebraic $k$-space of finite type then so is $Z^+$.

(ii) If $Z$ is a $k$-scheme of finite type then so is $Z^+$.
\end{cor}

\begin{proof}
Follows from  Theorem~\ref{t:attractors}(ii) because by \propref{p:Z^0closed}, $Z^0$ is a closed subspace of~$Z$.
\end{proof}

\sssec{Proof of Theorem~\ref{t:attractors} for schemes with a locally linear $\BG_m$-action}
\label{sss: locally linear}

\begin{defn} \label{d:locally linear} 
An action of $\BG_m$ on a $k$-scheme $Z$ is said to be \emph{locally linear} if
$Z$ can be covered by open affine subschemes preserved by the $\BG_m$-action.  
\end{defn} 

\begin{rem}   \label{r:locally linear} 
If $k$ is \emph{algebraically closed} and $Z_{\red}$ is a \emph{normal} separated\footnote{We do not know if separateness is really necessary in Sumihiro's theorem.} scheme of finite type over $k$ then by a theorem of H.~Sumihiro, \emph{any action of $\BG_m$ on  $Z$ is locally linear}. (The proof of this theorem is contained in \cite{Sum} and also in  \cite[p.20-23]{KKMS} and \cite{KKLV}.)
\end{rem}

For schemes equipped with a locally linear $\BG_m$-action Theorem~\ref{t:attractors} is very easy: it follows from the affine case (which is clear from Subsect.~\ref{sss:attractors-affine}) and the following lemma.

\begin{lem}   \label{l:U^+}
Let $Z$ be a $k$-space equipped with a $\BG_m$-action. Let $U\subset Z$ be a $\BG_m$-stable open subspace. Then the subspace $U^+\subset Z^+$ equals $(q^+)^{-1}(U^0)$, where $q^+$ is the natural morphism $Z^+\to Z^0$.
\end{lem}

\begin{proof}
For any test scheme $S$, we have to show that if
$f:\BA^1\times S\to Z$
is a $\BG_m$-equivariant morphism such that $\{0\}\times S\subset f^{-1}(U)$ then 
$f^{-1}(U)=\BA^1\times S$. This is clear because $f^{-1}(U)\subset\BA^1\times S$ is open and $\BG_m$-stable.
\end{proof}

\sssec{The attractor of a closed subspace}

\begin{lem}   \label{l:F+}
Let $Z$ be a $k$-space equipped with a $\BG_m$-action. Let $F\subset Z$ be a $\BG_m$-stable closed subspace. Then the subspace $F^+\subset Z^+$  equals $(p^+)^{-1}(F)$, where $p^+$ is the natural morphism $Z^+\to Z$.
\end{lem}

\begin{proof}
An $S$-point of $(p^+)^{-1}(F)$ is a  $\BG_m$-equivariant morphism $f:\BA^1\times S\to Z$ such that
$\BG_m\times S\subset f^{-1}(F)$. Since $ f^{-1}(F)$ is closed in $\BA^1\times S$ this implies that
$f^{-1}(F)=\BA^1\times S$, i.e., $f(\BA^1\times S)\subset F$.
\end{proof}

\sssec{The morphism $p^+:Z^+\to Z$}

\begin{prop} \label{p: p^+}
Let $Z$ be an algebraic $k$-space of finite type equipped with a $\BG_m$-action. 

(i) The morphism $p^+:Z^+\to Z$ is unramified (i.e., its geometric fibers are finite and reduced); 

(ii) If $Z$ is separated then $p^+:Z^+\to Z$ is a monomorphism (i.e., each geometric fiber of $p^+:Z^+\to Z$ is reduced and has at most one point);

(iii) If $Z$ is proper then each geometric fiber of $p^+:Z^+\to Z$ is reduced and has exactly one point.

(iv) If $Z$ is an affine scheme then $p^+:Z^+\to Z$ is a closed embedding. 

(v) The fiber of $p^+:Z^+\to Z$ over any geometric point of $Z^0\subset Z$ has a single point (even if $Z$ is not separated).

(vi) Let $z\in Z^0$ and $\zeta:=i^+(z)\in Z^+$, so $p^+(\zeta )=z$. Then the map of tangent spaces 
$T_{\zeta}Z^+\to T_zZ$ corresponding to $p^+:Z^+\to Z$ induces an isomorphism
$T_{\zeta}Z^+\iso (T_zZ)^+$, where $(T_zZ)^+\subset T_zZ$ is the non-negative part with respect to the 
$\BG_m$-action on $T_zZ$. Moreover, the tangent maps $T_zZ^0\to T_{\zeta}Z^+\to T_zZ^0$ corresponding to the morphisms $i^+:Z^0\to Z^+$ and $q^+:Z^+\to Z^0$ identify with the canonical maps 
$(T_zZ)^0\mono (T_zZ)^+\epi (T_zZ)^0$.
\end{prop}

\begin{proof}
Statement (iv) is clear from Subsect.~\ref{sss:attractors-affine}.
Statement (ii) was proved in Subsect.~\ref{sss:structures}(ii). 
Statement  (iii) follows from (ii) and the fact that any morphism from $\BA^1-\{ 0\}$ to a proper scheme extends to the whole $\BA^1$.

Let us prove (i). We can assume that $k$ is algebraically closed. Then we have to check that for any 
$\zeta\in Z^+(k)$ the map of tangent spaces
\begin{equation}  \label{e:differential}
 T_{\zeta}Z^+\to T_{p^+({\zeta})}Z
\end{equation}
induced by $p^+:Z^+\to Z$ is injective. Let $f:\BA^1\to Z$ be the $\BG_m$-equivariant morphism corresponding to $\zeta$. Then 
\begin{equation}  \label{e:TZ+l}
T_{\zeta}Z^+=\Hom_{\BG_m}(f^*\Omega^1_Z,\CO_{\BA^1}),
\end{equation}
and the map \eqref{e:differential} assigns to a $\BG_m$-equivariant morphism $\varphi :f^*\Omega^1_Z\to\CO_{\BA^1}$ the corresponding map between fibers at $1\in\BA^1$. So the kernel of \eqref{e:differential} consists of those
$\varphi\in\Hom_{\BG_m}(f^*\Omega^1_Z,\CO_{\BA^1})$ for which $\varphi |_{\BA^1-\{ 0\}}=0$. This implies that $\varphi =0$ because $\CO_{\BA^1}$ has no nozero sections supported at $0\in\BA^1$.

Let us deduce (vi) from formula \eqref{e:TZ+l}. Since $\zeta:=i^+(z)$ the morphism $f:\BA^1\to Z$
corresponding to $\zeta$ is constant, so $f^*\Omega^1_Z=T_z^*Z\otimes\CO_{\BA^1}$. Thus formula
\eqref{e:TZ+l} identifies $T_{\zeta}Z^+$ with the space
\[
\Hom_{\BG_m}(T_z^*Z,k[t])=\Hom_{\BG_m}((T_z^*Z)^+,k)=(T_zZ)^+.
\]

%

To prove the lemma, note that by $\BG_m$-equivariance, $f(t)=f(1)$ for $t\ne 0$. Now restricting $f$ to the
Henselization of $\BA^1$ at 0 we see that $f$ is constant.

Let us prove (v). After base change, we can assume that $k$ is algebraically closed and the point in question is a $k$-point $z_0\,$. Any $k$-point of $Z$ is closed (because $Z$ is an algebraic $k$-space of finite type). So we can apply \lemref{l:F+} to the $\BG_m$-stable closed subspace $F=\{z_0\}$ and get 
$p^{-1}(z_0)=F^+\simeq\Spec k\,$.
\end{proof}

\begin{example}  \label{ex:P^1}
Let  $Z$ be the projective line $\BP^1$ equipped with the usual action of $\BG_m\,$. Then $p^+:Z^+\to Z$ is the
canonical morphism $\BA^1\sqcup\{\infty\}\to\BP^1$. In particular, $p^+$ is  \emph{not a locally closed embedding.}
\end{example}

\begin{rem}
In the above example the restriction of $p^+:Z^+\to Z$ to each connected component of $Z^+$   \emph{is} a locally closed embedding. This turns out to be true in a surprisingly large class of situations, but there are also important examples when this is false. More details can be found in 
Appendix~\ref{s:Polish}.
\end{rem}

\begin{rem}   \label{r:prigoditsya}
It is easy to deduce from \propref{p: p^+}(i) that if the diagonal map $Z\to Z\times Z$ is a locally closed embedding (e.g., if $Z$ is a scheme) then the map 
$$Z^+\overset{(p^+,q^+)}\longrightarrow  Z\times Z^0$$
 is a monomorphism.
\end{rem}

\begin{prop} \label{p:contracting}
Let $Z$ be an algebraic $k$-space of finite type equipped with a $\BG_m$-action. 
The morphism $p^+:Z^+\to Z$ is an isomorphism if and only if the $\BG_m$-action on $Z$  can be extended to an $\BA^1$-action. In this case such extension is unique.
\end{prop}

\begin{rem}
It is somewhat surprising that in this proposition $Z$ does not have to be separated.
\end{rem}

The ``only if" part of \propref{p:contracting} follows from the fact that the $\BG_m$-action on $Z^+$ always extends to an $\BA^1$-action, see Subsect.~\ref{sss:structures}(i). The remaining parts of  
\propref{p:contracting} immediately follow from the next lemma.

\begin{lem}  \label{l:contracting}
Let $Z$ be an algebraic $k$-space of finite type equipped with a $\BA^1$-action. Equip $Z^+$ with the 
$\BA^1$-action from Subsect.~\ref{sss:structures}(i). Then $p^+:Z^+\to Z$ is an $\BA^1$-equivariant isomorphism.
\end{lem}

To prove the lemma, we will need the following

\begin{rem}     \label{r:section of unramified}
Let $X\overset{i}\longrightarrow Y\overset{\pi}\longrightarrow X$ be morphisms of algebraic spaces such that
$\pi\circ i=\id_X$ and $\pi :Y\to X$ is unramified at each point of $i(X)$, then $i:X\to Y$ is an open embedding. Indeed, $i$ is clearly a monomorphism, and it is also etale: to see this, look at the homomorphisms of Henselizations (or of completed local rings) induced by $i:X\to Y$ and $\pi :Y\to X$.
\end{rem}

\begin{proof}[Proof of Lemma~\ref{l:contracting}]
The $\BA^1$-action on $Z$ defines an $\BA^1$-equivariant morphism $g:Z\to Z^+$ such that the composition of the maps
\[ 
Z\overset{g}\longrightarrow Z^+\overset{p^+}\longrightarrow Z
\] 
equals $\id_Z$. It remains to show that $g$ is an isomorphism.

By \propref{p: p^+}(i), the morphism $p^+:Z^+\to Z$ is unramified. So $g$ is an open embedding by 
Remark~\ref{r:section of unramified},
 It remains to show that any point $\zeta\in Z^+$ is contained in $g(Z)$. Without loss of generality, we can assume that $\zeta$ is a $k$-point (otherwise do base change). Set
\[
U_{\zeta}:=\{t\in\BA^1\,|t\cdot\zeta\in g(Z)\},
\]
where $t\cdot\zeta$ denotes the action of $\BA^1$ on $Z^+$ from Subsect.~\ref{sss:structures}(i). We have to show that $1\in U_{\zeta}$. Since $U_{\zeta}$ is an open $\BG_m$-stable subset of $\BA^1$ it suffices to show that $0\in U_{\zeta}$.  
We claim that
\begin{equation} \label{e:0zeta}
0\cdot\zeta =g(q^+(\zeta )),
\end{equation}
where $q^+:Z^+\to Z^0$ is the canonical morphism. Indeed, it is easy to check that
\begin{equation} \label{e:applying p+}
p^+(0\cdot\zeta)=q^+(\zeta ), \quad p^+(g(q^+(\zeta )))=q^+(\zeta ).
\end{equation}
Since $q^+(\zeta )\in Z^0(k)$ the equality \eqref{e:0zeta} follows from \eqref{e:applying p+} and 
\propref{p: p^+}(v).
\end{proof}

\sssec{Smoothness} The following proposition is well known (at least, if $Z$ is a scheme).

\begin{prop}   \label{p:smoothness}
Suppose that an algebraic $k$-space $Z$ is smooth. Then $Z^0$ and $Z^+$  are  smooth. Moreover, the  morphism 
$q^+:Z^+\to Z^0$ is smooth.
\end{prop}

\begin{proof}
We will only prove that $q^+$ is smooth. (Smoothness of $Z^0$ can be proved similarly, and smoothness of $Z^+$ follows.)

It suffices to check that $q^+$ is formally smooth. 
Let $R$ be a $k$-algebra and $\bar R=R/I$, where 
$I\subset R$ is an ideal with $I^2=0$. Let $\bar f :\BA^1_{\bar R}\to Z$ be a $\BG_m$-equivariant morphism and let  $\bar f_0:\Spec\bar R\to Z^0$ denote the restriction of $\bar f$ to $0\subset\BA^1_{\bar R}\,$. Let $\varphi :\Spec R\to Z^0$ be any morphism extending $\bar f_0\,$. We have to extend 
$\bar f$  to a $\BG_m$-equivariant morphism $f :\BA^1_{R}\to Z$ so that $f_0=\varphi$.

Using smoothness of $Z$, it is easy to show that there is a not-necessarily equivariant morphism 
$f :\BA^1_{R}\to Z$ extending $\bar f$ with $f_0=\varphi$. Then standard arguments show that the obstruction to existence of a $\BG_m$-equivariant $f$ with the required properties belongs to
$$H^1(\BG_m\, ,M), \quad 
M:=H^0(\BA^1_{\bar R}\, ,\bar f^*\Theta_Z\otimes\CJ )\otimes_{\bar R}I,$$
where $\Theta_Z$ is the tangent bundle of $Z$ and $\CJ\subset\CO_{\BA^1_{\bar R}}$ is the ideal of the zero section. But $H^1$ of $\BG_m$ with coefficients in any $\BG_m$-module is zero.
\end{proof}

\ssec{Repellers}    \label{ss:repeller}
Set $\BA^1_-:=\BP^1-\{\infty\}$; this is a monoid with respect to multiplication containing $\BG_m$ as a subgroup. One has an isomorphism of monoids
\begin{equation}   \label{e:inversion}
\BA^1\iso \BA^1_-\, , \quad\quad t\mapsto t^{-1}.
\end{equation}

Given a $k$-space equipped with a $\BG_m$-action we set
\begin{equation}   \label{e:repel}
Z^-:=\GMaps(\BA^1_-,Z).
\end{equation}

\begin{defn}
$Z^-$ is called the \emph{repeller} of $Z$.
\end{defn}

Just as in Subsect.~\ref{sss:structures} one defines an 
action of the monoid $\BA^1_-$ on $Z^-$ extending the action of $\BG_m\,$, a
$\BG_m$-equivariant morphism $p^-:Z^-\to Z$, and  $\BA^1_-$-equivariant morphisms  
$q^-:Z^-\to Z^0$ and $i^-:Z^0\to Z^-$ (where $Z^0$ is equipped with the trivial $\BA^1_-$-action).
One has $q^-\circ i^-=\id_{Z^0}\,$, and the composition $p^-\circ i^-$ is equal to the canonical embedding $Z^0\mono Z$.

Using the isomorphism \eqref{e:inversion}, one can identify $Z^-$ with the attractor for the inverse 
action of $\BG_m$ on $Z$ (this identification is $\BG_m$-\emph{anti}-equivariant). Thus the results on attractors from Subsections~\ref{sss:attractors-affine} and \ref{ss:Results_attractors} imply similar results for repellers.

In particular, if $Z$ is the spectrum of a $\BZ$-graded algebra $A$ then $Z^-$ canonically identifies with 
$\Spec A^-$,  where $A^-$ is the maximal $\BZ_-$-graded quotient algebra of $A$.

\ssec{Attractors and repellers}
In this subsection $Z$ denotes an algebraic $k$-space of finite type equipped with a $\BG_m$-action.
\begin{lem}  \label{l:closed} 
The morphisms $i^{\pm}:Z^0\to Z^{\pm}$ are closed embeddings.
\end{lem}

\begin{proof}  
It suffices to consider $i^+$. 

By Theorem~\ref{t:attractors}(ii), the morphism $q^+:Z^+\to Z^0$ is separated. One has $q^+\circ i^+=\id_{Z^0}\,$. So $i^+$ is a closed embedding.
\end{proof}

Now consider the fiber product $Z^+\underset{Z}\times Z^-$ formed using the maps
$p^{\pm}:Z^{\pm}\to Z$.

\begin{prop}  \label{p:Cartesian} 
The map
\begin{equation}    \label{e:open-closed}
j:=(i^+,i^-):Z^0\to Z^+\underset{Z}\times Z^-
\end{equation}
is both an open embedding and a closed one.
\end{prop}

\begin{rem}    \label{r:affine case}
If $Z$ is affine then $j$ is an isomorphism (this immediately follows from the explicit description of $Z^{\pm}$ in the affine case, see Subsections~\ref{sss:attractors-affine} and \ref{ss:repeller}). In general, $j$ is not necessarily an isomorphism. To see this, note that by \eqref{e:attr} and \eqref{e:repel}, we have
\begin{equation}   \label{e:fiberprod}
Z^+\underset{Z}\times Z^-=\GMaps(\BP^1,Z)
\end{equation}
(where $\BP^!$ is equipped with the usual $\BG_m$-action), and a $\BG_m$-equivariant map $\BP^1\to Z$ does not have to be constant, in general.
\end{rem}

\begin{proof}[Proof of \propref{p:Cartesian}]
Writing $j$ as 
\[
Z^0=Z^0\underset{Z}\times Z^0\,\overset{(i^+,i^-)}\longrightarrow\, Z^+\underset{Z}\times Z^-
\]
and using \lemref{l:closed} we see that $j$ is a closed embedding.

Let us prove that $j$ is an open embedding. Let $\pi :Z^+\underset{Z}\times Z^-\to Z^0$ denote the composition 
$$
Z^+\underset{Z}\times Z^-\to Z^+\overset{q^+}{\longrightarrow} Z^0.
$$
Then $\pi\circ j=\id_{Z^0}\,$. Now by Remark~\ref{r:section of unramified}, it suffices to check that the tangent map 
\begin{equation}   \label{e:tange}
T_{j(z)}(Z^+\underset{Z}\times Z^-)\to T_zZ^0
\end{equation}
corresponding to $\pi$ is an isomorphism. Indeed, by \propref{p: p^+}(vi) and a similar statement for $Z^-$,
the map \eqref{e:tange} is just the identity map $(T_zZ)^+\cap (T_zZ)^-\to (T_zZ)^{\BG_m}\,$.
\end{proof}

\begin{rem}
The fact that $j$ is an open embedding can also be proved using \eqref{e:fiberprod} and the fact that 
every regular function on $\BP^1$ is constant. (This type of argument is used in the proof of \propref{p:open embeddings} below.)
\end{rem}

\begin{cor} \label{c:contractive}
(i) If the map $p^+:Z^+\to Z$ is etale then the maps 
$Z^0\overset{i^-}\longrightarrow Z^-\overset{q^-}\longrightarrow Z^0$ are isomorphisms.

(ii) If the map $p^-:Z^-\to Z$ is etale then the maps 
$Z^0\overset{i^+}\longrightarrow Z^+\overset{q^+}\longrightarrow Z^0$  are isomorphisms.
\end{cor}

Let us note that the statements converse to (i) and (ii) are also true (and well known).

\begin{proof}
Let us prove (ii). Since $p^-:Z^-\to Z$ is etale so is the projection $Z^+\underset{Z}\times Z^-\to Z^+$. So \propref{p:Cartesian} implies that the morphism $i^+:Z^0\to Z^+$ is etale. Since $i^+$ is also a monomorphism it is an open embedding. It remains to show that any point $\zeta\in Z^+$ is contained in $i^+(Z^0)$. Set 
\[
U_{\zeta}:=\{t\in\BA^1\,|\, t\cdot\zeta\in i^+(Z^0)\}.
\]
We have to show that $1\in U_{\zeta}$. But $U_{\zeta}$ is an open $\BG_m$-stable subset of $\BA^1$ containing $0$, so $U_{\zeta}=\BA^1$.
\end{proof}

\section{The space $\wt{Z}$}  \label{ss:deg}
We keep the conventions and notation of Subsect.~\ref{ss:conventions} and \secref{s:actions}. In particular, $k$ is an arbitrary field, and $Z$ denotes a $k$-space equipped with a $\BG_m$-action. The goal of this section is to construct and study a $k$-space $\wt{Z}$ equipped with a morphism $\wt{Z}\to\BA^1\times Z\times Z$ such that for $t\in\BA^1-\{0\}$ the fiber $\wt{Z}_t$ equals
the graph of the map $t:Z\iso Z$ and the space $\wt{Z}_0$ corresponding to $t=0$ equals $Z^+\underset{Z^0}\times Z^-$.

The organization of this section is as follows. In Subsect.~\ref{ss:tilde Z} we define the space $\wt{Z}$ and the main structures on it (e.g., the morphism $\wt{p}:\wt{Z}\to\BA^1\times Z\times Z$ and the action of 
$\BG_m^2$ on $\wt{Z}$). In Subsect.~\ref{ss:main results} we formulate the main results on $\wt{Z}$.
The fact that the space $\wt{Z}$ is algebraic  
is proved in Section~\ref{s:tildeZ}; however, in the case of a scheme equipped with a locally linear $\BG_m$-action the proof is much easier and is given in Subsect.~\ref{ss:locally linear}. In Subsect.~\ref{ss:affine} 
(resp.~\ref{ss:P^n}) we prove additional properties of $\wt{Z}$ in the case that $Z$ is affine 
(resp.~$Z=\BP^n$).

\ssec{The space $\wt{Z}$: definition and structures}    \label{ss:tilde Z}

\sssec{A family of hyperbolas}  \label{sss:family of hyperbolas}
Set $\BX:=\BA^2=\Spec k[\tau_1,\tau_2]$. We will always equip $\BX$ with the structure of a scheme over 
$\BA^1$ defined by the map
\[
\BA^2\to\BA^1 ,\quad (\tau_1,\tau_2)\mapsto \tau_1\cdot \tau_2\, .
\]

Let $\BX_t$ denote the fiber of $\BX$ over $t\in\BA^1$; in other words, $\BX_t\subset\BA^2$ is the curve defined by the equation $\tau_1\tau_2=t$. If $t\ne 0$ then $\BX_t$ is a hyperbola, while $\BX_0$ is the ``coordinate cross" $\tau_1\tau_2=0$. One has $\BX_0=\BX_0^+\cup\BX_0^-$, where
\begin{equation}  \label{e:X0+-}
\BX_0^+:=\{(\tau_1,\tau_2)\in\BA^2\,|\,  \tau_2=0\}\, , \quad \BX_0^-:=\{(\tau_1,\tau_2)\in\BA^2\,|\,  \tau_1=0\}\, .
\end{equation}

\sssec{The schemes $\BX_S\,$}   \label{sss:X_S}
For any scheme $S$ over  $\BA^1$ set
\begin{equation}
\BX_S:=\BX \underset{\BA^1}\times S\, .
\end{equation}
If $S=\Spec R$ we usually write $\BX_R$ instead of $\BX_S\,$.


\sssec{The $\BG_m$-action on $\BX_S\,$}    \label{sss:theaction}
We equip $\BX$ with the following hyperbolic $\BG_m$-action:
\begin{equation}   \label{e:hyperbolic}
\lambda\cdot (\tau_1,\tau_2):=(\lambda\cdot \tau_1,\; \lambda{}^{-1}\cdot \tau_2).
\end{equation}
This action preserves the morphism $\BA^2\to\BA^1$, so for any scheme $S$ over $\BA^1$ one gets an action of $\BG_m$ on $\BX_S\,$.

\begin{rem}  \label{r:theaction}
If $S$ is over $\BA^1-\{ 0\}$ then $\BX_S$ is $\BG_m$-equivariantly isomorphic to $\BG_m\times S$.
On the other hand, the ``coordinate cross" $\BX_0$ has irreducible components $\BX_0^{\pm}$ such that
$\BX_0^+$ (resp. $\BX_0^-$) is 
$\BG_m$-equivariantly isomorphic to $\BA^1$ (resp. to the scheme $\BA^1_-$ defined in 
Subsect.~\ref{ss:repeller}).
\end{rem}

\sssec{The space $\wt{Z}$}   \label{sss:thespace}
Given a $k$-space $Z$ equipped with a $\BG_m$-action, define
$\wt{Z}$ to be the following space over $\BA^1$: for any scheme $S$ over $\BA^1$ 
$$\Maps_{\BA^1} (S, \wt{Z}):=\gMaps (\BX_S,Z).$$
In other words, for any $k$-scheme $S$, an $S$-point of $\wt{Z}$ is a pair consisting of a morphism 
$S\to\BA^1$ and a $\BG_m$-equivariant morphism $\BX_S\to Z$.

Note that for any $t\in\BA^1(k)$ the fiber $\wt{Z}_t$ has the following description:
\begin{equation}
\wt{Z}_t=\GMaps (\BX_t\, ,Z).
\end{equation}

\begin{rem}   \label{r:alg space}
Later we will prove (see Theorem~\ref{t:tildeZ} and \propref{p:schemeness}) that if $Z$ is an algebraic $k$-space of finite type (resp. a $k$-scheme of finite type)  then so is $\wt{Z}$. For the spaces
$\wt{Z}\underset{\BA^1}\times (\BA^1-\{ 0\})$ and  $\wt{Z}_0:=\wt{Z}\underset{\BA^1}\times \{ 0\}$ 
this follows from the easy Propositions~\ref{p:outside 0} and \ref{p:tilde Z_0} below (the latter has to be combined with Theorem~\ref{t:attractors}).
\end {rem}

\sssec{The morphism $\wt{p}:\wt{Z}\to \BA^1\times Z\times Z$}  \label{sss:tilde p}
Any section $\sigma :\BA^1\to\BX$ of the morphism $\BX\to\BA^1$ defines a map
$\sigma^*:\gMaps (\BX_S\,,Z)\to\Maps (S\,,Z)$ and therefore a morphism $\wt{Z}\to Z$.
Let $\pi_1:\wt{Z}\to Z$ and $\pi_2:\wt{Z}\to Z$ denote the morphisms corresponding to the sections
\begin{equation}   \label{e:two sections}
t\mapsto (1,t)\in \BX_t \quad \text{ and }\quad t\mapsto (t,1)\in \BX_t\, ,
\end{equation}
respectively.  Let
\begin{equation}   \label{e:tilde p}
\wt{p}:\wt{Z}\to \BA^1\times Z\times Z
\end{equation}
denote the morphism whose first component is the tautological projection $\wt{Z}\to \BA^1$, and the second and the third components are $\pi_1$ and $\pi_2$, respectively.

For $t\in \BA^1$ let 
$$\wt{p}_t:\wt{Z}_t\to Z\times Z$$
denote the morphism induced by \eqref{e:tilde p} (as before, $\wt{Z}_t$ stands for the fiber of $\wt{Z}$ over $t$).

\medskip

It is clear that $(\wt{Z}_1,\wt{p}_1)$ identifies with $(Z,\Delta_Z:Z\to Z\times Z)$.
More generally, for $t\in \BA^1-\{0\}$ the pair $(\wt{Z}_t,\wt{p}_t)$ identifies with the graph of the map $Z\to Z$ given by the action of $t\in \BG_m\,$. Here is a slightly more precise statement.

\begin{prop}   \label{p:outside 0}
The morphism \eqref{e:tilde p} induces an isomorphism between  
$$\BG_m\underset{\BA^1}\times \wt{Z}\subset\wt{Z}$$ 
and the graph of the action morphism $\BG_m\times Z\to Z$.
\qed
\end{prop}

\begin{rem}
Later we will show that if the $\BG_m$-action on $Z$ extends to an $\BA^1$-action and if $Z$ is an algebraic $k$-space of finite type then the \emph{whole}
space $\wt{Z}$ identifies with the graph of the $\BA^1$-action, see \propref{p:2contracting}.
\end{rem}

\sssec{The space $\wt{Z}_0$}  \label{sss:tilde Z_0}
Now let us construct a canonical isomorphism $\wt{Z}_0\iso Z^+\underset{Z^0}\times Z^-$.

Recall that $\wt{Z}_0=\GMaps (\BX_0\, ,Z)$ and 
$\BX_0=\BX_0^+\cup\BX_0^-$, where $\BX_0^+$ and $\BX_0^-$ are defined by formula~\eqref{e:X0+-}.
One has $\BG_m$-equivariant isomorphisms
\begin{equation}
\BA^1\iso\BX_0^+, \; \;  s\mapsto (s,0); \quad\quad\quad \BA^1_-\iso\BX_0^-,  \; \;  s\mapsto (0,s^{-1}).
\end{equation}
They define a morphism
$$\wt{Z}_0=\GMaps (\BX_0\, ,Z)\to\GMaps (\BX_0^+ ,Z)\iso\GMaps (\BA^1 ,Z)=Z^+$$
and a similar morphism $\wt{Z}_0\to Z^-$.

\begin{prop}   \label{p:tilde Z_0}
Assume that the $k$-space $Z$ is algebraic.

(i) The above morphisms $\wt{Z}_0\to Z^{\pm}$ induce an isomorphism
\begin{equation}   \label{e:tilde Z_0}
\wt{Z}_0\iso Z^+\underset{Z^0}\times Z^-,
\end{equation}
where the fiber product is taken with respect to the maps $q^{\pm}:Z^{\pm}\to Z^0$ from Subsections~\ref{sss:structures}(iii) and \ref{ss:repeller}.

(ii) The corresponding diagram
\begin{equation}    \label{e:over Z times Z}
\xymatrix{
\wt{Z}_0 \ar[d]^{}\ar[r]^{\wt{p}_0}& Z\times Z\\\
Z^+\underset{Z^0}\times Z^-\ar@{^{(}->}[r]^{}&Z^+\times Z^-\ar[u]_{{p^+}\times {p^-}}
    }
\end{equation}
commutes.
\end{prop}

\begin{proof}
It is easy to check that our morphisms $\wt{Z}_0\to Z^{\pm}$ induce a morphism
\begin{equation}   \label{e:weaker tilde Z_0}
\wt{Z}_0\to Z^+\underset{Z^0}\times Z^-
\end{equation}
and that the corresponding diagram \eqref{e:over Z times Z} commutes. To prove that the map
\eqref{e:weaker tilde Z_0} is an isomorphism, apply the following well known lemma for
\[
Y=\BX_0\times S, \quad Y_1=\BX_0^+\times S, \quad Y_2=\BX_0^-\times S,
\]
where $S$ is a test scheme.
\end{proof}


\begin{lem}  \label{l:pushout}
Let $Y$ be a scheme and $Y_1,Y_2\subset Y$ closed subschemes whose scheme-theore\-ti\-cal union\footnote{By this we mean the supremum of  $Y_1$ and $Y_2$ in the poset of closed subschemes.}
equals $Y$. 
Then the square
\[
\CD
Y_1\cap Y_2 @>{}>>  Y_1 \\
@V{}VV   @V{}VV   \\ 
Y_2 @>{}>> Y
\endCD
\]
is co-Cartesian in the category of algebraic spaces; that is, for any algebraic space $Z$ the map
\begin{equation}  \label{e:closed_subsch}
\Maps (Y,Z)\to\Maps (Y_1,Z)\underset{\Maps (Y_1\cap Y_2,Z)}\times\Maps (Y_2,Z)
\end{equation}
is bijective.
\end{lem}

\begin{proof}
If $Z$ is an affine scheme the map \eqref{e:closed_subsch} is clearly bijective. Bijectivity of  
\eqref{e:closed_subsch} for any scheme $Z$ easily follows. For an arbitrary algebraic space $Z$ bijectivity of \eqref{e:closed_subsch} follows from the case where $Z$ is an affine scheme and the following result 
\cite[exp. IX, Theorem 4.7]{SGA1}: let $\Et^{\sep}_{\fin} (Y)$ be the category of separated\footnote{The separatedness assumption is harmless because any morphism from an affine scheme to an algebraic space $Z$ is separated (even if $Z$ itself is not separated).} etale schemes of finite type over $Y$, then the functor
\[
\Et^{\sep}_{\fin} (Y)\to\Et^{\sep}_{\fin} (Y_1)\underset{\Et^{\sep}_{\fin} (Y_1\cap Y_2)}\times\Et^{\sep}_{\fin} (Y_2)
\]
is an equivalence.
\end{proof}


\begin{prop}   \label{p:closed and open}
(i) Let $Y\subset Z$ be a $\BG_m$-stable closed subspace. Then the diagram
\[
\xymatrix{
 \wt{Y} \ar[d]_{\wt{p}_Y} \ar[r]^{}& \wt{Z} \ar[d]^{\wt{p}_Z}\\\
\BA^1\times Y\times Y\;\ar@{^{(}->}[r]^{}&\BA^1\times Z\times Z
    }
\]
is Cartesian. In particular, the morphism $\wt{Y}\to\wt{Z}$ is a closed embedding.

(ii) Let $Y\subset Z$ be a $\BG_m$-stable open subspace. Then the above diagram identifies
$\wt{Y}$ with an open subspace of the fiber product 
\[
\wt{Z}\underset{\BA^1\times Z\times Z}\times (\BA^1\times Y\times Y)\, .
\]
In particular, the morphism $\wt{Y}\to\wt{Z}$ is an open embedding.

\begin{proof}
(i) Let $S$ be a scheme over $\BA^1$ and $f:\BX_S\to Z$ a $\BG_m$-equivariant morphism. 
Formula~\eqref{e:two sections} defines two sections of the map $\BX_S\to S$. We have to show that if $f$ maps these sections to $Y\subset Z$ then $f(\BX_S)\subset Y$. By 
$\BG_m$-equivariance, we have 
$f(\BX'_S)\subset Y$, where $\BX'$ is the open subscheme $\BA^2-\{ 0\}\subset\BA^2=\BX$ and 
$\BX'_S:=\BX'\times_{\BA^1}S$. Since $\BX'_S$ is schematically dense in $\BX_S$ this implies that
$f(\BX_S)\subset Y\,$.

(ii) Just as before, we have a $\BG_m$-equivariant morphism $f:\BX_S\to Z$ such that $f(\BX'_S)\subset Y$.
The problem is now to show that the set
\[
\{ s\in S\,|\,\BX_s\subset f^{-1}(Y)\}
\]
is open in $S\,$. Indeed, its complement equals $\pr_S(\BX_S -f^{-1}(Y))$, where  $\pr_S :\BX_S\to S$ is the projection. The set $\pr_S(\BX_S -f^{-1}(Y))$ is closed in $S$ because
$\BX_S -f^{-1}(Y)$ is a closed subset of $\BX_S -\BX'_S$ and the morphism $\BX_S -\BX'_S\to S$ is closed (in fact, it is a closed embedding).
\end{proof}
\end{prop}

\sssec{Anti-action of $\BA^1\times\BA^1$ on $\wt{Z}$}  \label{sss:anti-action}
The reader may prefer to skip the rest of Subsect.~\ref{ss:tilde Z} for a while and proceed to 
Subsect.~\ref{ss:main results}.

 As usual, we consider $\BA^1$ as a monoid with respect to multiplication. In this subsection we will define an ``anti-action" of the monoid $\BA^2=\BA^1\times\BA^1$ on $\wt{Z}$ (the meaning of the word ``anti-action" will become clear soon). In Subsect.~\ref{sss:Gm2 on A2} we will use it to define an action of $\BG_m^2$ on 
 $\wt{Z}$.

The idea is as follows. Recall that $\BX:=\BA^2$, so the monoid $\BA^2$ acts on $\BX$. In particular, for any
$\lambda =(\lambda_1,\lambda_2)\in\BA^2(k)$ the action of $\BA^2$ on $\BX$ defines a morphism 
$\BX\to\BX\,$. For any $t\in\BA^1 (k)$ it induces a $\BG_m$-equivariant morphism 
$\BX_t\to\BX_{\lambda_1 t\lambda_2}$ (recall that $\BX_t$ denotes the fiber over $t$). Since 
$\wt{Z}_t:=\GMaps (\BX_t ,Z)$ we get a morphism
$\wt{Z}_{\lambda_1 t\lambda_2}\to\wt{Z}_t\,$. (Of course, if
$\lambda_1, \lambda_2\ne 0$ one can invert this morphism and get a morphism in the ``expected" direction, i.e., $\wt{Z}_t\to\wt{Z}_{\lambda_1 t\lambda_2}\,$.)

Same story if one works with $S$-points rather than $k$-points. Namely, suppose we have a $k$-scheme $S$ and $k$-morphisms $t:S\to\BA^1$ and $\lambda_1,\lambda_2:S\to\BA^1$. Let $\BX_t$ (resp. $\wt{Z}_t$) denote the fiber product $\BX\underset{\BA^1}\times S$ (resp. $\wt{Z}\underset{\BA^1}\times S$) with respect to $t:S\to\BA^1$. The morphism
\[
\BX\times S\overset{(\lambda_1,\,\lambda_2)}\longrightarrow \BX\times S
\]
maps $\BX_t\subset\BX\times S$ to $\BX_{\lambda_1 t\lambda_2}\subset\BX\times S$, and the 
$\BG_m$-equivariant morphism $\BX_t\to\BX_{\lambda_1 t\lambda_2}$ induces an $S$-morphism
\begin{equation}   \label{e:anti-action}
\phi_{\lambda_1,\,\lambda_2, t}:\wt{Z}_{\lambda_1 t\lambda_2}\to\wt{Z}_t\, .
\end{equation}
The morphisms \eqref{e:anti-action} have the following properties:

(i) compatibility with base change $S'\to S$;

(ii) $\phi_{1,1, t}$ equals the identity;

(iii) $\phi_{\lambda_1 \lambda'_1,\,\lambda_2 \lambda'_2, t}=
\phi_{\lambda_1,\,\lambda_2, t}\circ\phi_{\lambda'_1,\,\lambda'_2,\, \lambda_ 1t\lambda_2}\:$.

We use the word ``anti-action" to denote this structure on the triple $(\wt{Z}, \BA^1,\wt{Z}\to\BA^1)$.

\begin{rem}   \label{r:anti-action}
An additional property of the above anti-action will be formulated later, see 
Subsect.~\ref{sss:property of anti-action}.
\end{rem}

\begin{rem}  
In Appendix~\ref{s: 2-categorical framework} we will translate the barbarism "anti-action" into a more standard 
categorical language. Note that the notion of twisted arrow category (see Subsect.~\ref{ss:Tw}) is implicit in 
formula~\eqref{e:anti-action}.
\end{rem}

\noindent {\bf Exercise.} Describe the compositions 
$$\phi_{1,0,1}:\wt{Z}_0\to\wt{Z}_1\iso Z\, ,\quad \phi_{0,1,1}:\wt{Z}_0\to\wt{Z}_1\iso Z\, ,\quad
  \phi_{0,0,1}:\wt{Z}_0\to\wt{Z}_1\iso Z$$
and the idempotent endomorphisms
$\phi_{1,0,0}\, ,\, \phi_{0,1,0}\, ,\,  \phi_{0,0,0}\in\End (\wt{Z}_0)$ in terms of the isomorphism 
$$\wt{Z}_0\iso Z^+\underset{Z^0}\times Z^-$$ 
from \propref{p:tilde Z_0}.

\sssec{Action of $\BG_m^2$ on $\wt{Z}$}  \label{sss:Gm2 on A2}
We equip $\BA^1$ with the following action of $\BG_m^2\,$:
\begin{equation}   \label{e:Gm on A1}
(\lambda_1,\lambda_2)*t:=\lambda_1^{-1}\lambda_2t\, , \quad \quad 
\lambda_i\in\BG_m, \;\, t\in\BA^1 .
\end{equation}
We lift it to an  action of $\BG_m^2$ on $\wt{Z}$ using the isomorphisms
\begin{equation}   \label{e:Gm2 on tilde Z}
(\phi_{\lambda_1^{-1},\,\lambda_2\,, t})^{-1}=
\phi_{\lambda_1,\,\lambda_2^{-1},\,\lambda_1^{-1} t\lambda_2}:\wt{Z}_t\iso \wt{Z}_{\lambda_1^{-1} t\lambda_2}
\end{equation}
where $\phi$ is the morphism \eqref{e:anti-action}.

It is easy to check that the morphism $\wt{p}:\wt{Z}\to \BA^1\times Z\times Z$ 
and the isomorphism $$\wt{Z}_0\iso Z^+\underset{Z^0}\times Z^-$$ from \propref{p:tilde Z_0} are $\BG_m^2$-equivariant. So all morphisms in diagram \eqref{e:over Z times Z} are $\BG_m^2$-equivariant.

\begin{rem}
In \cite{DrGa1} we make a different choice of signs in formulas \eqref{e:Gm on A1}-\eqref{e:Gm2 on tilde Z}.
Namely, the action of $\BG_m^2$ on $\BA^1$ is defined there by $(\lambda_1,\lambda_2)*t:=\lambda_1^{-1}\lambda_2^{-1}t$, and its lift to an action of $\BG_m^2$ on $\wt{Z}$ is defined using 
 the isomorphism 
 $(\phi_{\lambda_1^{-1},\,\lambda_2^{-1}, t})^{-1}:\wt{Z}_t\iso \wt{Z}_{\lambda_1^{-1} t\lambda_2^{-1}}\,$.
\end{rem}

\ssec{Main results on $\wt{Z}$}   \label{ss:main results}

\sssec{Formulation of the main results} 
\begin{thm}   \label{t:tildeZ}
Let $Z$ be any algebraic $k$-space of finite type equipped with a $\BG_m$-action. Then
$\wt{Z}$ is an algebraic $k$-space of finite type.
\end{thm}

In full generality, the theorem will be proved in Section~\ref{s:tildeZ}. In Subsect.~\ref{ss:locally linear} 
we will prove it in the case that $Z$ is a scheme equipped with a locally linear $\BG_m$-action (moreover, we will show that under these assumptions $\wt{Z}$ is a scheme). This case is enough for most practical purposes.

\medskip

\emph{From now on we assume that $Z$ is an algebraic $k$-space of finite type.}

\begin{prop}   \label{p:schemeness}
(i) If $Z$ is separated then so is $\wt{Z}$.

(ii) If $Z$ is a scheme then so is $\wt{Z}$.
\end{prop}

The proof will be given in Subsect.~\ref{sss:schemeness} below.

\begin{prop}   \label{p:2smoothness}
If $Z$ is smooth then the canonical morphism $\wt{Z}\to\BA^1$ is smooth.
\end{prop}

\begin{proof}
It suffices to check formal smoothness. We proceed just as in the proof of \propref{p:smoothness}.
Let $R$ be a $k$-algebra equipped with a morphism $\Spec R\to\BA^1$.  Let $\bar R=R/I$, where 
$I\subset R$ is an ideal with $I^2=0$. Let $\bar f\in\gMaps (\BX_{\bar R}\, , Z)$. We have to show that $\bar f$ can be lifted to an element of $\gMaps (\BX_R\, , Z)$. Since $\BX_R$ is affine and $Z$ is smooth there is no obstruction to lifting $\bar f$ to an element of $\Maps (\BX_R, Z)$. Then standard arguments show that the obstruction to existence of a $\BG_m$-equivariant lift is in $H^1(\BG_m\, ,M)$, where 
$M:=H^0(\BX_{\bar R}\, ,\bar f^*\Theta_Z)\otimes_{\bar R}I$ and $\Theta_Z$ is the tangent bundle.
But $H^1$ of $\BG_m$ with coefficients in any $\BG_m$-module is zero.
\end{proof}

\sssec{Properties of $\wt{p}:\wt{Z}\to \BA^1\times Z\times Z$}   \label{sss:props tilde p}

\begin{prop}   \label{p:props tilde p}
(i) The morphism $\wt{p}:\wt{Z}\to \BA^1\times Z\times Z$ is unramified.

(ii) If $Z$ is separated then $\wt{p}$ is a monomorphism.
\end{prop}

\begin{proof}
Theorem~\ref{t:tildeZ} implies that properties (i) and (ii) can be checked fiberwise. By \propref{p:outside 0}, the map $\wt{p}:\wt{Z}\to \BA^1\times Z\times Z$ is a monomorphism over $\BA^1-\{ 0\}\subset\BA^1$. It remains to consider the morphism $\wt{p}_0:\wt{Z}_0\to Z\times Z$. By \propref{p:tilde Z_0}, this is equivalent to considering the composition 
\[
Z^+\underset{Z^0}\times Z^-\mono Z^+\times Z^-\overset{p^+\times p^-}\longrightarrow Z\times Z\, .
\]
By \propref{p: p^+}(i-ii), this composition is unramified, and if $Z$ is separated then it is a monomorphism.
\end{proof}

\begin{rem}
If $Z$ is affine then $\wt{p}:\wt{Z}\to \BA^1\times Z\times Z$ is a closed embedding,
see \propref{p:2new tilde} below.
\end{rem}

\begin{rem}   \label{r:P^n}
Suppose that $Z$ admits a $\BG_m$-equivariant locally closed embedding into
a projective space $\BP(V)$, where $\BG_m$ acts linearly on $V$. We claim that in this case the morphism
$\wt{p}:\wt{Z}\to \BA^1\times Z\times Z$ is a \emph{locally} closed 
embedding\footnote{Note that the map $p^{\pm}:Z^{\pm}\to Z$ is typically not a locally closed embedding, see Example \ref{ex:P^1}.}. By \propref{p:closed and open}, it suffices to check this if $Z=\BP(V)$. This will be done in Subsect.~\ref{ss:P^n} below.
\end{rem}

\begin{rem}
If $Z$ is the projective line $\BP^1$ equipped with the usual $\BG_m$-action then the map
$\wt{p}:\wt{Z}\to \BA^1\times Z\times Z$ is \emph{not a closed} embedding (because the scheme
$\wt{Z}_0=Z^+\underset{Z^0}\times Z^-$ is not proper).
\end{rem}

\begin{rem}
Let $Z$ be the curve obtained from $\BP^1$ by gluing $0$ with $\infty$. Equip $Z$ with the $\BG_m$-action induced by the usual action on $\BP^1$. Then $\wt{p}:\wt{Z}\to \BA^1\times Z\times Z$ is not a locally closed embedding. In fact, already $\wt{p}_0:\wt{Z}_0\to Z\times Z$ is not a locally closed embedding (because the maps $p^{\pm}:Z^{\pm}\to Z$ are not).
\end{rem}

\sssec{Description of $\wt{Z}$ if the $\BG_m$-action on $Z$ extends to an action of $\BA^1$ or $\BA^1_-\,$.} \label{sss:contractive}
Recall that by \propref{p:contracting}, a $\BG_m$-action on $Z$ has at most one extension to an action of the multiplicative monoid $\BA^1$ and such extension exists if and only if the morphism $p^+:Z^+\to Z$ is an isomorphism. Of course, this remains true if $\BA^1$ is replaced by the monoid $\BA^1_-$ 
defined in Subsect.~\ref{ss:repeller} and $p^+$ is replaced by $p^-:Z^-\to Z\,$.

\begin{prop} \label{p:2contracting}
Suppose that a $\BG_m$-action on $Z$ extends to an $\BA^1$-action. Then

(i) the morphism $\wt{p}:\wt{Z}\to \BA^1\times Z\times Z$ is a monomorphism, which identifies  
$\wt{Z}$ with the graph of the $\BA^1$-action on $Z$; in particular, the composition
\begin{equation}  \label{e:first iso}
\wt{Z}\overset{\wt{p}}\longrightarrow\BA^1\times Z\times Z\to\BA^1\times Z\times\Spec k=\BA^1\times Z
\end{equation}
is an isomorphism;

(ii) the inverse of \eqref{e:first iso} is the morphism
\begin{equation}   \label{e:beta}
\BA^1\times Z\to\wt{Z}
\end{equation}
corresponding to the $\BG_m$-equivariant map $\BX\times Z\to Z$ defined by
\[
(\tau_1,\tau_2,z)\mapsto\tau_1\cdot z\, ,\quad\quad (\tau_1,\tau_2)\in\BX\, ,\; z\in Z\, .
\]
\end{prop}

\begin{proof} 
Let $\alpha :\wt{Z}\to\BA^1\times Z$ denote the composition \eqref{e:first iso} and 
$\beta :\BA^1\times Z\to\wt{Z}$ the morphism~\eqref{e:beta}. It is easy to see that $\alpha\circ\beta=\id$.
The problem is to show that $\beta\circ\alpha=\id$. To do this, it suffices to prove that $\alpha$ is a monomorphism. But being a monomorphism is a fiberwise condition, so it suffices to show that $\beta$ induces an isomorphism between fibers over any $t\in\BA^1$. For $t\ne 0$ this follows from \propref{p:outside 0}. If $t=0$ then by \propref{p:tilde Z_0}, the morphism in question is the composition 
$$Z^+\underset{Z^0}\times Z^-\to Z^+\overset{p^+}\longrightarrow Z.$$ 
By \propref{p:contracting}, $p^+$ is an isomorphism.
So the projection $q^-:Z^-\to Z^0$ is also an isomorphism by \corref{c:contractive}(i).
\end{proof}

The above proposition formally implies the following one.

\begin{prop} \label{p:dilating}
Suppose that a $\BG_m$-action on $Z$ extends to an action of the monoid $\BA^1_-\,$. Then

(i) the morphism $\wt{p}:\wt{Z}\to \BA^1\times Z\times Z$ is a monomorphism, which identifies  
$\wt{Z}$ with 
\[
\{(t,z_1,z_2\,)\in\BA^1\times Z\times Z\,|\, z_1=t^{-1}\cdot z_2\,\};
\]
in particular, the composition
\begin{equation} \label{e:second iso}
\wt{Z}\overset{\wt{p}}\longrightarrow\BA^1\times Z\times Z\to\BA^1\times \Spec k\times Z=\BA^1\times Z
\end{equation}
is an isomorphism;

(ii) the inverse of \eqref{e:second iso} is the morphism
\begin{equation}   \label{e:2beta}
\BA^1\times Z\to\wt{Z}
\end{equation}
corresponding to the $\BG_m$-equivariant map $\BX\times Z\to Z$ defined by
\[
(\tau_1,\tau_2,z)\mapsto\tau_2^{-1}\cdot z\, ,\quad\quad (\tau_1,\tau_2)\in\BX\, ,\; z\in Z\, .
\]
\end{prop}

\sssec{Proof of \propref{p:schemeness}}   \label{sss:schemeness}
If $Z$ is separated then $\wt{p}:\wt{Z}\to \BA^1\times Z\times Z$ is a monomorphism by 
\propref{p:props tilde p}(ii). Any monomorphism is separated. Proposition~\ref{p:schemeness}(i) follows.

To prove  \propref{p:schemeness}(ii), we need the following well known fact.

\begin{prop}
A separated quasi-finite morphism between algebraic spaces is quasi-affine. In particular, it is schematic.
\end{prop}

For the proof, see \cite[Theorem A.2]{LM}  or  \cite[ch. II, Theorem 6.15]{Kn}. 

\begin{cor}   \label{c:schematic}
A monomorphism between algebraic $k$-spaces of finite type is schematic.
\end{cor}

%

Now it is easy to prove  \propref{p:schemeness}(ii) under an additional assumption that $Z$ is separated: indeed, in this case  \propref{p:props tilde p}(ii) allows to apply  \corref{c:schematic} to the morphism
$\wt{p}:\wt{Z}\to \BA^1\times Z\times Z$. 

To prove \propref{p:schemeness}(ii) in general, we will apply
\corref{c:schematic} to a more complicated morphism $\wt{p}'$ constructed below.

Recall that $\BX:=\BA^2=\Spec k[\tau_1,\tau_2]$ and that $\BX$ is equipped with the structure of a scheme over 
$\BA^1$ defined by the map $(\tau_1,\tau_2)\mapsto \tau_1\tau_2\,$. Let $\BB\subset\BX$ be the line defined by the equation $\tau_1=\tau_2$, then $\BB$ is finite and flat over $\BA^1$. So for any open subscheme $U\subset Z$
there is an algebraic space $\underline U$ over $\BA^1$ such that
for any scheme $S$ over $\BA^1$ 
$$\Maps_{\BA^1} (S,\underline U):=\Maps (\BB\underset{\BA^1}\times S,U);$$
moreover, if $U$ is affine 
then $\underline U$ is scheme. The canonical morphism 
$\underline U\to\underline Z$ is an open embedding.

The embedding $\BB\mono\BX$ induces a morphism $\alpha :\wt{Z}\to\underline Z$ over $\BA^1$.
Combining it with $\wt{p}:\wt{Z}\to \BA^1\times Z\times Z$ one gets a morphism
\[
\wt{p}':\wt{Z}\to \underline Z\underset{\BA^1}\times(\BA^1\times Z\times Z)=\underline Z\times Z\times Z\, .
\]

\begin{lem}  \label{l:tilde p'}
Suppose that $Z$ is a scheme (or more generally, an algebraic space such that the diagonal map 
$Z\to Z\times Z$  is a locally closed embedding). Then
the map $\wt{p}':\wt{Z}\to\underline Z\times Z\times Z$ is a monomorphism.
\end{lem}

\begin{proof}
Follows from \remref{r:prigoditsya} combined with Propositions~\ref{p:outside 0} and \ref{p:tilde Z_0}.
\end{proof}


\begin{proof}[Proof of \propref{p:schemeness}(ii)]
Let $Z$ be a scheme. Then $\wt{Z}\underset{\BA^1}\times (\BA^1-\{ 0\}$ is a scheme by \propref{p:outside 0}.
So to prove that  $\wt{Z}$ is a scheme, it suffices to show that for any point $\zeta\in\wt{Z}_0$ there exists an open subscheme $V\subset\wt{Z}$ containing $\zeta\,$. Let $z\in Z$ be the image of $\zeta$ under the composition
\[
\wt{Z}_0\iso Z^+\underset{Z^0}\times Z^-\to Z^0\mono Z\, .
\]
Let $U\subset Z$ be an open affine containing $z$. Then the open subspace 
$\underline U\subset\underline Z$ is a scheme. Define an open subspace $V\subset\wt{Z}$ by
$$V:=(\wt{p}')^{-1}(\underline U\times Z\times Z)\, .$$
By \lemref{l:tilde p'} and \corref{c:schematic}, $V$ is a scheme. It is clear that $\zeta\in V$.
\end{proof}

\ssec{The case where $Z$ is affine}  \label{ss:affine}
\sssec{The scheme $\BX_R\,$}  
Let $R$ be an algebra over $k[t]$, so $S:=\Spec R$ is a scheme over $\BA^1$. In this situation
the scheme $\BX_S:=\BX \underset{\BA^1}\times S$ introduced in Subsect.~\ref{sss:X_S}
will be denoted by $\BX_R\,$. It has the following explicit description:
 \begin{equation}  \label{e:X_R}
\BX_R:=\Spec A_R, \quad \mbox{ where } A_R:=R[\tau_1,\tau_2]/(\tau_1\tau_2-t)\, .
\end{equation}

It is clear that $A_R$ is a free $R$-module with basis $e_n\,$, $n\in\BZ\,$, where
\begin{equation}   \label{e:e_n}
e_n=\tau_1^n \;\;\mbox{ for } n\ge 0, \quad e_n=\tau_2^{-n} \;\;\mbox{ for }  n\le 0\, .
\end{equation}
The $\BG_m$-action on $\BX_R$ defines a $\BZ$-grading on $A_R$. The element  $e_n$ defined by 
\eqref{e:e_n} has degree $n$ with respect to this grading.

\sssec{The space $\wt{Z}$ in the case that $Z$ is affine}  Recall that $\wt{Z}$ is the space over $\BA^1$
such that 
\begin{equation}   \label{e:2tilde Z}
\Maps_{\BA^1} (\Spec R, \wt{Z}):=\gMaps (\BX_R\, ,Z)
\end{equation}
for any  algebra $R$ over $k[t]$.

\begin{prop} \label{p:2new tilde}
Assume that $Z$ is affine. Then the morphism $\wt{p}:\wt{Z}\to \BA^1\times Z\times Z$ is a closed embedding. In particular, $\wt{Z}$ is an affine $k$-scheme of finite type.
\end{prop}

\begin{proof}
If $Z$ is a closed subscheme of an affine scheme $Z'$ and the proposition holds for $Z'$ then it holds for $Z$ by  \propref{p:closed and open}(i). So we are reduced to the case that $Z$ is a finite-dimensional vector space equipped with a linear $\BG_m$-action.   

If the proposition holds for affine schemes $Z_1$ and $Z_2$ then it holds for $Z_1\times Z_2\,$.
So we are reduced to the case that $Z=\BA^1$ and $\lambda\in\BG_m$ acts on $\BA^1$ as multiplication by 
$\lambda^n$, $n\in\BZ$. 

In this case it is straightforward to compute $\wt{Z}$ and $\wt{p}$ using 
\eqref{e:2tilde Z}, \eqref{e:X_R}, and the definition of $\wt{p}$ from Subsect.~\ref{sss:tilde p}.
In particular, one checks that $\wt{p}$ identifies $\wt{Z}$ with the closed subscheme of 
$\BA^1\times Z\times Z$ defined by the equation $x_2=t^nx_1$ if $n\ge 0$ and by the equation 
$x_1=t^{-n}x_2$ if $n\le 0$ (here $t,x_1,x_2$ are the coordinates on $\BA^1\times Z\times Z=\BA^3$).
\end{proof}

As before, assume that $Z$ is affine.  Then by \propref{p:2new tilde}, the morphism $\wt{p}$ identifies $\wt{Z}$ with the closed subscheme $\wt{p}(\wt{Z})\subset\BA^1\times Z\times Z$. By \propref{p:outside 0}, the intersection of $\wt{p}(\wt{Z})$ with the open subscheme 
$$\BG_m\times Z\times Z\subset \BA^1\times Z\times Z$$
is equal to the graph of the action map $\BG_m\times Z\to Z$.
Hence, $\wt{Z}$ contains the closure of the graph in $\BA^1\times Z\times Z$. In general, this containment is
not an equality\footnote{E.g., take $Z$ to be the hypersurface in $\BA^{2n}$ defined by the equation 
$x_1y_1+\ldots x_ny_n=0$ and define the $\BG_m$-action by $\tilde x_i=\lambda x_i\,$, 
$\tilde y_i=\lambda^{-1} y_i\,$.}. However, one has the following

\begin{prop}  \label{p:affine and smooth}
If $Z$ is affine and smooth then 
$$\wt{p}(\wt{Z})=\overline{\Gamma},$$
where $\Gamma\subset\BG_m\times Z\times Z$ is the graph of of the action map $\BG_m\times Z\to Z$
and  $\overline{\Gamma}$ denotes its scheme-theoretical closure in $ \BA^1\times Z\times Z\,$.
\end{prop}

\begin{proof}
This immediately follows from \propref{p:2smoothness}. 
%
\end{proof}

\sssec{Explicit description of $\wt{Z}$ in the case that $Z$ is affine}  
This subsection can be skipped by the reader.

Define a map $\mu :\BZ\times\BZ\to\BZ_+$ by
\begin{equation}
\mu (n_1,n_2):=(|n_1|+|n_2|-|n_1+n_2|)/2\, .
\end{equation}
So if $n_1,n_2$ are nonzero and have opposite signs then  $\mu (n_1,n_2)=\min (|n_1|,|n_2|)$; otherwise one has $\mu (n_1,n_2)=0$.

\begin{prop}
If $Z$ is the spectrum of a $\BZ$-graded $k$-algebra $B$ then $\wt{Z}=\Spec\wt{B}$, where $\wt{B}$ is the 
$k[t]$-algebra with generators 
$$[b],\quad b\in B_n\, ,\quad n\in\BZ\, ,$$ 
and defining relations
$$[b_1\cdot b_2]=t^{\mu (n_1,n_2)}\cdot [b_1]\cdot [b_2],\quad b_1\in B_{n_1}\, ,\, b_2\in B_{n_2}\, ,\quad
n_1,n_2\in\BZ\, ,$$

$$[\lambda_1b_1+\lambda_2b_2]=\lambda_1[b_1] +\lambda_2[b_2], 
\quad \lambda_i\in k,\, b_i\in B_n\, ,\, n\in\BZ\, .$$
\end{prop}

\begin{proof}
By \eqref{e:2tilde Z} and \eqref{e:X_R}, for any $k[t]$-algebra $R$, a   morphism of $k[t]$-algebras
$\wt{B}\to R$ is the same as a morphism of graded $k$-algebras $\varphi :B\to A_R\,$. Our $A_R$ is a free 
$R$-module whose basis is formed by elements $e_n$ defined by \eqref{e:e_n}. Let $B_n$ denote the $n$-th graded component of $B$, then for $b\in B_n$ one has  $\varphi (b)=\varphi_n (b)e_n$, where 
$\varphi_n :B_n\to R$ is some $k$-linear map.
It is easy to check that
\[
e_{n_1}e_{n_2}=t^{\mu (n_1,n_2)}e_{n_1+n_2}\, ,
\]
so the condition $\varphi (b_1b_2)=\varphi (b_1)\varphi (b_2)$ can be rewritten as
$$\varphi_{n_1+n_2}(b_1b_2)=t^{\mu (n_1,n_2)}\varphi_{n_1}(b_1)\varphi_{n_2}(b_2),\quad b_1\in B_{n_1}\, ,\, b_2\in B_{n_2}\, ,
\quad n_1,n_2\in\BZ\, .$$
The proposition follows. 
\end{proof}

\ssec{Proof of Theorem~\ref{t:tildeZ} in the case of a locally linear $\BG_m$-action}   
\label{ss:locally linear}
Let $Z$ be a $k$-scheme of finite type equipped with a $\BG_m$-action. Suppose that the action is locally linear, i.e., $Z$ can be covered by open affine $\BG_m$-stable subschemes $U_i\,$. Let us show that under this assumption $\wt{Z}$ is a $k$-scheme of finite type.

By \propref{p:2new tilde}, each $U_i$ is an affine $k$-scheme of finite type. By  \propref{p:closed and open}(ii), for each $i$ the canonical morphism $\wt{U}_i\to\wt{Z}$ is an open embedding. It remains to show that $\wt{Z}$ is covered by the open subschemes $\wt{U}_i\,$. 

It suffices to check that for each $t\in\BA^1$ the fiber $\wt{Z}_t$ is covered by the open subschemes 
$(\wt{U}_i)_t\,$. For $t\ne 0$ this is clear from \propref{p:outside 0}. It remains to consider the case $t=0$.

By \propref{p:tilde Z_0}, $\wt{Z}_0=Z^+\underset{Z^0}\times Z^-$. So a point of $\wt{Z}_0$ is a pair
$(z^+,z^-)\in Z^+\times Z^-$ such that $q^+(z^+)=q^-(z^-)$. The point $q^+(z^+)=q^-(z^-)$ is contained in some
$U_i\,$. By Lemma~\ref{l:U^+}, we have $z^+,z^-\in U_i\,$. So our point $(z^+,z^-)\in\wt{Z}_0$ belongs to 
$ (\wt{U}_i)_0\,$. \qed

%

%

\ssec{The morphism $\wt{p}:\wt{Z}\to \BA^1\times Z\times Z$ in the case $Z=\BP^n$}  \label{ss:P^n}
In this subsection (which can be skipped by the reader) we prove the following statement promised in 
\remref{r:P^n}.

\begin{prop}
Let $Z$ be a projective space $\BP^n$ equipped with an arbitrary $\BG_m$-action.
Then the morphism $\wt{p}:\wt{Z}\to \BA^1\times Z\times Z$ is a locally closed embedding.
\end{prop}

\begin{proof}
For a suitable coordinate system in $\BP^n$, the $\BG_m$-action is given by
\[
\lambda *(z_0:\ldots :z_n)=(\lambda^{m_0}z_0: \ldots :\lambda^{m_n}z_n),\quad \lambda\in\BG_m \, .
\]
Let $U_i\subset Z=\BP^n$ denote the open subset defined by the condition $z_i\ne 0$. It is affine, so by
\propref{p:2new tilde}, the canonical morphism $\wt{U}_i\to\BA^1\times U_i\times U_i$ is a closed embedding.
Thus to finish the proof of the proposition, it suffices to show that 
$\wt{p}^{-1}(\BA^1\times U_i\times U_i)=\wt{U}_i\,$. By \propref{p:outside 0}, 
$\wt{p}^{-1}(\BG_m\times U_i\times U_i)=\BG_m\underset{\BA^1}\times \wt{U}_i\,$. So it remains to prove that the morphism $\wt{p}_0:\wt{Z}_0\to Z\times Z$ has the following property:
$(\wt{p}_0)^{-1}(U_i\times U_i)=(\wt{U}_i)_0\,$. Identifying $\wt{Z}_0$ with $Z^+\underset{Z^0}\times Z^-$ and using \lemref{l:U^+}, we see that it remains to prove the following lemma.
\end{proof}

\begin{lem}
Let $z^+,z^-\in\BP^n$. Suppose that
\[
\lim_{\lambda\to 0}\lambda*z^+=\lim_{\lambda\to\infty }\lambda*z^-=\zeta\, .
\]
If $z^+,z^-\in\ U_i$ then $\zeta\in U_i\,$.
\end{lem}

\begin{proof}
Write $z^+=(z^+_0:\ldots :z^+_n)$, $z^-=(z^-_0:\ldots :z^-_n)$, $\zeta =(\zeta_0:\ldots :\zeta_n)$.
We have $z^{\pm}_i\ne 0$, and the problem is to show that $\zeta_i\ne 0$.

Suppose that $\zeta_i= 0$. Choose $j$ so that $\zeta_j\ne 0\,$. Then $z^{\pm}_j\ne 0$ and
\[
\lim_{\lambda\to 0}\lambda^{m_i-m_j}(z_i/z_j)=\zeta_i/\zeta_j=0, \quad
\lim_{\lambda\to \infty}\lambda^{m_i-m_j}(z_i/z_j)=\zeta_i/\zeta_j=0\, .
\]
This means that $m_i>m_j$ and $m_i<m_j$ at the same time, which is impossible.
\end{proof}

\section{Some openness results}  \label{s:openness}
In this section $Z$ denotes an algebraic $k$-space of finite type equipped with a $\BG_m$-action.

The main results 
are Propositions~\ref{p:2open embeddings} and \ref{p:2open embedding}.
They say that certain morphisms involving $\wt{Z}$ are open embeddings.


\propref{p:2open embeddings} is used in \cite{DrGa1} in a crucial way. 
Propositions~\ref{p:2open embeddings} and \ref{p:2open embedding} are both used in 
the proof of  \propref{p:openness}.


\ssec{The fiber products $Z^-\times_Z\wt{Z}$ and $\wt{Z}\times_ZZ^+$}   \label{ss:fiber products my}
The constructions and results of this subsection are used in \cite{DrGa1} (in the verification of the adjunction properties).

\sssec{Definition of the fiber products} \label{sss:the fiber products}
In Subsect.~\ref{sss:tilde p} we defined morphisms $\pi_1,\pi_2:\wt{Z}\to Z$.
We will study the fiber product 
\begin{equation}    \label{e:fibered1}
Z^-\underset{Z}\times \wt{Z}\, ,
\end{equation}
formed using $\pi_1:\wt{Z}\to Z$ and the fiber product
\begin{equation}     \label{e:fibered2}
\wt{Z}\underset{Z}\times Z^+ \, ,
\end{equation}
formed using $\pi_2:\wt{Z}\to Z$.
Note that both fiber products are spaces over $\BA^1$ (because $\wt{Z}$ is).

\sssec{Formulation of the result}  \label{sss:defining the 2 maps}
Consider the composition
\begin{equation}   \label{e:embedding1}
\BA^1\times Z^+\to\wt{Z^+}=\wt{Z^+}\underset{Z^+}\times Z^+ \to \wt{Z}\underset{Z}\times Z^+\, ,
\end{equation}
where the first arrow is the morphism \eqref{e:beta} for the space $Z^+$ and the second arrow comes from the morphism $p^+:Z^+\to Z$. Consider also the similar composition
\begin{equation}   \label{e:embedding2}
\BA^1\times Z^-\to\wt{Z^-}=Z^-\underset{Z^-}\times\wt{Z^-}  \to Z^-\underset{Z}\times\wt{Z} \, ,
\end{equation}
where the first arrow is the morphism \eqref{e:2beta} for the space $Z^-$. 
In \cite{DrGa1} we use the following result.

\begin{prop}   \label{p:2open embeddings}
The compositions \eqref{e:embedding1} and \eqref{e:embedding2} are open embeddings.
\end{prop}

Note that unlike the situation of \propref{p:Cartesian}, these embeddings are usually not closed.

\begin{rem}
By Propositions~\ref{p:2contracting} and \ref{p:dilating}, the maps $\BA^1\times Z^+\to\wt{Z^+}$ and 
$\BA^1\times Z^-\to\wt{Z^-}$ are isomorphisms, so \propref{p:2open embeddings} means that the morphisms
\[
\wt{Z^+}\to\wt{Z}\underset{Z}\times Z^+ ,\quad \wt{Z^-}\to Z^-\underset{Z}\times\wt{Z} 
\]
are open embeddings.
\end{rem}

\begin{rem}  \label{r:gamma is an iso}
Using \eqref{e:fiberprod}, it is easy to see that if every $\BG_m$-equivariant map $\BP^1\otimes_k\bar k\to Z\otimes_k\bar k$ is constant then the maps \eqref{e:embedding1} and \eqref{e:embedding2} are surjective.
In this case they are isomorphisms by \propref{p:2open embeddings}.
\end{rem}

\sssec{Plan}
We will interpret the fiber products \eqref{e:fibered1} and \eqref{e:fibered2}
as spaces of $\BG_m$-equivariant maps. More precisely, we will define
schemes  $\BXm$ and $\BXp$ over $\BA^1$ equipped with $\BG_m$-action, such that for any scheme $S$ over $\BA^1$ one has natural bijections
\begin{equation}   \label{e:bij+}
\Maps_{\BA^1}(S,\wt{Z}\underset{Z}\times Z^+ )\iso\gMaps (\BXp_S\, ,Z),  \quad\quad 
\BXp_S:=\BXp \underset{\BA^1}\times S
\end{equation}

\begin{equation}   \label{e:bij-}
\Maps_{\BA^1}(S,Z^-\underset{Z}\times \wt{Z} )\iso\gMaps (\BXm_S\, ,Z), \quad\quad 
\BXm_S:=\BXm \underset{\BA^1}\times S\, .
\end{equation}
Then we will give a simple description of $\BXpm$. We will see that after reformulating 
\propref{p:2open embeddings} in terms of $\BX^{\pm}$ it becomes almost obvious.

\sssec{Definition of $\BXpm$}
We will define $\BXpm$ so that the bijections \eqref{e:bij+}-\eqref{e:bij-} are tautological.

We have
\[
\Maps_{\BA^1}(S\, , \wt{Z} )=\gMaps (\BX_S\, ,Z)\, ,
\]
\begin{equation} \label{e:2bij+}
\Maps_{\BA^1}(S\, , \BA^1\times Z^+ )=\gMaps (S\times\BA^1,Z)\,,
\end{equation}
\begin{equation}  \label{e:2bij-}
\Maps_{\BA^1}(S\, , \BA^1\times Z^-)=\gMaps (S\times\BA^1_-\,\, ,Z)\,,
\end{equation}
\[
\Maps_{\BA^1}(S\, , \BA^1\times Z )=\gMaps (S\times\BG_m\,,Z)\, .
\]

Recall that the maps $\pi_1,\pi_2:\wt{Z}\to Z$  used in Subsect.~\ref{sss:the fiber products}
come from the two sections of the morphism $\BX\to\BA^1$ that are given by formula \eqref{e:two sections}; namely, $\pi_1$ corresponds to the section $t\mapsto (1,t)$ and $\pi_2$ to the section $t\mapsto (t,1)$. These two sections define two $\BG_m$-equivariant maps $\BA^1\times\BG_m\to\BX\,$, where the  $\BG_m$-action on $\BX$ is defined by \eqref{e:hyperbolic}.
Namely, the section $t\mapsto (t,1)$ defines the map
\begin{equation}   \label{e:+section}
\BA^1\times\BG_m\to\BX\, , \quad (t,\lambda)\mapsto (\lambda t,\lambda^{-1}) \, ,
\end{equation}
and the section $t\mapsto (1,t)$ defines the map
\begin{equation}   \label{e:-section}
\BA^1\times\BG_m\to\BX\, ,  \quad(t,\lambda)\mapsto(\lambda ,\lambda^{-1} t)  \, .
\end{equation}
Note that both maps are open embeddings.

\begin{defn}
(i) $\BXp$ is the push-out of the diagram of open embeddings 
\begin{equation}   \label{e:+diagram}
\BA^1\times\BA^1\hookleftarrow \BA^1\times\BG_m\hookrightarrow\BX\,
\end{equation} 
in which the right arrow is the map \eqref{e:+section}.

(ii) $\BXm$ is the push-out of the diagram of open embeddings \,
\begin{equation}   \label{e:-diagram}
\BA^1\times\BA^1_-\hookleftarrow \BA^1\times\BG_m\hookrightarrow\BX
\end{equation} 
in which the right arrow is the map \eqref{e:-section}.
\end{defn}

Both \eqref{e:+diagram} and \eqref{e:+diagram} are diagrams in the category of schemes over $\BA^1$
equipped with a $\BG_m$-action over $\BA^1$ (in the case of $\BA^1\times\BA^1$ the structure of scheme over $\BA^1$ is given by the \emph{first} projection $\BA^1\times\BA^1\to\BA^1$). So $\BXp$ and $\BXm$ are also in this 
category.\footnote{Moreover, one can define  an action of the torus $\BG_m^2$ on each of the diagrams
\eqref{e:+diagram}-\eqref{e:-diagram} so that they become diagrams in the category of toric varieties 
(a.k.a. toric embeddings); then $\BXp$ and $\BXm$ are also in this category. The above $\BG_m$-action is a part of the $\BG_m^2$-action.
}

The bijections \eqref{e:bij+}-\eqref{e:bij-} are clear.

\sssec{Description of $\BXpm$}   \label{sss:blow-up}
We claim that \emph{both schemes $\BXp$ and $\BXm$ are 
isomorphic to the blow-up of $\BA^2$ at a point.} Here is a more precise statement, whose verification is straightforward.

\begin{lem} \label{l:sigma+-}
(i) The morphisms
\[
\BA^1\times\BA^1\to\BA^1\times\BA^1, \quad \quad (t,\lambda )\mapsto (t,\lambda t)
\]
\[
\BX\to\BA^1\times\BA^1, \quad \quad (\tau_1\, ,\tau_2)\mapsto (\tau_1\tau_2\, ,\tau_1)
\]
are compatible via diagram \eqref{e:+diagram}. The corresponding morphism 
$\sigma^+:\BXp\to\BA^1\times\BA^1$ is a blow-up at the point $(0,0)\in\BA^1\times\BA^1$.

(ii) The morphisms
\[
\BA^1\times\BA^1_-\to\BA^1\times\BA^1_-\, , \quad \quad (t,\lambda )\mapsto (t,\lambda t^{-1})
\]
\[
\BX\to\BA^1\times\BA^1_-\, , \quad \quad (\tau_1\, ,\tau_2)\mapsto (\tau_1\tau_2\, ,\tau_2^{-1})
\]
are compatible via diagram \eqref{e:-diagram}. 
The corresponding morphism $\sigma^-:\BXm\to\BA^1\times\BA^1_-$
is a blow-up at the point $(0,\infty )\in\BA^1\times\BA^1_-\,$.

(iii) Both $\sigma^+$ and $\sigma^-$ are $\BG_m$-equivariant morphisms of schemes over $\BA^1$. \qed
\end{lem}

\sssec{The canonical morphisms $\BA^1\times Z^+\to \wt{Z}\times_Z Z^+$ and 
$\BA^1\times Z^-\to Z^-\times_Z\wt{Z}\,$}  \label{sss:sigma*}

For any scheme $S$ over $\BA^1$, the morphisms $\sigma^{\pm}$ from \lemref{l:sigma+-}(i-ii) induce morphisms
\[
\sigma^+_S :\BXp_S\to S\times\BA^1, \quad\quad \sigma^-_S :\BXm_S\to S\times\BA^1_- \, .
\]
By \eqref{e:bij+}-\eqref{e:bij-} and \eqref{e:2bij+}-\eqref{e:2bij-}, these morphisms induce canonical maps
\[
(\sigma^+_S)^* :\Maps_{\BA^1}(S,\BA^1\times Z^+)\to \Maps_{\BA^1}(S,\wt{Z}\times_Z Z^+), 
\]
\[
(\sigma^-_S)^*:\Maps_{\BA^1}(S,\BA^1\times Z^-)\to \Maps_{\BA^1}(S,Z^-\times_Z\wt{Z}) ,
\]
which are natural in $S$. These maps define canonical morphisms
\begin{equation}  \label{e:sigma+*}
(\sigma^+)^* :\BA^1\times Z^+\to \wt{Z}\times_Z Z^+, \quad\quad 
\end{equation}

\begin{equation}  \label{e:sigma-*}
(\sigma^-)^* :\BA^1\times Z^-\to Z^-\times_Z\wt{Z}\, .
\end{equation}

\begin{lem}
The morphisms \eqref{e:embedding1} and \eqref{e:embedding2} are equal, respectively, to \eqref{e:sigma+*} and \eqref{e:sigma-*}. 
\end{lem}

We skip the verification of the lemma, which is straightforward.

The lemma implies that \propref{p:2open embeddings} is equivalent to the following one.

\begin{prop}  \label{p:open embeddings}
The morphisms \eqref{e:sigma+*} and \eqref{e:sigma-*} are open embeddings.
\end{prop}

We will prove the part of \propref{p:open embeddings} about $(\sigma^-)^*$. We will use the following property of the morphism $\sigma^-:\BXm\to\BA^1\times\BA^1_-\,$.

\begin{lem}  \label{l:sigma is Stein}
Let $S$ be a spectrum of an Artinian local ring equipped with a morphism $S\to\BA^1$. Then the morphism
$\sigma^-_S:\BXm_S\to S\times\BA^1_-$ has the following property: the map
\[
\CO_{S\times\BA^1_-}\to (\sigma^-_S)_*\CO_{\BXm_S}
\]
is an isomorphism (here $(\sigma^-_S)_*$ denotes the naive
direct image rather than the derived one).
\end{lem}

\begin{proof}
If $S$ is a spectrum of a field the statement is clear from the explicit description of $\sigma^-$ given in 
\lemref{l:sigma+-}(ii). The case of a general Artinian local ring follows by devissage (one uses flatness of
$\BXm_S$ and $S\times\BA^1_-$ over $S$).
\end{proof}

\begin{rem} It is easy to prove \lemref{l:sigma is Stein} for \emph{any} scheme $S$ over $\BA^1$ and for the derived direct image 
 $R(\sigma^-_S)_*$ instead of the naive one. However, the above minimalistic formulation of 
 \lemref{l:sigma is Stein} will allow us to skip the proof of \propref{p:2open embedding}(i) 
 (because it is identical to that of \propref{p:open embeddings}).
 \end{rem}
 
We will also use the following general lemma, which is proved in Appendix~\ref{s:very general}.

\begin{lem}  \label{l:postponed}
Let $A$, $B$, $Z$ be algebraic $k$-spaces and $f:A\to B$ a surjective morphism with $f_*\CO_A=\CO_B$
(here $\CO_A$, $\CO_B$ are sheaves on the etale sites $A_{\ET}$, $B_{\ET}$ and $f_*$ is understood in the non-derived sense). Then

(i) the map $\Maps(B , Z)\to\Maps(A ,Z)$ induced by $f$ is injective;

(ii) if $B_0\subset B$ is a closed subspace containing $B_{\red}$ and $A_0=f^{-1}(B_0)$ then the diagram
\[
\CD
\Maps(B\, , Z) @>{}>>  \Maps(A\, ,Z) \\
@V{}VV   @V{}VV   \\ 
\Maps(B_0\, , Z) @>{}>>  \Maps(A_0\, ,Z)
\endCD
\]
induced by $f$ is Cartesian. 
\end{lem}

Now let us prove the part of \propref{p:open embeddings} about the morphism
$(\sigma^-)^*$.  To prove that $(\sigma^-)^*$ is an open embedding, it suffices to 
show that it is etale and induces an injective map of field-valued points. This amounts to checking the following:

(a) if $S$ is a spectrum of a field equipped with a morphism $S\to\BA^1$ then the map
\[
(\sigma^-_S)^*:\gMaps(S\times\BA^1_-\, , Z)\to\gMaps(\BXm_S\, ,Z)
\]
is injective;

(b) let $S$ be a spectrum of an Artinian local ring equipped with a morphism $S\to\BA^1$ (so $S_{\red}$ is a spectrum of a field), then the diagram
\[
\CD
\gMaps(S\times\BA^1_-\, , Z) @>{(\sigma^-_S)^*}>>  \gMaps(\BXm_S\, ,Z) \\
@V{}VV   @V{}VV   \\ 
\gMaps(S_{\red}\times\BA^1_-\, , Z) @>{(\sigma^-_{S_{\red}})^*}>>  \gMaps(\BXm_{S_{\red}}\, ,Z)
\endCD
\]
is Cartesian.

To prove this, it suffices to apply \lemref{l:postponed} for 
$A=\BXm_S\,$, $B=S\times\BA^1$, $f=\sigma^-_S$ and also for
$A=\BG_m\times\BXm_S\,$, $B=\BG_m\times S\times\BA^1$, $f=\id_{\BG_m}\times\sigma^-_S\,$.
\lemref{l:postponed} is applicable by \lemref{l:sigma is Stein}.

\ssec{The fiber product $\wt{Z}\times_Z\ldots\times_Z\wt{Z}\;$}   \label{ss:field theory}
The material of this subsection is used in the proof of \propref{p:Phi_Z}. We also think it is interesting on its own.

\sssec{Plan}
In Subsect.~\ref{ss:fiber products my} we described the fiber products 
$Z^-\underset{Z}\times\wt{Z}$ and $\wt{Z}\underset{Z}\times Z^+$ and constructed open embeddings
\begin{equation}  \label{e:the embeddings}
 \BA^1\times Z^-\mono Z^-\underset{Z}\times\wt{Z}\, , \quad\quad
\BA^1\times Z^+\mono\wt{Z}\underset{Z}\times Z^+ .
\end{equation}

Similarly to the above fiber products, one defines $\wt{Z}\underset{Z}\times\wt{Z}$ and, more generally,
the $n$-fold fiber product
\begin{equation}   \label{e:tilde Zn}
\wt{Z}_n:=\wt{Z}\underset{Z}\times\ldots\underset{Z}\times\wt{Z}
\end{equation}
using the projections $\pi_1,\pi_2:\wt{Z}\to Z$. 

We will describe $\wt{Z}_n$ as a space of maps
and construct an open embedding
\begin{equation}   \label{e:open_future}
\wt{Z}\underset{\BA^1}\times\BA^n\mono\tilde Z_n\, ,
\end{equation}
which is an isomorphism if $Z$ is affine.
The  morphism $\BA^n\to\BA^1$ implicit in formula \eqref{e:open_future} is the multiplication map
\begin{equation}   \label{e:mult_n}
(t_1,\ldots ,t_n)\mapsto t_1\cdot\ldots\cdot t_n \, .
\end{equation}

It will be clear that the embeddings \eqref{e:the embeddings} can be obtained by base change from the embedding \eqref{e:open_future} for $n=2$.

The strategy will be similar to the one used in Subsect.~\ref{ss:fiber products my}. The role of the blow-up of 
$\BA^2$ (see Subsect.~\ref{sss:blow-up}) will be played by a certain ``very small" resolution of singularities of the 
scheme
\begin{equation}   \label{e:XAn}
\BX_{\BA^n}:=\BX\underset{\BA^1}\times\BA^n \, ;
\end{equation}
here the fiber product is defined using the map \eqref{e:mult_n}, so it is, in fact, the hypersurface
\[ 
t_1\cdot\ldots\cdot t_n=uv\, .
\] 
The above-mentioned small resolution is well known for $n=2$.

\sssec{$\wt{Z}_n$ as a space of maps}   \label{sss:as a map space}
Let $\CC_n$ denote the category of spaces\footnote{Recall that ``space" just means ``functor with the sheaf property".} over $\BA^n$ equipped with a $\BG_m$-action over $\BA^n$.
For instance, $\BX$ and $\wt{Z}$ are objects of $\CC_1\,$, and the space $\wt{Z}_n$ defined by 
\eqref{e:tilde Zn} is an object of $\CC_n$ (because $\wt{Z}\in\CC_1$).

Now we will define a scheme $\BX_n\in\CC_n$ such that
for any scheme $S$ over $\BA^n$ one has
\begin{equation}    \label{e:maps to tilde Zn}
\Maps_{\BA^n}(S,\wt{Z}_n)=\gMaps ((\BX_n)_S\, , Z)  \, ,
\end{equation}
where $(\BX_n)_S:=\BX_n\underset{\BA^n}\times S\,$. (For instance, if $n=1$ then $\BX_n=\BX\,$.)

First, set
$$U_r:=\BA^{r-1}\times\BX\times\BA^{n-r}, \quad \quad 1\le r\le n\, .$$
Note that $U_r\in\CC_n$ because $\BX\in\CC_1\,$. In $\CC_1$ we have two open embeddings 
$\BG_m\times\BA^1\to\BX$
defined by \eqref{e:+section}-\eqref{e:-section}. 
Multiplying them by $\BA^{r-1}$ on the left and $\BA^{n-r}$ on the right one gets two open embeddings 
$\BG_m\times\BA^n\mono U_r$ in the category  $\CC_n\,$. Let $\alpha_r:\BG_m\times\BA^n\mono U_r$ be the embedding corresponding to \eqref{e:+section} and $\beta_r:\BG_m\times\BA^n\mono U_r$ the one
corresponding to \eqref{e:-section}.

Now define $\BX_n\in\CC_n$ to be the colimit (a.k.a. inductive limit) of the diagram 
\begin{equation}   \label{e:zigzag}
\xymatrix{
U_1&  &U_2& &\ldots&\\
  &\BG_m\times\BA^n\ar@{_{(}->}[lu]_{\alpha_1}\ar@{^{(}->}[ru]^{\beta_2}&&
  \BG_m\times\BA^n\ar@{_{(}->}[lu]_{\alpha_2}\ar@{^{(}->}[ru]^{\beta_3}  &&
  }
\end{equation}
Then the bijection \eqref{e:maps to tilde Zn} is tautological. The next lemma says that the space $\BX_n$ is, in fact, a nice scheme.

\begin{lem}   \label{l:smooth & flat}
(i) The canonical morphisms 
$$U_r=\BA^{r-1}\times\BX\times\BA^{n-r} \to \BX_n\, ,\quad\quad 1\le r\le n$$  
are  open embeddings, and their images cover $\BX_n\,$. 

(ii) $\BX_n$ is a smooth scheme over $k$ of dimension $n+1$, which is flat over $\BA^n$.
\end{lem}

\begin{proof}
Statement (i) is proved by induction. Statement (ii) follows.
\end{proof}

\begin{rem}   \label{r:1curves}
It is clear that the fiber of $\BX_n$ over each field-valued point of $\BA^n$ is a curve. (It is obtained by gluing hyperbolas. Such gluing is non-tautological only if some of these hyperbolas are degenerate.)
\end{rem}

\begin{rem}  \label{r:2curves}
Here is a more precise version of the previous remark. For $m\ge 0$ let $C_m$ denote the following curve:
take $m+1$ copies of $\BP^1$, denoted by $(\BP^1)_i\, $, $0\le i\le m$; then for all $i<m$ glue 
$0\in (\BP^1)_i$ with $\infty\in(\BP^1)_{i+1}$ and finally, remove $\infty\in (\BP^1)_0$ and 
$0\in (\BP^1)_m\,$. 
It is easy to see that the fiber of $\BX_n$ over each field-valued point of $\BA^n$ is isomorphic to $C_m$ for some $m$, $0\le m\le n\,$.
\end{rem}

\sssec{The locally closed embedding $\BX_n\mono\BA^n\times (\BP^1)^{n+1}\,$}    \label{sss:q-projective}
We will first construct a quasi-projective scheme $\BX'_n\in\CC_n\,$; more precisely, $\BX'_n$ will be a locally closed subscheme of the product $\BA^n\times (\BP^1)^{n+1}$. Then we will construct a $\CC_n$-isomorphism
$\BX_n\iso\BX'_n\,$.

Points of $\BP^1$ will be denoted by $(p:q)$. We equip $\BP^1$ with the usual action of $\BG_m\,$, i.e.,
$\lambda\in\BG_m$ takes $(p:q)$ to $(\lambda p:q)$.

\medskip

\noindent {\bf Convention.} Let $\xi ,\xi'\in\BP^1$, $\xi =(p:q)$, $\xi' =(p':q')$, $t\in\BA^1$. Then 
\begin{equation}   \label{e:convention}
\mbox{we write }\quad t\cdot\xi=\xi' \quad\mbox{as a shorthand for }\quad  tpq'=p'q\, .
\end{equation}
(Thus $0\cdot\infty=\xi'$ for \emph{any} $\xi'\in\BP^1$.)

\medskip

We equip $(\BP^1)^{n+1}$ with the diagonal action of $\BG_m\,$. So $\BA^n\times(\BP^1)^{n+1}$ is an object of $\CC_n\,$. Points of $\BA^n$ will be denoted by $(t_1,\ldots ,t_n)$. Points of $(\BP^1)^{n+1}$ will be denoted by $(\xi_0\, ,\ldots ,\xi_n )$, where $\xi_i\in\BP^1$.

\begin{defn}
$\BX'_n\subset\BA^n\times (\BP^1)^{n+1}$ is the locally closed subscheme defined by the  inequalities
\begin{equation}  \label{e:inequalities}
\xi_0\ne \infty\, ,\quad\quad \xi_n\ne 0
\end{equation}
and the equations
\begin{equation}  \label{e:the_equations}
\xi_{i-1}=t_i\cdot\xi_i\, ,\quad\quad 1\le i\le n\, ;
\end{equation}
here the equations are understood according to Convention \eqref{e:convention}.
\end{defn}

The subscheme $\BX'_n\subset\BA^n\times (\BP^1)^{n+1}$ is $\BG_m$-stable, so  $\BX'_n\in\CC_n\,$.

Let $U'_r\subset\BX'_n$ denote the $\BG_m$-stable open subscheme defined by the inequalities
\begin{equation}  \label{e:2inequalities}
\xi_{r-1}\ne\infty\, ,\quad\xi_r\ne 0\,.
\end{equation}
Note that the inequalities \eqref{e:inequalities} follow from \eqref{e:2inequalities} and \eqref{e:the_equations}.

\begin{lem}
(i) The open subschemes $U'_r$ cover $\BX_n\,$.

(ii) If $r_1\le r_2\le r_3$ then $U'_{r_1}\cap U'_{r_3}\subset U'_{r_2}\:$.
\end{lem}

\begin{proof}
(i) Let $(\xi_0\, ,\ldots ,\xi_n )\in\BX'$. Let $r$ be the minimal number such that $\xi_r\ne 0\,$. Then 
$\xi_{r-1}=0\ne\infty$, so $(\xi_0\, ,\ldots ,\xi_n )\in U'_r\,$.

(ii) Let $(\xi_0\, ,\ldots ,\xi_n )\in U'_{r_1}\cap U'_{r_3}\,$. Since $\xi_{r_3-1}\ne\infty$ and  $\xi_{r_1}\ne 0$
the equations \eqref{e:the_equations} imply that $\xi_{r_2-1}\ne\infty$ and  $\xi_{r_2}\ne 0\,$.
\end{proof}

\begin{cor}   \label{c: colimit too}
$\BX'_n$ is the colimit of the diagram 
\begin{equation}   \label{e:2zigzag}
\xymatrix{
U'_1&  &U'_2& &\ldots&\\
  &U'_1\cap U'_2\ar@{_{(}->}[lu]_{}\ar@{^{(}->}[ru]^{}&&
  U'_2\cap U'_3\ar@{_{(}->}[lu]_{}\ar@{^{(}->}[ru]^{}  &&
  }
\end{equation}
\end{cor}

Now we will construct a $\CC_n$-isomorphism between diagrams \eqref{e:zigzag} and  \eqref{e:2zigzag}.

Recall that the scheme $U_r$ from diagram \eqref{e:zigzag} equals $\BA^{r-1}\times\BX\times\BA^{n-r}$, and the coordinates on $\BX$ are denoted by $\tau_1,\tau_2\,$.

\begin{lem}   \label{l:identifying diagarms}
(i) Formulas $\tau_1=\xi_{r-1}\,$, $\tau_2=\xi_r^{-1}$ define a $\CC_n$-isomorphism $U'_r\iso U_r\,$. Its inverse is given by
\[
\xi_i=t_{i+1}\cdot\ldots\cdot t_{r-1}\cdot\tau_1 \quad\mbox{ for } i<r\, , \quad\quad
\xi_i=(\tau_2\cdot t_{r+1}\cdot\ldots\cdot t_i)^{-1}\quad\mbox{ for } i\ge r\, .
\]

(ii) There exists an isomorphism between diagrams \eqref{e:2zigzag} and  \eqref{e:zigzag} inducing the above isomorphism $U'_r\iso U_r$ for each $r\in\{1,\ldots, n\}$.
\end{lem}

The proof is straightforward. Let us just say that the composition 
$$U'_r\cap U'_{r+1}\iso\BG_m\times\BA^n\to\BG_m$$
is given by $\xi_r$ (note that the values of $\xi_r$ on $U'_r\cap U'_{r+1}$ are in $\BP^1-\{ 0,\infty\}=\BG_m$).

\medskip

Finally, \corref{c: colimit too} tells us that the isomorphism between diagrams \eqref{e:2zigzag} and  \eqref{e:zigzag} constructed in \lemref{l:identifying diagarms} induces a  $\CC_n$-isomorphism 
$\BX'_n\iso\BX_n\,$. We will always identify $\BX'_n$ with $\BX_n$ using this isomorphism.

\sssec{The map $\BX_n\to\BX_{\BA^n}\,$} Recall that $\BX_{\BA^n}:=\BX\underset{\BA^1}\times\BA^n$; equivalently, $\BX_{\BA^n}$ is the hypersurface 
\begin{equation}   \label{e:hypersurface}
t_1\cdot\ldots\cdot t_n=uv\, . 
\end{equation}
The equations \eqref{e:the_equations} imply that $\xi_0\cdot\xi_n^{-1}=t_1\cdot\ldots\cdot t_n\,$, so we have 
a morphism 
\begin{equation}    \label{e:Xn to XAn}
\BX_n=\BX'_n\to\BX_{\BA^n}
\end{equation}
defined by
\[
u=\xi_0\, ,\quad\quad v=\xi_n^{-1}\, .
\]
\begin{lem}
The morphism \eqref{e:Xn to XAn} is projective and small.
\end{lem}

\begin{proof}
The morphism \eqref{e:Xn to XAn} is projective because it is a composition 
\[
\BX'_n\mono\BX_{\BA^n}\times (\BP^1)^{n-1}\to \BX_{\BA^n}
\]
in which the first arrow is a closed embedding and the second one is the projection. By \remref{r:2curves},
the fibers of the morphism \eqref{e:Xn to XAn} have dimension $\le 1$ (the fibers that have more than one point are chains of projective lines). Finally, it is easy to check that the map \eqref{e:Xn to XAn} is an isomorphism over $\BX_{\BA^n}-F$, where 
$F\subset\BX_{\BA^n}$ is a closed subset of codimension 3; namely, a point 
$(t_1\,,\ldots ,t_n\,, u,v)\in \BX_{\BA^n}$ is in $F$ if and only if $u=v=0$ and $t_i=0$ for more than one~$i\,$.
\end{proof}

\sssec{The open embedding $\wt{Z}\times_{\BA^1}\BA^n\mono\tilde Z_n\,$}
By \eqref{e:maps to tilde Zn}, the map \eqref{e:Xn to XAn} induces a morphism
\begin{equation} \label{e:2open embedding}
\wt{Z}\underset{\BA^1}\times\BA^n\to\tilde Z_n\, . 
\end{equation}
\begin{prop}    \label{p:2open embedding}
(i) This morphism is an open embedding. 

(ii) If every $\BG_m$-equivariant map $\BP^1\otimes_k\bar k\to Z\otimes_k\bar k$ is constant then the map \eqref{e:2open embedding} is an isomorphism.
\end{prop}


\begin{proof}
Statement (i) is proved just as \propref{p:open embeddings} (to prove an analog of \lemref{l:sigma is Stein}, use \remref{r:2curves} and the flatness statement from \lemref{l:smooth & flat}(ii)).

Statement (ii) is a consequence of (i) and the following corollary of \remref{r:2curves}: any fiber of the morphism  \eqref{e:Xn to XAn} is either a point or a chain of projective lines each of which is equipped with the standard $\BG_m$-action.
\end{proof}

\sssec{Toric action}
Recall that $\BX_{\BA^n}$ is the variety of solutions to the equation  \eqref{e:hypersurface}.
Let $T\subset\BX_{\BA^n}$ denote the set of those solutions all of whose coordinates are nonzero. This is a group with respect to multiplication. The torus $T$ acts on $\BX_{\BA^n}$ by multiplication; in fact,
$\BX_{\BA^n}$ is a \emph{toric variety} with respect to $T$. There is a unique structure of toric variety on 
$\BX_n$ (with the same torus $T$) such that the map \eqref{e:Xn to XAn} is a morphism of toric varieties.
Therefore one can describe $\BX_n$, $\BX_{\BA^n}$, and the map \eqref{e:Xn to XAn} using the language of fans, see \cite[Ch.~I]{KKMS}.

Note that the action of $\BG_m$ on $\BX_n$ and $\BX_{\BA^n}$ considered above is a part of the $T$-action.

\sssec{Relation with the anti-action from Subsect.~\ref{sss:anti-action}} \label{sss:relation anti-action}
The morphism \eqref{e:2open embedding} can be expressed in terms of the anti-action from 
Subsect.~\ref{sss:anti-action}. Let us explain this for $n=2$.

Given a $k$-scheme $S$ and morphisms $t_1,t_2:S\to\BA^1$, we have a commutative diagram
\begin{equation}   \label{e:anti-action diagram}
\CD
\wt{Z}_{t_1 t_2} @>{\phi_{t_1,1, t_2}}>>  \wt{Z}_{t_2} \\
@V{\phi_{1,t_2, t_1}}VV   @VV{\phi_{1,t_2, 1}}V   \\ 
\wt{Z}_{t_1} @>{\phi_{t_1,1,1}}>>\wt{Z}_1@=Z
\endCD
\end{equation}
whose arrows are given by the anti-action of $\BA^2$, see formula~\eqref{e:anti-action}. It is easy to check that the morphism $\phi_{1,t_2, 1}:\wt{Z}_{t_2}\to\wt{Z}_1=Z$ comes from the morphism $\pi_1:\wt{Z}\to Z$ defined in Subsect.~\ref{sss:tilde p} and the morphism $\phi_{t_1,1,1}:\wt{Z}_{t_1}\to\wt{Z}_1=Z$ comes from $\pi_2\,$. So diagram \eqref{e:anti-action diagram} defines a morphism 
$\wt{Z}_{t_1 t_2}\to \wt{Z}\underset{Z}\times\wt{Z}$.  As $t_1$ and $t_2$ vary, we get a morphism
\begin{equation}   \label{e:open emb for n=2}
\wt{Z}\underset{\BA^1}\times\BA^2\to\wt{Z}\underset{Z}\times\wt{Z}. 
\end{equation}
It is straightforward to check that it is equal to the morphism \eqref{e:2open embedding} for $n=2$.

\sssec{Remark} \label{sss:property of anti-action} 
By virtue of Subsect.~\ref{sss:relation anti-action}, one can interpret \propref{p:2open embedding} as a property of the anti-action from Subsect.~\ref{sss:anti-action}. Since the map  \eqref{e:2open embedding} involves a number $n$, we have, in fact, a sequence of properties for $n=2,3,4,...\; $. However, it is easy to see that the property for $n=2$ implies the rest.

\section{Proof of Theorem~\ref{t:attractors}} \label{s:attractors}

\ssec{Plan}

Let $Z$ be an algebraic $k$-space of finite type equipped with a $\BG_m$-action.
We have to prove that $Z^+$ is an algebraic space and that
%
%
%
$q^+:Z^+\to Z^0$ is an affine morphism of finite type.
To this end, we will decompose the morphism $q^+:Z^+\to Z^0$ as 
\[
Z^+\to\Zp\to Z^0,
\]
where $\Zp$ is defined in Subsect.~\ref{ss:Zp} below. Then we will prove that the morphism $Z^+\to\Zp$ is, in fact, an isomorphism and the morphism $\Zp\to Z^0$ is an affine morphism of finite type.

\ssec{The space $\Zp$} \label{ss:Zp}
For $n\in\BZ_+$ let $(\BA^1)_n\subset\BA^1$ denote the $n$-th infinitesimal neighborhood of $0\in\BA^1$, i.e., 
$(\BA^1)_n:=\Spec k[t]/(t^{n+1})$. 
Set $Z^+_n:=\GMaps ((\BA^1)_n\, , Z)$. Note that $Z^+_0=Z^0$. The spaces $Z^+_n$ form a projective system. 

\begin{defn}
$\Zp:= \limfromn Z^+_n$. Equivalently, $\Zp =\GMaps (\widehat{\BA}^1,Z)$, where $\widehat{\BA}^1$ is the formal completion of $\BA^1$ at $0\in\BA^1$, i.e., $\BA^1:= \limton (\BA^1)_n\,$.
\end{defn}

The multiplicative monoid $\BA^1$ acts on itself by multiplcation, and this action preserves the subschemes 
$(\BA^1)_n\subset \BA^1$. So $\BA^1$ acts on the spaces $Z^+_n$ and $\Zp$.

\ssec{A theorem which implies Theorem~\ref{t:attractors}}
The embeddings 
$$\Spec k=\{ 0\}\mono\widehat{\BA}^1\mono\BA^1$$ 
induce morphisms
\[
\GMaps (\BA^1,Z)\to\GMaps (\widehat{\BA}^1,Z)\to \GMaps (\Spec k,Z)
\]
or equivalently,
\begin{equation}   \label{e:factorization}
Z^+\to\Zp\to Z^0.
\end{equation}

The composition in \eqref{e:factorization} equals $q^+:Z^+\to\Zp$. So
Theorem~\ref{t:attractors} follows from the next one.

\begin{thm}   \label{t:2attractors}
(i) The  morphism $Z^+\to\Zp$ is an isomorphism.

(ii) $\Zp$ is an algebraic space. Moreover,  the  morphism $\Zp\to Z^0$ is affine and of finite type.
\end{thm}

The easier statement (ii) will be proved in the next subsection. Statement (i) of the theorem will be proved in 
Subsections~\ref{ss:2attractors-proof}-\ref{ss:2attractors-2proof}.

\ssec{Proof of Theorem~\ref{t:2attractors}(ii)} \label{ss:A+}
We will first construct a finitely generated $\BZ_+$-graded quasi-coherent $\CO_{Z^0}$-algebra $\CA^+$
(see Definition~\ref{d:A+}). Then we will construct an isomorphism $Z^+\iso\Spec\CA^+$ of spaces over $Z^0$. Thus we will get an explicit description of $Z^+$ in the spirit of Subsect.~\ref{sss:attractors-affine}.

\medskip

Let $J_n$ denote the sheaf of $n$-jets of functions on $Z$. In particular, 
\begin{equation}   \label{e:J1}
J_0=\CO_{Z}\, , \quad J_1=\CO_{Z}\oplus\Omega^1_Z\, . 
\end{equation}
Let $J_n^0$ denote the pullback of $J_n$ to $Z^0$.

Each $J_n$ is a coherent $\CO_{Z}$-algebra.\footnote{More precisely, $J_n$ has two $\CO_{Z}$-algebra structures (the ``left" one and the ``right" one). We make a choice between left and right once and for all.} So $J_n^0$ is a coherent $\CO_{Z^0}$-algebra.
In addition, the $\CO_{Z^0}$-algebra is $\BZ$-graded: the grading corresponds to the $\BG_m$-action on
$J_n^0\,$. The epimorphism 
$$J_n^0\epi J_0^0=\CO_{Z^0}$$
will be called \emph{augmentation}. 
If $0\le m\le n$ then $J_m^0$ identifies with the quotient of $J_n^0$ by the $(m+1)$-th power of the augmentation ideal $\Ker (J_n^0\epi\CO_{Z^0})$.

\medskip

Let $\CA^+_n$ denote the quotient of $J_n^0$ by the ideal generated by the degree 0 component of
$\Ker (J_n^0\epi\CO_{Z^0})$ and by the components of negative degrees of $J_n^0$. Clearly $\CA^+_n$
is a $\BZ_+$-graded coherent $\CO_{Z^0}$-algebra whose degree 0 component equals $\CO_{Z^0}\,$.

\begin{lem} \label{l:truncating}
If $0\le m\le n$ then $\CA^+_m$ identifies with the quotient of $\CA^+_n$ by the $(m+1)$-th power of the augmentation ideal $\Ker (\CA^+_n\epi\CO_{Z^0})$. 
\end{lem}

\begin{proof}
Follows from a similar property of the algebras $J_n^0\,$.
\end{proof}

\lemref{l:truncating} implies that $\Ker (\CA^+_n\epi\CA^+_m)$ is concentrated in degrees $>m$; in other words, the degree $m$ component of $\CA^+_n$ does not depend on $n$ if $n\ge m$.

\begin{defn}   \label{d:A+}
$\CA^+$ is the $\BZ_+$-graded quasi-coherent $\CO_{Z^0}$-algebra whose degree $m$ component is the  
degree $m$ component of $\CA^+_n$, where $n\ge m$. In other words, $\CA^+$ is the projective limit of 
$\CA^+_n$ in the category of $\BZ_+$-graded quasi-coherent $\CO_{Z^0}$-algebras.
\end{defn}

\begin{lem}   \label{l:finite-generation}
(i) The $\CO_{Z^0}$-algebra $\CA^+$ is finitely generated.

(ii) Suppose that the pullback of $\Omega^1_Z$ to $Z^0$ is concentrated in degrees $\le n$ with respect to the $\BZ$-grading corresponding to the $\BG_m$-action. Then the $\CO_{Z^0}$-algebra $\CA^+$ is generated by its graded components of degrees $\le n$.
\end{lem}

\begin{proof}
Statement~(i) follows from \lemref{l:truncating} and the fact that $\CA^+_1$ is coherent.
Statement~(ii) follows from \lemref{l:truncating} and the following description of $\CA^+_1$, which is an immediate corollary of \eqref{e:J1}: $\CA^+_1=\CO_{Z^0}\oplus\CM$, where $\CM$ is the strictly positive part of the pullback of $\Omega^1_Z$ to $Z^0$.
\end{proof}

Theorem~\ref{t:attractors}(ii) immediately follows from \lemref{l:finite-generation}(i) and the next proposition.

\begin{prop}
$\Zp$  is canonically isomorphic to $\Spec\CA^+$ as a space over $Z^0$.
\end{prop}

\begin{proof}
Let $S=\Spec R$ be an affine $k$-scheme. Fix a morphism $\varphi :S\to Z^0$. Set 
$$A^+_R:=H^0(S,\varphi^*\CA^+).$$
We have to construct a canonical bijection
\begin{equation}   \label{e:biject}
\Maps_{Z^0}(S,\Zp)\iso\Hom_R(A^+_R\, ,R),
\end{equation}
where $\Maps_{Z^0}$ stands for the set of morphisms of spaces over $Z^0$ and 
$\Hom_R$ stands for the set of $R$-algebra homomorphisms.

Define $\Phi :S\mono Z\times S$ by $\Phi :=(\varphi,\id_S)$. By definition, elements of 
$\Maps_{Z^0}(S,\Zp)$ correspond to $\BG_m$-equivaraint $S$-morphisms 
$f:\widehat{\BA}^1\times S\to Z\times S$ whose restriction to $\{ 0\}\times S$ equals $\Phi :S\to Z\times S$. 
Such $f$ is the same as a sequence of $S$-morphisms $f_n:\BA^1_n\times S\to Z\times S$ such that the
restriction of $f_n$ to  $\{ 0\}\times S$ equals $\Phi :S\to Z\times S$ and the restriction of $f_n$ to 
$\BA^1_{n-1}\times S$ equals $f_{n-1}\,$. Now we need the following standard fact.

\begin{lem}
The $n$-th infinitesimal neighborhood of the subspace $\Phi(S)\subset Z\times S$ equals 
$\Spec J_n^R$, where  $J_n^R:=H^0(S,\varphi^*J_n^0)$.  In other words, $\Spec J_n^R$ represents the functor that to a scheme $T$ associates the set of morphisms $g:T\to Z\times S$ such that the restriction of $g$ to some closed subscheme $T'\subset T$ with 
  $\CI_{T'}^{n+1}=0$ factors through $\Phi :S\mono Z\times S$ (here $\CI_{T'}\subset\CO_T$ is the ideal of $T'$).
\qed 
\end{lem}


The lemma shows that an $S$-morphism $\BA^1_n\times S\to Z\times S$ whose
restriction to  $\{ 0\}\times S$ equals $\Phi :S\to Z\times S$ is the same as an 
$S$-morphism $\BA^1_n\times S\to \Spec J_n^R$ whose restriction to  $\{ 0\}\times S$ is the canonical embedding $S\mono \Spec J_n^R$. Thus elements of the l.h.s. of \eqref{e:biject} correspond to homomorphisms of augmented topological algebras
\[
\limfromn J_n^R\to R[[t]]
\]
compatible with the $\BZ$-gradings (here $t\in R[[t]]$ has degree 1 and the augmentation $R[[t]]\to R$ is the ``constant term" map). Such a homomorphism has to kill all elements of negative degrees and all degree 0 elements of the augmentation ideal of $J_n^R$. Thus elements of the l.h.s. of \eqref{e:biject} correspond to
graded $R$-algebra homomorphisms $A^+_R\to R[t]$. 

Finally, graded $R$-algebra homomorphisms $A^+_R\to R[t]$ are in bijection with elements of the r.h.s of 
\eqref{e:biject}: to a graded homomorphism $A^+_R\to R[t]$ one associates its composition with 
$\ev_1:R[t]\to R$, where $\ev_1$ is evaluation at $t=1$.
\end{proof}

\ssec{Proof of Theorem~\ref{t:2attractors}(i) modulo \lemref{l:complete_local}}  \label{ss:2attractors-proof}
Our goal is to prove that the morphism $Z^+\to\Zp$ is an isomorphism.

\begin{lem}   \label{l:mono}
The morphism $Z^+\to\Zp$ is a monomorphism.
\end{lem}

\begin{proof}
We have to prove the injectivity of the map
\[
\gMaps (S\times\BA^1,Z)\to\gMaps (S,\Zp)
\]
for any $k$-scheme $S$. Since $\Zp$ has finite type over $k$ we can assume that $S$ is Noetherian (and moreover, has finite type over $k$).

Let $f_1,f_2:S\times\BA^1\to Z$ be $\BG_m$-equivariant morphisms having the same restriction to the formal neighborhood of $S\times\{ 0\}\subset S\times\BA^1$. We have to prove that $f_1=f_2$.
Let $E$ denote the equalizer of $f_1,f_2$, i.e., the preimage of the diagonal with respect to 
$(f_1,f_2):S\times\BA^1\to Z\times Z$. Clearly $E$ is a scheme of finite presentation over $S$ equipped with an 
$\BA^1$-action and an $\BA^1$-equivariant monomorphism $\nu :E\mono S\times\BA^1$. Moreover, the \subscheme\footnote{The quotation marks are due to the fact that $\nu$ is not necessarily a locally closed embedding. Of course, $\nu$ is a locally closed embedding if $Z$ is separated 
or if $Z$ is a scheme.}
 $E\subset S\times\BA^1$ contains the formal neighborhood of
$S\times\{ 0\}\subset S\times\BA^1$. So $\nu$ is etale at $S\times\{ 0\}\subset E$. Let $E'$ be the maximal open subscheme of $E$ such that $\nu|_{E'}$ is etale. Then $\nu|_{E'}$ is an open embedding. So
$E'$ is an open subsheme of $S\times\BA^1$ containing $S\times\{ 0\}$ and stable with respect to the  $\BA^1$-action. Therefore $E'=S\times\BA^1$, $E=S\times\BA^1$, and $f_1=f_2$.
\end{proof}

\begin{rem}   \label{r:commutation_colimits}
It is clear that the space $Z^+$ is locally of finite presentation (i.e., the corresponding functor 
\{$k$-algebras\}$\to$\{sets\} commutes with filtering inductive limits).
\end{rem}

\begin{lem}  \label{l:complete_local}
Let $R$ be a complete local Noetherian $k$-algebra. Then the map $Z^+(R)\to\Zp (R)$ is bijective.
\end{lem}

Let us assume this lemma for now; it will be proved in Subsect.~\ref{ss:2attractors-2proof}. 

\begin{proof}[Proof of Theorem~\ref{t:2attractors}(i)]
We have to show that the morphism $Z^+\to\Zp$ is an isomorphism. By \lemref{l:mono}, it is a monomorphism, so it remains to show that for any point\footnote{Here ``point" is understood in the sense of \cite[Ch.2, Definition 6.1]{Kn}.} $z\in \Zp$ the morphism $Z^+\to\Zp$ admits a section over some etale neighborhood of $z$. By \remref{r:commutation_colimits}, we can replace ``etale neighborhood" by ``Henselization". By Artin approximation \cite[Theorem~1.10]{Ar} and \remref{r:commutation_colimits}, one can replace ``Henselization" by ``spectrum of the completed local ring of~$z\,$". It remains to 
use \lemref{l:complete_local}
\end{proof}

\begin{rem}
If $Z$ is separated the proofs of  Theorems~\ref{t:attractors} and \ref{t:2attractors}(i) can be simplified (in particular, Artin approximation is unnecessary). Namely, if $Z$ is separated it is easy to prove directly that the canonical morphism $Z^+\to\Zp\times Z$ is a closed embedding.
Combining this with Theorem~\ref{t:2attractors}(ii), one immediately gets Theorem~\ref{t:attractors}, and Theorem~\ref{t:2attractors}(i) follows from \lemref{l:complete_local} in the particular case of \emph{Artinian} local $k$-algebras.
\end{rem}

\ssec{A descent theorem of Moret-Bailly}  \label{ss:descent}
To prove \lemref{l:complete_local}, we need the following result from \cite{MB}.

\begin{thm}    \label{t:MB}
Let $S$ be a $k$-scheme and $Y\subset S$ a closed subscheme whose defining ideal in $\CO_S$ is finitely generated. Let $S'$ be a scheme flat and affine over $S$ such that the map $S'\times_SY\to Y$ is an isomorphism. Set $U:=S-Y$, $U':=U\times_SS'$.
Then for any algebraic $k$-space $Z$ the map
\[
Z(S)\iso Z(S')\underset{Z(U')}\times Z(U)  
\]
is bijective.
\end{thm}

This is Theorem 1.2 from \cite{MB}. If $Z$ is a scheme the proof is easy (see  \cite{MB}); more generally, there is an easy proof if the diagonal map $Z\to Z\times Z$ is a locally closed embedding. In the general case, the proof  from \cite{MB} uses Proposition 4.2 from \cite{FR}, which says that in the situation of Theorem~\ref{t:MB} the functor
\[
\QC (S)\iso \QC (S')\underset{\QC (U')}\times \QC (U)  
\]
is an equivalence; here $\QC$ stands for the category of quasi-coherent $\CO$-modules.

\begin{rem}
As far as I understand, Theorem~\ref{t:MB} easily implies a similar statement without assuming $S'$ to be affine over $S$. However, this slight generalization of Theorem~\ref{t:MB} will not be used in this article.
\end{rem}

\begin{rem}   \label{r:MB}
For any Noetherian ring $A$, Theorem~\ref{t:MB} is applicable in the following situation:

\[
S=\Spec A[t],\quad Y=\Spec A[t]/(t)\subset S,\quad U=\Spec A[t,t^{-1}],
\]

\[
S'=\Spec A[[t]], \quad U'=\Spec A((t)).
\]

\end{rem}

\ssec{Proof of \lemref{l:complete_local}}    \label{ss:2attractors-2proof}

By \lemref{l:mono}, we only have to prove that the map $Z^+(R)\to\Zp (R)$ is surjective.

The coordinate on $\BA^1_R$ will be denoted by $t$. So $\BA^1_R=\Spec R[t]$, and the formal completion of 
$\BA^1_R$ along $0$ is $\Spf R[[t]]$, where $R[[t]]$ is equipped with the $t$-adic topology. 

By definition, an element of $\Zp (R)$ is a $\BG_m$-equivariant morphism $\hat f:\Spf R[[t]]\to Z$.
Using the fact that $R$ and $R[[t]]$ are Henselian one easily checks that 
$$\Maps (\Spf R[[t]], Z)=\Maps (\Spec R[[t]], Z),$$
so we can also consider $\hat f$ as a morphism $\Spec R[[t]]\to Z$.
Let  $\hat f':\Spec R((t))\to Z$ denote the restriction of  $\hat f:\Spec R[[t]]\to Z$ to $\Spec R((t))$.
The problem is to extend $\hat f$  to a $\BG_m$-equivariant morphism 
$f:\BA^1_R=\Spec R[t]\to Z$. By Theorem~\ref{t:MB} and \remref{r:MB}, this problem is equivalent to extending $\hat f'$ to a $\BG_m$-equivariant morphism $f':(\BG_m)_R=\Spec R[t,t^{-1}]\to Z$.

Specifying $f'$ is the same as specifying its restriction to $\{1\}\subset (\BG_m)_R\,$; denote it by $z\in Z(R)$.
The requirement that $f'|_{\Spec R((t))}=\hat f'$ translates into the following condition:
\[
\iota (z)=\tilde z,
\]
where $\iota :Z(R)\to Z(R((t))\,)$ is induced by the embedding $R\mono R((t))$ and $\tilde z:\Spec R((t))\to Z$ is the composition
\[
\Spec R((t))\overset{(t^{-1},\hat f')}\longrightarrow \BG_m\times Z\longrightarrow Z
\]
(the second morphism is the action map).

\begin{rem} \label{r:down-to-earth}
In down-to-earth terms, $\tilde z (t):=t^{-1}\cdot \hat f(t)$, and the problem is to prove that $\tilde z (t)$ does not depend on $t$. This ``should be" true because $\BG_m$-equivariance of $\hat f$ implies that
\begin{equation}   \label{e:down-to-earth}
\tilde z (\lambda t)=\tilde z (t) \quad\mbox{ for } \lambda\in\BG_m\, .
\end{equation}
\end{rem}

Let us now transform Remark~\ref{r:down-to-earth} into a proof. The precise meaning of
\eqref{e:down-to-earth} is that 
\begin{equation}   \label{e:scientific}
\alpha (\tilde z)=\beta (\tilde z),
\end{equation}
where $\alpha:Z(R((t))\,)\to Z(R[\lambda,\lambda^{-1}]((t))\,)$ (resp. $\beta:Z(R((t))\,)\to Z(R[\lambda,\lambda^{-1}]((t)))\,$)
is induced by the natural embedding $R((t))\to R[\lambda,\lambda^{-1}]((t))$ (resp. by the homomorphism of topological $R$-algebras $R((t))\to R[\lambda,\lambda^{-1}]((t))$ such that $t\mapsto\lambda t$).
We want to conclude from \eqref{e:scientific} that $\tilde z\in Z(R((t))\,)$ is the image of a unique $z\in Z(R)$.
Let us proceed in two steps.

\medskip

\noindent {\bf Step 1.} Assume that $R$ is Artinian.  Then so is $R((t))$. Let  $z_0\in Z$ denote the image of 
the unique point of $\Spec R((t))$ and $O_{Z,z_0}$ the corresponding Henselian local ring.
Since $R((t))$ is Henselian the morphism $\tilde z:\Spec R((t))\to Z$ factors through $\Spec O_{Z,z_0}$.
So $\tilde z$ defines a homomorphism   $\varphi:O_{Z,z_0}\to R((t))$, and the problem is to show that
$\varphi (O_{Z,z_0})\subset R$. Indeed, if $p\in R((t))$ belongs to $\varphi (O_{Z,z_0})$ then by
\eqref{e:scientific}, $p$ satisfies the identity $p(\lambda t)=p(t)$, so $p\in R$.

\medskip

\noindent {\bf Step 2.}  Now drop the Artinian assumption. Let $m\subset R$ be the maximal ideal. Set 
$\Rn:=R/m^n$. Let $\tilde z_n\in Z(\Rn ((t))\,)$ be the image of $\tilde z$. By Step 1, $\tilde z_n$ comes from
a unique $z_n\in Z(\Rn )$. Since $R$ is a complete local ring the sequence $z_n$ defines a point $z\in Z(R)$, i.e., a morphism $z:\Spec R\to Z$. We have to prove that the composition 
$\Spec R((t))\to \Spec R\overset{z}\longrightarrow Z$ equals $\tilde z:\Spec R((t))\to Z$.
Just as in the proof of \lemref{l:mono}, let $E$ denote the equalizer of the two morphisms 
$\Spec R((t))\to Z$; this is a $\BG_m$-stable\footnote{Let us explain the precise meaning of the word 
``$\BG_m$-stable" here; we have to do it because $\BG_m$ does not act on the $R$-scheme $\Spec R((t))\,$. 
Instead, we have the ``action" morphism $a:\Spec R[\lambda ,\lambda^{-1}]((t))\to \Spec R((t))$ corresponding to the continuous $R$-algebra homomorphism $R((t))\to R[\lambda ,\lambda^{-1}]((t))$ such that $t\mapsto\lambda t$. By definition, $\BG_m$-stability of $E$ means that $a^{-1}(E)=p^{-1}(E)$, where $p:\Spec R[\lambda ,\lambda^{-1}]((t))\to \Spec R((t))$ is the obvious morphism.
} \subscheme of $\Spec R((t))$ containing $\Spec\Rn ((t))$ for each $n\in\BN$. Just as in the proof of  
\lemref{l:mono}, this implies that $E$ contains a $\BG_m$--stable subscheme $E'$ open in $\Spec R ((t))$ and 
containing $\Spec (R/m) ((t))$. Let us show\footnote{This step is unnecessary if $Z$ is separated: indeed, in this 
case $Z$ is a \emph{closed} subscheme of $\Spec R ((t))$ containing  
$\Spec\Rn ((t))$ for all $n$, so $E=\Spec R((t))$ and we are done.} that such $E'$ has to be equal to
$\Spec R((t))$.

Choose a closed subscheme $F\subset\Spec R((t))$ whose complement equals $E'$ and let
$I\subset R((t))$ be the corresponding ideal.

\begin{lem}   \label{l:properties of I}
(i) $I+m((t))=R((t))$.

(ii) Let $I'\subset R((t))$ be the ideal of all formal series $\sum\limits_i r_it^i$, $r_i\in R$, such that the series
$\sum\limits_i r_i\lambda^i t^i\in R[\lambda ,\lambda^{-1}]((t))$ belongs to $I\cdot R[\lambda,\lambda^{-1}]((t))$.
Then $I$ is contained in the radical of $I'$.
\end{lem}

\begin{proof}
The open subset $E'=(\Spec R ((t)))-F$ contains $\Spec (R/m) ((t))$, so 
$F\cap\Spec (R/m) ((t))=\emptyset$. This translates into (i). The fact that $E'$ is $\BG_m$-stable translates into (ii).
\end{proof}

It remains to show that any ideal $I\subset R((t))$ with properties (i)-(ii) from the lemma is the unit ideal.
Since $I$ is contained in the radical of $I'$ property (i) implies that $I'+m((t))=R((t))$, so $I'$ contains an element
of the form $\sum_i r_it^i$, where $r_i\in R$ and
\begin{equation}   \label{e:r_0}
r_0\in 1+m.
\end{equation}
By the definition of $I'$,  one has an equality of the form
\[
\sum\limits_i r_i\lambda^i t^i=\sum_{j=1}^n g_jh_j, \quad g_j\in R[\lambda,\lambda^{-1}]((t)),\; h_j\in I .
\]
Equating the coefficients of $\lambda^0$ in this equality, we see that $r_0\in I$. On the other hand, 
 $r_0$ is invertible by \eqref{e:r_0}. So $I$ is the unit ideal, and we are done.

\section{Proof of Theorem~\ref{t:tildeZ}}   \label{s:tildeZ}

In this section we prove Theorem~\ref{t:tildeZ}, which says that for any algebraic $k$-space of finite type equipped with a $\BG_m$-action, the space $\wt{Z}$ defined in Subsect.~\ref{sss:thespace} is 
an algebraic $k$-space of finite type.

We will use M.~Artin's technique for proving representability.\footnote{Instead of M.~Artin's technique one could use the one from \cite{Mur} (which does not rely on Artin's Approximation Theorem). This would not make the proof of representability more constructive.} In particular, in Subsect.~\ref{ss:proof_modulo} we use \cite[Theorem 1.6]{Ar2} to prove existence of a scheme equipped with a surjective etale morphism to $\tilde Z$. (Unfortunately, such proof of existence is not really constructive.)

We will be using the notation $\BX$ and $\BX_S$ introduced in 
Subsections~\ref{sss:family of hyperbolas}-\ref{sss:X_S}. Recall that $\BX:=\BA^2=\Spec k[\tau_1,\tau_2]$
and for any scheme $S$ over  $\BA^1$ we set $\BX_S:=\BX \underset{\BA^1}\times S$, where $\BX$ is mapped to $\BA^1$ by $(\tau_1,\tau_2)\mapsto \tau_1\cdot \tau_2\,$.

\ssec{Plan}
We will use the canonical morphism $\wt{p}:\wt{Z}\to \BA^1\times Z\times Z$, see Subsect.~\ref{sss:tilde p}.
As explained in Subsect.~\ref{sss:props tilde p}, representability of $\wt{Z}$ would immediately imply that $\wt{p}$ is unramified.

To prove representability of $\wt{Z}$, we will first prove some properties of $\wt{p}$, which are weaker than being representable and unramified. Namely, in Subsect.~\ref{ss:diagonalley}, we prove that the diagonal morphism
\begin{equation}   \label{e:diagonalley}
\Delta :\wt{Z}\to \wt{Z}\underset{\BA^1\times Z\times Z}\times \wt{Z}
\end{equation}
is an open embedding (in particular, it is \emph{representable}). This immediately implies that the morphism
$\wt{p}:\wt{Z}\to \BA^1\times Z\times Z$ is \emph{formally} unramified. Then we prove  another property of 
$\wt{p}$ (see \propref{p:intersection}) and deduce from it \propref{p:formal}, which is a strong form of pro-representability. Proposition~\ref{p:formal} implies ``openness of formal  etaleness" for morphisms from schemes to $\wt{Z}$ (see \corref{c:openness_etaleness}).
After that, it remains to check effective pro-representability, see 
Subsections~\ref{ss:proof_modulo}-\ref{ss:effective}.

Finally, in Subsect.~\ref{ss:conormal} (which is not used in the rest of the article) we give a reasonable ``upper bound" for the conormal sheaf of $\wt{Z}$ with respect to the unramified morphism 
$$\wt{p}:\wt{Z}\to \BA^1\times Z\times Z.$$ 
This bound is closely related to the proof of \propref{p:intersection}.   

\subsection{The diagonal morphism}  \label{ss:diagonalley}
\begin{prop}   \label{p:diagonalley}
The diagonal morphism \eqref{e:diagonalley} is an open embedding.
\end{prop}

Let us prove the proposition.
We have to show that for any scheme $S$ and any morphisms $\varphi_1,\varphi_2:S\to \wt{Z}$ giving rise to the same morphism $h:S\to\BA^1\times Z\times Z$, the equalizer $\Eq (\varphi_1,\varphi_2)$ is representable by an open subscheme of $S$. Let $f_1,f_2:\BX_S\to Z$ be the $\BG_m$-equivariant morphisms corresponding to  $\varphi_1,\varphi_2\,$ and let $E:=\Eq (f_1,f_2)$ be their equalizer. Then $E$ is a scheme of finite presentation over $\BX_S$ equipped with a monomorphism $E\mono\BX_S$. Moreover, since
$\varphi_1$ and $\varphi_2$ correspond to the same morphism $h:S\to\BA^1\times Z\times Z$ we have 
$$E\supset \BX'_S,$$
where $\BX'$ is the open subscheme $\BA^2-\{ 0\}\subset\BA^2=\BX$ and $\BX'_S:=\BX'\times_{\BA^1}S$.
Now it remains to prove the following lemma.

\begin{lem}     \label{l:diagonalley}
Let $S$ be a scheme over $\BA^1$. Let $E$ be a scheme of finite presentation over $\BX_S$ such that the map
$E\to\BX_S$ is a monomorphism. 
Assume that the morphism $\BX'_S\mono\BX_S$ factors through $E$. 

Let $U$ be the set of all $s\in S$ such that the corresponding morphism $E_s\to\BX_s$ is an isomorphism (here $E_s$ and $\BX_s$ are the fibers of $E$ and $\BX_S$ over $s$). Then

(i) the subset $U\subset S$ is open;

(ii) the map $E\times_SU\to\BX_U$ is an isomorphism.
\end{lem}

\begin{rem}
If the monomorphism $E\to\BX_S$ is a closed embedding then \lemref{l:diagonalley} is obvious;
moreover, in this case $U=S$. So if $Z$ is separated then \propref{p:diagonalley} is obvious; moreover, in this case the map \eqref{e:diagonalley}  is an isomorphism (i.e., $\wt{p}:\wt{Z}\to \BA^1\times Z\times Z$ is a monomorphism).
\end{rem}

\begin{proof}
We proceed in 3 steps. 

\medskip

\noindent {\bf Step 1.} 
Assume that $S$ is Artinian. Then statement (i) is tautological, and we may assume $U \ne \emptyset$, so the morphism $E\mono\BX_S$ is a closed embedding. Let $\CI\subset\CO_{\BX_S}$ be the ideal corresponding to $E\subset\BX_S\,$. Since 
$E\supset \BX'_S$ the restriction of $\CI$ to $\BX'_S$ is zero. This easily implies that $\CI=0$. So
$E=\BX_S\,$, which proves statement (ii).

\medskip

\noindent {\bf Step 2.} 
Assume that $S$ is Noetherian. Let $\wt{E}\subset E$ be the biggest open subscheme such that the morphism
 $\wt{E}\to\BX_S$ is etale. Then $\wt{E}$ is an open subscheme of $\BX_S$ containing $\BX'_S\,$. Applying the result of Step 1 to Artinian closed subschemes of $S$, we see that $\BX_s\subset\wt{E}$ for any $s\in U$. This allows to replace $E$ by $\wt{E}$; in other words, we can assume that the morphism $E\mono\BX_S$ is an open embedding. Then statements (i) and (ii) are clear because $\BX_S-E$ is a closed subset of 
 $\BX_S-\BX'_S$ and the morphism $\BX_S-\BX'_S\to S$ is closed (in fact, it is a closed embedding). 

\medskip

\noindent {\bf Step 3.} 
Since $E$ is of finite presentation we can remove the Noetherian assumption.
\end{proof}

Thus we have proved \propref{p:diagonalley}. Before formulating some corollaries of it, let us make an obvious remark.

\begin{rem}   \label{r:2commutation_colimits}
It is clear that the space $\wt{Z}$ is locally of finite presentation (i.e., the corresponding functor 
\{$k$-algebras\}$\to$\{sets\} commutes with filtering inductive limits).
\end{rem}

\begin{cor}   \label{c:morphisms from schemes}
Let $S$ be a $k$-scheme. Then

(i) any morphism $S\to\wt{Z}$ is representable;

(ii) if $S$ is locally of finite presentation over $k$ then any morphism $S\to\wt{Z}$ is locally of finite presentation.
\end{cor}

\begin{proof}
It suffices to show that the diagonal morphism $\wt{Z}\to\wt{Z}\times\wt{Z}$ is representable and locally of finite presentation. Both properties follow from Proposition~\ref{p:diagonalley}. (The second property also follows from \remref{r:2commutation_colimits}.)
\end{proof}

\begin{cor}   \label{c:form_unr}
The morphism $\wt{p}:\wt{Z}\to \BA^1\times Z\times Z$ is formally unramified. In other words, for any commutative diagram
\begin{equation}  \label{e:obs1}
\xymatrix{
S_0 \ar[d]_{}\ar@{^{(}->}[r]^{}& S\ar[d]^{}\\
\wt{Z}\ar[r]^{}&\BA^1\times Z\times Z
    }
\end{equation}
where $S$ is a scheme and $S_0$ is a closed subscheme defined by a nilpotent ideal, there exists at most one way to complete \eqref{e:obs1} to a commutative diagram
\begin{equation}  \label{e:obs2}
\xymatrix{
S_0 \ar[d]_{}\ar@{^{(}->}[r]^{}& S\ar[ld]^{}\ar[d]^{}\\
\wt{Z}\ar[r]^{}&\BA^1\times Z\times Z
}
\end{equation}
\end{cor}

\begin{proof}
Follows from Proposition~\ref{p:diagonalley}.
\end{proof}

\begin{rem}  \label{r:nil-ideal}
In \corref{c:form_unr} the condition ``$S_0$ is defined by a nilpotent ideal" can be replaced by a weaker condition $S_0\supset S_{\red}\,$. This follows from \remref{r:2commutation_colimits}.
\end{rem}

\ssec{Constructing formal neighborhoods}   \label{ss:formal}

\sssec{The property of $\wt{p}:\wt{Z}\to \BA^1\times Z\times Z$ to be proved.}
Fix a commutative diagram \eqref{e:obs1}. Say that a morphism of schemes $T\to S$ is \emph{liftable} (with respect to this diagram) if there exists a morphism $T\to\wt{Z}$ such that the corresponding digram
\begin{equation}  \label{e:obs3}
\xymatrix{
T\underset{S}\times S_0 \ar[d]_{}\ar@{^{(}->}[r]^{}&T\ar[ld]^{}\ar[d]^{}\\
\wt{Z}\ar[r]^{}&\BA^1\times Z\times Z
}
\end{equation}
commutes (note that such a morphism $T\to\wt{Z}$ is unique by \corref{c:form_unr}). Let us explain that the vertical arrows of \eqref{e:obs3} are obtained by composing the vertical arrows of \eqref{e:obs1} with the morphisms $T\to S$ and $T\times_SS_0\to S_0\,$.

\begin{prop}   \label{p:intersection}
For any commutative diagram \eqref{e:obs1} the corresponding functor
\[
T\mapsto\{\mbox{liftable morphisms } T\to S\}
\]
is representable by a closed subscheme $\underline S\subset S$.
\end{prop}

The proof of \propref{p:intersection} will use the following lemma, which is very abstract ($\wt{Z}$ and 
$\BA^1\times Z\times Z$ can be replaced by any spaces or functors).

\begin{lem}   \label{l:nilp2}
It suffices to prove \propref{p:intersection} if $n(S_0,S)\le 2$. Here $n(S_0,S)$ is the nilpotence degree of the ideal of the closed subscheme $S_0\subset S$.
\end{lem}

\begin{proof}
Proceed by induction on $n(S_0,S)$. If $n(S_0,S)>2$ we can choose a closed subscheme $S'\subset S$ containing $S_0$ so that $n(S_0,S')<n(S_0,S)$ and $n(S',S)\le 2$. Applying \propref{p:intersection} to the embedding $S_0\mono S'$ we get a closed subscheme $\underline S'\mono S'$
and a commutative diagram
\[
\xymatrix{
S_0 \ar[rd]_{}\ar@{^{(}->}[r]^{}&\underline S'\ar[d]\ar@{^{(}->}[r]^{} &S'\ar@{^{(}->}[r]^{} &S\ar[ld]^{}\\
&\wt{Z}\ar[r]^{}&\BA^1\times Z\times Z
}
\]
such that for any liftable morphism $f:T\to S$ one has $T\times_SS'=T\times_S\underline S'$. Then for any liftable 
$f:T\to S$ one has 
\[
n(T\times_S\underline S',T)=n(T\times_SS',T)\le n(S',S)\le 2,
\]
so $f:T\to S$ factors through the first infinitesimal neighborhood of $\underline S'$ in $S$. Replacing $S$ by this neighborhood we can assume that $n(\underline S',S)\le 2$. Now it remains to apply \propref{p:intersection} to the embedding $\underline S'\mono S$.
\end{proof}

The proof of \propref{p:intersection} given below is straightforward; the elementary \lemref{l:delo} is its heart.

\sssec{Proof of \propref{p:intersection}}   \label{ss:intersection-proof}

By \lemref{l:nilp2}, we can assume that $\CI^2=0$, where $\CI\subset\CO_S$ is the ideal of $S_0$.

Recall that for any scheme $S$ over $\BA^1$, an $\BA^1$-morphism $S\to\wt{Z}$ is the same as a $\BG_m$-equivaraiant morphism $\BX_S\to Z$, where $\BX_S:=\BX\times_{\BA^1}S$. We can think of an 
$\BA^1$-morphism $S\to \BA^1\times Z\times Z$ as a $\BG_m$-equivariant morphism $\BY_S\to Z$, where
$\BY_S:=\BY\times_{\BA^1}S$, $\BY:=\BA^1\times (\BG_m\sqcup \BG_m)$, and the $\BG_m$-action on 
$\BY=\BA^1\times\BG_m\times(\Spec k\sqcup\Spec k)$ comes from the $\BG_m$-action on the $\BG_m$~factor by translations. The morphism 
$\wt{p}:\wt{Z}\to \BA^1\times Z\times Z$ comes from the $\BG_m$-equivariant morphism $\nu :\BY\to\BX$ whose
restriction to the first copy of $\BA^1\times \BG_m$ is given by 
\[
(t,\lambda)\mapsto (\lambda ,\lambda^{-1}\cdot t)
\]
and whose restriction to the second copy of $\BA^1\times \BG_m$ is given by
\[
(t,\lambda)\mapsto (\lambda\cdot t,\lambda^{-1}).
\]
(Note that both restrictions are open embeddings.)

So a diagram \eqref{e:obs1} corresponds to the following data:

\begin{enumerate}
\item[(i)] a scheme $S$ over $\BA^1$ and a closed subscheme $S_0\subset S$ defined by an ideal 
$\CI\subset\CO_S$ such that $\CI^2=0$;

\item[(ii)] a $\BG_m$-equivariant morphism $f_0:\BX_{S_0}\to Z$;

\item[(iii)] a $\BG_m$-equivariant morphism $f:\BY_S\to Z$  whose restriction to $\BY_{S_0}$ is equal to the composition of $\nu_{S_0}:\BY_{S_0}\to\BX_{S_0}$ and $f_0:\BX_{S_0}\to Z$.
\end{enumerate}

Clearly (iii) is equivalent to the following datum: 

\begin{enumerate}
\item[(iii$'$)] a lift of the  composition $f_0^{\Kdot}\CO_Z\to\CO_{\BX_{S_0}}\to (\nu_{S_0})_*\CO_{\BY_{S_0}}$ to a $\BG_m$-equivariant morphism $f_0^{\Kdot}\CO_Z\to (\nu_S)_*\CO_{\BY_S}$.
\end{enumerate}
Here each algebraic space is equipped with the etale topology, and $f_0^{\Kdot}$ denotes the pullback with respect to $(f_0)_{\ET}: (\BX_S)_{\ET}=(\BX_{S_0})_{\ET}\to Z_{\ET}\,$.

We can rewrite [(iii$'$) as follows:

\begin{enumerate}
\item[(iii$''$)] a lift of the  morphism $f_0^{\Kdot}\CO_Z\to\CO_{\BX_{S_0}}$
to a $\BG_m$-equivariant morphism 
\begin{equation}  \label{e:to_lift}
f_0^{\Kdot}\CO_Z\to\CO_{\BX_{S_0}}\underset{(\nu_{S_0})_*\CO_{\BY_{S_0}}}\times (\nu_S)_*\CO_{\BY_{S}}\, .
\end{equation}
\end{enumerate}

Extending diagram \eqref{e:obs1} to diagram  \eqref{e:obs2} is equivalent to lifting the map \eqref{e:to_lift}
further to a morphism $f_0^{\Kdot}\CO_Z\to\CO_{\BX_{S}}\,$. By \corref{c:form_unr}, there is at most one such lift. This also follows from the first part of the next lemma.

\begin{lem}   \label{l:delo}
(a) The morphism $\CO_{\BX_S}\to(\nu_S)_*\CO_{\BY_{S}}$ is injective.

(b) Set $\CF_S:=\Coker (\CO_{\BX_S}\to(\nu_S)_*\CO_{\BY_{S}})$ and let $\pr_S :\BX_S\to S$ denote the projection. Then $(\pr_S)_*\CF_S$ is a free $\CO_S$-module (of countable rank).
\end{lem}

\begin{proof}
It suffices to consider the case where the morphism $S\to\BA^1$ is an isomorphism. In this case we have to check that the map
\[
k[\tau_1,\tau_2]\to k[\tau_1,\tau_2,\tau_1^{-1}]\times k[\tau_1,\tau_2,\tau_2^{-1}]
\]
is injective and its cokernel is a free module over $k[\tau_1\tau_2]\subset k[\tau_1,\tau_2]$. Injectivity is clear. The cokernel identifies via the map $(u,v)\mapsto u-v$ with
\[
k[\tau_1,\tau_2,\tau_1^{-1}]+k[\tau_1,\tau_2,\tau_2^{-1}]\subset k[\tau_1,\tau_2,\tau_1^{-1},\tau_2^{-1}],
\]
which is a module over $k[\tau_1\tau_2]$ freely generated by the elements $1$ and $\tau_i^{-n}$, where $n\in\BN$ and $i=1,2$.
\end{proof}

\begin{proof}[End of the proof of \propref{p:intersection}] Let  $\CF_S$ and $\pr_S$ be as in \lemref{l:delo}(b). The obstruction to solving our lifting problem is a morphism $\obs:f_0^{\Kdot}\CO_Z\to \CF_S\otimes (\pr_S)^*\CI$.
Using the equality $\CI^2=0$, it is easy to prove that  $\obs$ is a derivation\footnote{To prove this, set $B:=\CO_{\BX_{S_0}}\underset{(\nu_{S_0})_*\CO_{\BY_{S_0}}}\times (\nu_S)_*\CO_{\BY_{S}}$, $C:=\CO_{\BX_{S_0}}$, $B':=\CO_{\BX_S}$. We have ring homomorphisms $B'\mono B\to C$, whose composition is surjective. Moreover, the ideal $J:=\Ker (B\epi C)$ satisfies $J^2=0$. Then the composition 
$B\to B/B'\iso J/(J\cap B')$ is a \emph{derivation} of $B$ with coefficients in the $C$-module $J/(J\cap B')$. Our $\obs$ is the composition of this derivation with a ring homomorphism $f_0^{\Kdot}\CO_Z\to B$.} with respect to the ring homomorphism $f_0^{\Kdot}\CO_Z\to \CO_{\BX_{S_0}}$. So we can rewrite 
$\obs$ as a morphism of quasi-coherent  
$\CO_{\BX_{S_0}}$-modules
 $f_0^*\Omega^1_Z\to \CF_S\otimes (\pr_S)^*\CI$ and then (using the fact that $\BX_S$ is affine over 
 $S$) as a morphism of  quasi-coherent  $\CO_{S_0}$-modules\footnote{Of course, this morphism is also $\BG_m$-equivariant and commutes with the action of  the algebra $(\pr_{S_0})_*\CO_{\BX_{S_0}}$. }
 \begin{equation}   \label{e:the_obstruction}
 (\pr_{S_0})_* f_0^*\Omega^1_Z\to  (\pr_S)_* \CF_S\otimes\CI \,  .
 \end{equation}
 Now let us explain how to construct the closed subscheme $\underline S\subset S$ from \propref{p:intersection}. By \lemref{l:delo}(b), $(\pr_S)_* \CF_S$ is a free $\CO_S$-module. After choosing a basis in it, we can think of the morphism \eqref{e:the_obstruction} as an (infinite) collection of morphisms 
 $(\pr_{S_0})_* f_0^*\Omega^1_Z\to \CI $. Let $\CI_1\subset\CI$ be the submodule (or equivalently, the ideal) generated by their images. Finally, let $\underline S\subset S$ be the closed subscheme corresponding to 
 $\CI_1\subset\CO_S\,$. It is easy to see that $\underline S$ has the property from  \propref{p:intersection}. 
 \end{proof}
 
 \sssec{A digression about $\obs$}  \label{sss:obsdigression}
 The obstruction $\obs$ studied at the end of the proof of  \propref{p:intersection} lives in the group
 $\Hom^{\BG_m}(f_0^*\Omega^1_Z\, , \CF_S\otimes (\pr_S)^*\CI )$, which can be identified (after some work) with
 \begin{equation}   \label{e:where obs lives}
 \Hom ( ((\pr_{S_0})_*(f_0^*\Omega^1_Z\otimes\omega_{\BX_{S_0}/S_0}))^{\BG_m},\CI),
 \end{equation}
 where $\omega_{\BX_{S_0}/S_0}$ is the relative dualizing sheaf.
 This fact will not be used in the proof of Theorem~\ref{t:tildeZ}, so we skip the details.

 \sssec{Constructing formal neighborhoods}     \label{sss:formal}
 
 \begin{lem}   \label{l:unrami}
 Let $S_0$ be a $k$-scheme of finite type. The following properties of a morphism $\varphi :S_0\to\wt{Z}$ are equivalent:
 
 (i) $\varphi$ is formally unramified;
 
 (ii) $\varphi$ is unramified;
 
 (iii) the composition 
 \begin{equation}    \label{e:thecomp}
 S_0\overset{\varphi}\longrightarrow\wt{Z}\overset{\wt{p}}\longrightarrow\BA^1\times Z\times Z
 \end{equation} 
 is unramified.
 \end{lem}
 
  \begin{proof}
  By \corref{c:morphisms from schemes}, we have (i)$\Leftrightarrow$(ii).
  Since $S_0$ is of finite type property (iii) is equivalent to the composition \eqref{e:thecomp} being formally unramified. The latter is equivalent to (i) by \lemref{c:form_unr}.
  \end{proof}
  


\begin{defn}   \label{d:formal neighborhood}
Let $F$ and $G$ be contravariant functors 
$\mbox{\{affine }k\mbox{-schemes\}}\to\mbox{\{sets\}}$ which are locally of finite presentation. Let 
$\varphi :F\to G$ be a formally unramified morphism. Then the \emph{formal neighborhood} of  $F$ with respect to $\varphi :F\to G$ is the following contravariant functor 
$\mbox{\{affine }k\mbox{-schemes\}}\to\mbox{\{sets\}}$:
\begin{equation}   \label{e:form neighb}
T\mapsto G(T)\underset{G(T_{\red})}\times F (T_{\red}).
\end{equation}
\end{defn}
Note that the functor \eqref{e:form neighb} is again locally of finite presentation (because it is a fiber product of functors each of which is locally of finite presentation). 

\begin{rem}
The canonical morphism from $F$ to the above formal neighborhood is a monomorphism (to prove this, represent $T_{\red}$ as the projective limit of all closed subschemes of $T$ defined by finitely generated nilpotent ideals; then
use that $\varphi$ is formally unramified and $F$ and $G$ are locally of finite presentation).
\end{rem}

\begin{rem}
We will apply Definition~\ref{d:formal neighborhood} only in the following situation: $F$ is a $k$-scheme of finite type and $G$ is either an algebraic $k$-space of finite type or the space $\wt{Z}$ (which will \emph{eventually} be shown to be an algebraic $k$-space of finite type).
\end{rem}

\begin{prop}   \label{p:formal}
Let $S_0$ be a $k$-scheme of finite type and suppose that $\varphi :S_0\to\wt{Z}$ has the equivalent properties (i)-(iii)  from \lemref{l:unrami}. Let  $S_{\infty}$ denote the formal neighborhood of $S_0$ with respect to $\varphi :S_0\to\wt{Z}$. Let $S'_{\infty}\,$ denote the formal neighborhood of $S_0$ with respect to $\wt{p}\circ\varphi :S_0\to\BA^1\times Z\times Z$. Then

(i) the morphism  $S_{\infty}\to S'_{\infty}$ is a closed embedding;

(ii) $S_{\infty}$ can be represented as an inductive limit of a diagram
\begin{equation}   \label{e:embeddings_diagram}
S_0\mono S_1\mono S_2\mono\ldots
\end{equation}
in which each $S_n$ is a $k$-scheme of finite type, the morphisms  $S_n\to S_{n+1}$ are closed embeddings, 
and for each $N\ge n$ the $n$-th infinitesimal neighborhood of $S_0$ in $S_N$ equals $S_n\,$.
\end{prop}

\begin{rem}
The diagram \eqref{e:embeddings_diagram} is unique up to isomorphism (indeed, $S_n$ can be characterized as the biggest closed subscheme of $S_{\infty}$ containing $S_0$ and such that the $(n+1)$-st power of the ideal of $\CO_{S_n}$ defining $S_0$ equals 0).
 One could call $S_n$ the \emph{ $n$-th infinitesimal neighborhood of $S_0$ with respect to $\varphi :S_0\to\wt{Z}\,$.}
\end{rem}

\begin{proof}

Standard arguments prove the following property\footnote{This property holds for the formal neighborhood of $S_0$ with respect to any unramified morphism $g:S_0\to Y$, where $Y$ is any algebraic $k$-space of finite type (in particular, for $Y=\BA^1\times Z\times Z$, $g=\wt{p}\circ\varphi$). To prove it,
use that such $g$ can be locally factored as a closed embedding followed by an etale morphism (here ``locally" means ``etale-locally with respect to $Y$ and Zariski-locally with respect to $S_0\,$").}
of $S'_{\infty}$: it can be (uniquely) represented as an inductive limit of a diagram
\[
S_0\mono S'_1\mono S'_2\mono\ldots
\]
in which each $S'_n$ is a $k$-scheme of finite type, the morphisms  $S'_n\to S'_{n+1}$ are closed embeddings, 
and for each $N\ge n$ the $n$-th infinitesimal neighborhood of $S'_0$ in $S'_N$ equals $S'_n\,$.
Set 
$$S_n:=S'_n\underset{S'_{\infty}}\times S_{\infty}\, .$$
It remains to prove that for every $n\in\BN$ the morphism $S_n\to S'_n$ is a closed embedding. 
To this end, consider the diagram
\[
\xymatrix{
S_0 \ar[d]_{}\ar@{^{(}->}[r]^{}& S'_n\ar[d]^{}\\
\wt{Z}\ar[r]^{}&\BA^1\times Z\times Z
    }
\]
of type \eqref{e:obs1}. Applying \propref{p:intersection} to this diagram, one gets a closed subscheme of 
$S'_n\,$. It is easy to check that this closed subscheme equals $S_n\,$ (use \corref{c:form_unr} and \remref{r:nil-ideal}).
 \end{proof}

  \sssec{An openness lemma}
  %

 
 \begin{cor}  \label{c:openness_etaleness}
 Let $S_0$ be a $k$-scheme of finite type and  $\varphi :S_0\to\wt{Z}$ a morphism.
 Let $s\in S_0$ be a closed point. Suppose that 
 
 \begin{enumerate}
\item[(i)] the morphism $s\to\wt{Z}$ is a monomorphism (so the formal neighborhood of $s$ in $\wt{Z}$ is well-defined and \propref{p:formal} is applicable to it);

\item[(ii)] $\varphi$ induces an isomorphism between the formal neighborhoods of $s$ in $S_0$ and $\wt{Z}$.
\end{enumerate}

 \noindent Then there is an open subscheme $U\subset S_0$ containing $s$ such that the restriction of 
 $\varphi$ to $U$ is etale.
  \end{cor}

\begin{proof}
Since $\wt{p}:\wt{Z}\to\BA^1\times Z\times Z$ is formally unramified, condition (ii) implies that 
the morphism $\wt{p}\circ\varphi :S_0\to\BA^1\times Z\times Z$ is unramified at $s$.
So after shrinking $S_0$ we can assume that the morphism $\wt{p}\circ\varphi$ is unramified. Then
\propref{p:formal} is applicable. 

Let $S_1$ be as in \propref{p:formal}(ii). Let $\CI\subset\CO_{S_1}$ be the ideal of the closed subscheme $S_0\subset S_1$ and let $F\subset S_0$ be the support of the coherent $\CO_{S_0}$-module $\CI$.
Let us check that the open subscheme $U:=S_0-F$ has the required properties.

Condition (ii) implies that $s\in U$. 
After replacing $S$ by $U$ we get
$\CI=0$, i.e., $S_1=S_0\,$. This implies that $S_{\infty}=S_0$. So $\varphi$ is formally etale.
By \corref{c:morphisms from schemes}, this implies that  $\varphi$ is etale.
\end{proof}

 \ssec{Proof of Theorem~\ref{t:tildeZ} modulo \propref{p:effective}}   \label{ss:proof_modulo}

 Given a $k[t]$-algebra $R$, set 
$$\wt{Z}(R):=\Maps_{\BA^1} (\Spec R, \wt{Z}).$$

\begin{prop} \label{p:effective}
Let $R$ be a complete local Noetherian $k[t]$-algebra and $m\subset R$ the maximal ideal. Then the map
\begin{equation}    \label{e:bijection_to_prove}
\wt{Z}(R)\to\limfromn \wt{Z}(R/m^n)
\end{equation}
is bijective.
\end{prop}

The proposition will be proved in Subsect.~\ref{ss:effective}. Now let us deduce Theorem~\ref{t:tildeZ} from
\propref{p:effective} and the results of Subsections~\ref{ss:diagonalley}-\ref{ss:formal}.

We already know that $\wt{Z}\times_{\BA^1}(\BA^1-\{0\})$ and $\wt{Z_0}:=\wt{Z}\times_{\BA^1}\{0\}$ are algebraic $k$-spaces of finite type (see \remref{r:alg space}). So to prove Theorem~\ref{t:tildeZ}, it suffices to check that $\wt{Z}$ is an algebraic $k$-space \emph{locally} of finite type. 

 Let $\wt{z}_0$ be a 
spectrum of a finite extension of $k$ equipped with a morphism $\varphi :\wt{z}_0\to\wt{Z}$.
Our goal is to construct a $k$-scheme $S$ of finite type equipped with an etale morphism $S\to\wt{Z}$ whose fiber over $\wt{z}_0$ is non-empty.

The morphism $\varphi :\wt{z}_0\to\wt{Z}$ factors as $\wt{z}_0\to \wt{z}'_0\overset{\varphi'}\to\wt{Z}$, where
$\wt{z}'_0$ is a spectrum of a field and $\varphi':\wt{z}_0\to\wt{Z}$ is a \emph{mono}morphism.\footnote{This follows from the fact that $\wt{Z}\times_{\BA^1}(\BA^1-\{0\})$ and $\wt{Z_0}:=\wt{Z}\times_{\BA^1}\{0\}$ are algebraic $k$-spaces of finite type. Alternatively, one can define $\wt{z}'_0$ to be the quotient of $\wt{z}_0$ by the equivalence relation on $\wt{z}_0$ corresponding to $\varphi$ and then define $\varphi':\wt{z}_0\to\wt{Z}$ using the fact that $\wt{Z}$ is an fppf sheaf.}
So after replacing  $\wt{z}_0$ by $\wt{z}'_0$ we can assume that already $\varphi :\wt{z}_0\to\wt{Z}$ is a monomorphism.

Applying \propref{p:formal} to the monomorphism $\varphi :\wt{z}_0\mono\wt{Z}$, we see that the formal neighborhood of $\wt{z}_0$ in $\wt{Z}$ equals $\Spf A$ for some complete Noetherian local ring $A$. 
Applying \propref{p:effective}, we upgrade this pro-representability result to \emph{effective} pro-representability; in other words,  we get a morphism $\Spec A\to\wt{Z}$ extending the morphism $\Spf A\to\wt{Z}$. Using Theorem~1.6 from M.~Artin's work \cite{Ar2} and the fact that 
$\wt{Z}$ is locally of finite presentation, we get a $k$-scheme $S'$ of finite type equipped with a closed point $s_0\in S'$ and a morphism $(S',s_0)\to (\wt{Z},\wt{z}_0)$ inducing an isomorphism between the formal completions. By \corref{c:openness_etaleness}, $s_0$ has a Zariski neighborhood $S\subset S'$ such that the morphism $S\to\wt{Z}$ is etale. Thus we have proved Theorem~\ref{t:tildeZ} modulo \propref{p:effective}.

\ssec{Proof of \propref{p:effective}}      \label{ss:effective}

The proof below is parallel to that of \lemref{l:complete_local}.

\bigskip

If $t$ is invertible in $R$ the statement is clear because 
$\wt{Z}\times_{\BA^1}(\BA^1-\{ 0\})\simeq\BG_m\times Z$ is an algebraic space. So from now we will assume that $t\in m$.

Set $A_R:=R[\tau_1,\tau_2]/(\tau_1\tau_2-t)$. Recall that
\[
\wt{Z} (R):=\gMaps (\BX_R ,Z), \quad\quad \BX_R:=\Spec A_R\, .
\]
We will use the following notation:
\[
\BX_R [t^{-1}]:=\Spec A_R [t^{-1}]; \quad\quad\quad \BX_R [\tau_i^{-1}]:=\Spec A_R[\tau_i^{-1}],\quad  i=1,2;
\]

\[
 \hat A_R:=R[[\tau_1,\tau_2]]/(\tau_1\tau_2-t); \quad\quad\quad  \hat\BX_R:=\Spec \hat A_R\, ;
\]

\[
\hat\BX_R [t^{-1}]:=\Spec \hat A_R [t^{-1}]; \quad\quad\quad \hat\BX_R [\tau_i^{-1}]:=\Spec \hat A_R[\tau_i^{-1}],
\quad  i=1,2.
\]

Applying Moret-Bailly's Theorem~\ref{t:MB} for 
\[
S=\BX_R\, , \; S'=\hat\BX_R\, , \; Y=\Spec R/(t)=\Spec R[\tau_1,\tau_2]/(\tau_1,\tau_2,\tau_1\tau_2-t)\subset\BX_R
\]
and then applying Zariski descent to the covering
\[
\BX_R-Y=\BX_R [\tau_1^{-1}]\cup \BX_R [\tau_2^{-1}]
\]
one gets an exact sequence
\[
\wt{Z} (R)\to F(R)\rightrightarrows G_1 (R)\times G_2 (R)\times\gMaps (\BX_R [t^{-1}],Z),
\]
where 
\begin{equation}   \label{e:F}
F(R):=\gMaps (\hat\BX_R,Z)\times\gMaps (\BX_R [\tau_1^{-1}],Z)\times
\gMaps (\BX_R [\tau_2^{-1}],Z),
\end{equation}

\begin{equation}    \label{e:G_i}
G_i(R):=\gMaps (\hat\BX_R [\tau_i^{-1}],Z), \quad i=1,2.
\end{equation}

\begin{lem}   \label{l:still exact}
The sequence
\[
\wt{Z} (R)\to F(R)\rightrightarrows G_1 (R)\times G_2 (R)
\]
is still exact.
\end{lem}

\begin{proof}
It suffices to show that the map
\[
\gMaps (\BX_R [t^{-1}],Z)\to \gMaps (\hat\BX_R [t^{-1}],Z)
\]
is injective. 

Let $f_1,f_2:\BX_R [t^{-1}]\to Z$ be $\BG_m$-equivariant morphisms. Using a $\BG_m$-equivariant isomorphism
$$\BX_R [t^{-1}]\simeq\BG_m\times\Spec R[t^{-1}], $$
we see that the equalizer $E:=\Eq (f_1,f_2)$ equals $E_0\underset{\Spec R[t^{-1}]}\times \BX_R [t^{-1}]$ for some scheme $E_0$ equipped with a monomorphism $\mu :E_0\mono\Spec R[t^{-1}]$.

Now suppose that $f_1$ and $f_2$ have equal images in $\gMaps (\hat\BX_R [t^{-1}],Z)$. Then $\mu$ becomes an isomorphism after base change with respect to the morphism 
$\pi :\hat\BX_R [t^{-1}]\to\Spec R[t^{-1}]$. But $\pi$ is faithfully flat (because $\hat\BX_R$ is faithfully flat over 
$\Spec R$). So $\mu$ is an isomorphism and therefore $f_1=f_2$.
\end{proof}

We have to prove that the map \eqref{e:bijection_to_prove} is bijective. Let $F$ and $G_i$ be as in
\eqref{e:F}-\eqref{e:G_i}.

It is easy to see that the map
\[
F(R)\to\limfromn F(R/m^n)
\]
is bijective.\footnote{To prove this, analyze separately each of the three factors in the r.h.s. of \eqref{e:F}. To analyze the second and third factor, use a $\BG_m$-equivariant $R$-isomorphism $\BX_R[\tau_i^{-1}]\iso(\BG_m)_R$\;.} So by \lemref{l:still exact}, it remains to show that the map 
$$G_i(R)\to \limfromn G_i(R/m^n)$$
is injective for $i=1.2$. Let us prove this for $i=1$. We will proceed  as at Step 2 of the proof of \lemref{l:complete_local} (see Subsect.~\ref{ss:2attractors-2proof}).

Suppose that $f_1,f_2\in G_1 (R)$ have equal images in $G_1 (R/m^n)$ for each $n\in\BN$. Let $E$ denote the equalizer of the $\BG_m$-equivariant morphisms $f_1,f_2:\hat\BX_R [\tau_1^{-1}]\to Z$.
Just as at Step 2 of the proof of \lemref{l:complete_local}, we see that $E$ contains an open
$\BG_m$-stable subscheme $E'\subset \hat\BX_R [\tau_1^{-1}]$ such that $E'\supset\hat\BX_{R/m} [\tau_1^{-1}]$. It remains to show that such $E'$ has to be equal to $\hat\BX_R [\tau_1^{-1}]$.

Choose a closed subscheme $F\subset \hat\BX_R$ whose complement equals $E'$ and let
$$I\subset\hat A_R:=R[[\tau_1,\tau_2]]/(\tau_1\tau_2-t)$$
be the corresponding ideal. The inclusion $E'\supset\hat\BX_{R/m} [\tau_1^{-1}]$ and the fact
that $E'$ is $\BG_m$-stable translate into the following properties of $I$.

\begin{lem}   \label{l:2properties of I} 
(i) The image of $I$ in  $\hat A_{R/m}$ 
contains $\tau_1^N$ for some $N\in\BZ_+\,$.

(ii) Let $\varphi :\hat A_R\to\hat A_{R[\lambda ,\lambda^{-1}]}$ be the continuous $R$-algebra homomorphism such that
$$\tau_1\mapsto \lambda \tau_1, \quad \tau_2\mapsto \lambda^{-1} \tau_2\, ;$$
then $I$ is contained in the radical of the ideal $I':=\varphi^{-1}(I\cdot \hat A_{R[\lambda ,\lambda^{-1}]})$. \qed
\end{lem}

It remains to prove the following.

\begin{lem}
Suppose that an ideal $I\subset\hat A_R$ has properties (i)-(ii) from \lemref{l:2properties of I}.
Then $I$ contains a power of $\tau_1\,$. 
\end{lem}

\begin{proof}
Since $I$ is contained in the radical of $I'$ property (i) implies that $I'$ contains an element $u$
of the form 
$$u=\sum\limits_{i=0}^{\infty}r_i\tau_1^i+\sum\limits_{i<0}r_i\tau_2^{-i},$$ where $r_i\in R$ and
\begin{equation}   \label{e:r_N}
r_N\in 1+m
\end{equation}
for some $N\in\BZ_+\,$. We will show that then
\begin{equation}   \label{e:We will show}
\tau_1^N\in I\,. 
\end{equation}
Since $u\in I':=\varphi^{-1}(I\cdot \hat A_{R[\lambda ,\lambda^{-1}]})$ we have 
\begin{equation}   \label{e:belongs}
\varphi (u)\in I\cdot \hat A_{R[\lambda ,\lambda^{-1}]}\, .
\end{equation}
Recall that $\hat A_{R[\lambda ,\lambda^{-1}]}:=R[\lambda ,\lambda^{-1}][[\tau_1,\tau_2]](\tau_1\tau_2-t)$.
So by the definition of $\varphi$ (see \lemref{l:2properties of I}), property \eqref{e:belongs} means that in
$R[\lambda ,\lambda^{-1}][[\tau_1,\tau_2]]$ one has an equality of the form
\begin{equation}  \label{e:where we equate coefficients}
\sum\limits_{i=0}^{\infty}r_i\lambda^i \tau_1^i+\sum\limits_{i<0}r_i\lambda^i\tau_2^{-i}=\sum_{j=1}^n g_jh_j, \quad g_j\in R[\lambda,\lambda^{-1}][[\tau_1,\tau_2]],\; 
h_j\in \widetilde I ,
\end{equation}
where $\widetilde I\subset R][[\tau_1,\tau_2]]$ is the pre-image of 
$I\subset\hat A_{R}:=R[[\tau_1,\tau_2]](\tau_1\tau_2-t)$. Equating the coefficients of $\lambda^N$ in 
\eqref{e:where we equate coefficients}, we see that $r_N \tau_1^N$ viewed as an element of $R[[\tau_1,\tau_2]]$ belongs to
$\widetilde I$. So $r_N \tau_1^N$ viewed as an element of $\hat A_R$ belongs to $I$. By \eqref{e:r_N}, $r_N$ is invertible, so we get \eqref{e:We will show}.
\end{proof}

\ssec{Virtual conormal sheaf of $\wt{Z}$ with respect to $\BA^1\times Z\times Z$}       \label{ss:conormal}
Let $\CN$ denote the conormal sheaf of $\wt{Z}$ with respect to the unramified morphism 
$\wt{p}:\wt{Z}\to\BA^1\times Z\times Z$. We are going to define another coherent sheaf $\CN'$ on $\wt{Z}$ such that $\CN$ is canonically a quotient of $\CN'$. One could call $\CN'$ the \emph{virtual conormal sheaf}.

Here is the definition of $\CN'$:
for any affine scheme $S$ equipped with a morphism $\varphi :S\to\wt{Z}$
\[
H^0(S,\varphi^*\CN'):=H^0(\BX_S,f^*\Omega^1_Z\otimes\omega_{\BX_S/S})^{\BG_m},
\]
where $f:\BX_S\to Z$ is the $\BG_m$-equivariant morphism corresponding to $\varphi$ and 
$\omega_{\BX_S/S}$ is the relative dualizing sheaf.

The following facts are not used in the rest of the article; we formulate them for completeness.

First, 
Subsection~\ref{sss:obsdigression}
yields a \emph{canonical epimorphism} $\CN'\epi\CN$. Indeed, by formula~\eqref{e:where obs lives}, for any coherent $\CO_{\wt{Z}}$-module $\CI$ we have a canonical injection $\Hom (\CN ,\CI)\mono\Hom (\CN' ,\CI)$, which is functorial in $\CI$.

Second, let $\wt{Z}_{\der}$ denote the derived version of $\wt{Z}$ (to define it, replace the space 
$\GMaps$ from the definition of $\wt{Z}$ by its derived version). 
Let $\CN_{\der}$ denote the conormal complex of $\wt{Z}_{\der}$ with respect to $\BA^1\times Z\times Z$ (in other words, $\CN_{\der}[1]$ is the relative cotangent complex of $\wt{Z}_{\der}$ with respect to $\BA^1\times Z\times Z$). Then \emph{$\CN'$ identifies with the 0-th cohomology sheaf of $\CN_{\der}\,$.} (We skip the details.)



\appendix

\section{Proof of \lemref{l:postponed}} \label{s:very general}

\begin{lem}  \label{l:no factorization}
Let $A$ and $B$ be algebraic $k$-spaces and $f:A\to B$ a surjective morphism with $f_*\CO_A=\CO_B$
(here $\CO_A$, $\CO_B$ are sheaves on the etale sites $A_{\ET}$, $B_{\ET}$ and $f_*$ is understood in the non-derived sense). Suppose that $f:A\to B$ factors as 
\[
A\overset{f'}\longrightarrow B' \overset{i}\mono B,
\]
where $i:B' \mono B$ is a monomorphism of finite type. Then $i$ is an isomorphism (i.e., $B'=B$).
\end{lem}


\begin{proof}
It suffices to show that for every $b\in B$ there exists an etale morphism $(B_1,b_1)\to (B,b)$ such that $i$ becomes an isomorphism after base change to $B_1\,$.

Note that since $f$ is surjective so is $i:B'\mono B$. In other words, $B'$ and $B$ have the same field-valued points. In particular, $b\in B'$.

The monomorphism $i:B'\mono B$ has finite type, so after etale base change $(B_1,b_1)\to (B,b)$, we can assume\footnote{E.g., see Lemma 37.17.2 from \cite{St} (whose ``tag" is 04HI).} that there is an open subspace 
$\wt{B}\subset B'$ which is closed in $B$ (and therefore closed in $B'$). Set $\wt{A}:=(f')^{-1}(\wt{B})$, then
$\wt{A}\subset A$ is both open and closed. Let $1_{\wt{A}}\in H^0(A,\CO_A)$ denote the characteristic function of $\wt{A}$. Since $f_*\CO_A=\CO_B$ the map $H^0(B,\CO_B )\to H^0(A,\CO_A)$ is an isomorphism. So $1_{\wt{A}}$ comes from an idempotent element of $H^0(B,\CO_B)$. After shrinking $B$, we can assume that this element equals~$1$. This means that $1_{\wt{A}}=1$, $\wt{A}=A$, $\wt{B}=B'$, and $i:B'\mono B$ is a closed embedding.

Let $\CI_{B'}\subset\CO_B$ be the ideal of the closed subspace $B'\subset B$. Then  
$\CI_{B'}\subset\Ker (\CO_B\to f_*\CO_A)=0$. So $B'=B$.
\end{proof}

Now let us prove \lemref{l:postponed}. It says the following:

\begin{lem}  
Let $A$, $B$, $Z$ be algebraic $k$-spaces and $f:A\to B$ a surjective morphism with $f_*\CO_A=\CO_B$
Then

(i) the map $Maps(B , Z)\to\Maps(A ,Z)$ induced by $f$ is injective;

(ii) if $B_0\subset B$ is a closed subspace containing $B_{\red}$ and $A_0=f^{-1}(B_0)$ then the diagram
\[
\CD
\Maps(B\, , Z) @>{}>>  \Maps(A\, ,Z) \\
@V{}VV   @V{}VV   \\ 
\Maps(B_0\, , Z) @>{}>>  \Maps(A_0\, ,Z)
\endCD
\]
induced by $f$ is Cartesian. 
\end{lem}

\begin{proof}
(i) Let $g_1,g_2:B\to Z$ be morphisms such that $g_1\circ f=g_2\circ f$. Let $i:B'\mono B$ denote the equalizer of $g_1$ and $g_2$ ( i.e., the preimage of the diagonal with respect to 
$(g_1,g_2):B\to Z\times Z$). By \lemref{l:no factorization}, $i$ is an isomorphism. So $g_1=g_2\,$.

(ii) Suppose that $g_0\in\Maps(B_0\, , Z)$ and $h\in\Maps(A\, ,Z)$ have the same image in $\Maps(A_0\, ,Z)$.
We have to extend $g_0:B_0\to Z$ to a morphism $g:B\to Z$ whose composition with $f:A\to B$ equals $h$.

We have a canonical $k$-algebra homomorphism 
\begin{equation}  \label{e:what we have}
g_0^{\Kdot}\CO_Z\to\CO_{B_0}\, , 
\end{equation}
where 
$g_0^{\Kdot}$ denotes the sheaf-theoretical pullback with respect to $(g_0)_{\ET}: (B_0)_{\ET}\to Z_{\ET}\,$. Note that $B_{\ET}=(B_0)_{\ET}\,$.  So extending $g_0$ to a morphism $g:B\to Z$ is equivalent to lifting the map~\eqref{e:what we have} to a $k$-algebra homomorphism $\varphi :g_0^{\Kdot}\CO_Z\to\CO_B\,$.
Define $\varphi$ to be the composition 
$$g_0^{\Kdot}\CO_Z\to f_*\CO_A\iso\CO_B\,,$$
where the first arrow comes from the homomorphism $f^{\Kdot}g_0^{\Kdot}\CO_Z=h^{\Kdot}\CO_Z\to\CO_A$
corresponding to $h:A\to Z$.
\end{proof}

\section{Some results of Bia{\l}ynicki-Birula, Konarski, and Sommese}   \label{s:Polish}

Recall that if $Z$ is separated then $p^+:Z^+\to Z$ is a monomorphism. But already if $Z$ is the projective line equipped with the standard $\BG_m$-action, the morphism $p^+:Z^+\to Z$ is not a locally closed embedding.

\begin{thm}
Let $Z$ be a separated scheme over an algebraically closed field $k$ equipped with a $\BG_m$-action.
Then each of the following conditions ensures that the restriction of $p^+:Z^+\to Z$ to each connected component \footnote{Using the $\BA^1$-action on $Z^+$, it is easy to see that each connected component of $Z^+$ is the preimage of a connected component of $Z^0$ with respect to the map $q^+:Z^+\to Z^0\,$.} 
of $Z^+$  is a locally closed embedding:

(i) $Z$ is smooth;

(ii) $Z$ is normal and quasi-projective;

(iii) $Z$ admits a $\BG_m$-equivariant locally closed embedding into
a projective space $\BP(V)$, where $\BG_m$ acts linearly on $V$. 
\end{thm}

 Case (i) is due to A.~Bia{\l}ynicki-Birula \cite{Bia}.

Case (iii) immediately follows from the easy case $Z=\BP(V)$. Case (ii) turns out to be a particular case of (iii)
because by Theorem~1 from \cite{Sum}, if $Z$ is normal and quasi-projective then it admits a $\BG_m$-equivariant locally closed embedding into a projective space.

\medskip

In case (i) the condition that $Z$ is a scheme (rather than an algebraic space) is essential, as shown by A.~J.~Sommese \cite{Som}. In case (ii) the quasi-projectivity condition is essential, as shown by 
J.~Konarski \cite{Kon} using a method developed by J.~Jurkiewicz \cite{Ju1,Ju2}. 
In this example $Z$ is a 3-dimensional toric variety which is proper but not projective; it is constructed by drawing a certain picture on a 2-sphere, see the last page of \cite{Kon}.

In case (ii) normality is clearly essential (to see this, take $Z$ to be the curve obtained from $\BP^1$ by gluing $0$ with 
$\infty$).

%
%
%
%

\section{A categorical framework}   \label{s: 2-categorical framework}
This section is largely due to my discussions with N.~Rozenblyum. In particular, he educated me about the ``twisted arrow category" (see Subsect.~\ref{ss:Tw}).


\ssec{The goal of this Appendix}   \label{ss:the goal}
Let $Z$ be an algebraic $k$-space of finite type equipped with a $\BG_m$-action. In this situation we have a non-flat family of correspondences 
$$Z\leftarrow \wt{Z}_t\to Z,\quad\quad t\in\BA^1$$  
(which was constructed  in \secref{ss:deg}) and also the following two correspondences relating $Z$ with $Z^0$ (which were constructed in Subsections~\ref{ss:attr} and \ref{ss:repeller}):
\begin{equation}   \label{e:2+ and - correspondences}
Z^0\overset{q^+}\longleftarrow Z^+\overset{p^+}\longrightarrow Z, \quad\quad
Z\overset{p^-}\longleftarrow Z^-\overset{q^-}\longrightarrow Z^0.
\end{equation}
It is natural to ask what kind of categorical structure is formed by all these correspondences combined together.
In Subsect.~\ref{ss:What we want} we give 
an answer to this question (in a format convenient for proving the main results of \cite{DrGa1,DrGa2}). If one disregards \eqref{e:2+ and - correspondences} and considers only the correspondences $\wt{Z}_t$ then what one gets is a \emph{lax action} of the multiplicative monoid $\BA^1$ on $Z$ \emph{by correspondences} extending the given action of $\BG_m\,$ (the precise meaning of these words is explained in Subsect.~\ref{ss:What we want}).

 
\ssec{The structure of this Appendix}
In Subsections~\ref{ss:corresp}-\ref{ss:P_M} we give some categorical preliminaries. In 
Subsect.~\ref{ss:unsatisfactory} we formulate \propref{p:unsatisfactory}, which gives a very understandable but unsatisfactory answer to the question posed in Subsect.~\ref{ss:the goal}. In Subsect.~\ref{ss:to fix} we explain what has to be changed. After more categorical preliminaries, we formulate in Subsect.~\ref{ss:What we want} the answer to the question from Subsect.~\ref{ss:the goal}. 

Subsections~\ref{ss:Tw}-\ref{ss:openness} are devoted to the actual construction of the functor promised in Subsect.~\ref{ss:What we want}.  Let us note that the family of hyperbolas $\tau_1\tau_2=t$ appears in this construction for purely categorical reasons, see 
formula~\eqref{e:again hyperbolas}.

The categorical structure described in  Subsect.~\ref{ss:What we want} is slightly reformulated in Subsect.~\ref{ss:Reinterpretation}. In the remaining part of  Appendix~\ref{s: 2-categorical framework} we briefly explain how this structure can be used to prove the main result of \cite{DrGa1}. (Main idea: the lax action of $\BA^1$ on $Z$ by correspondences induces an action of $\Dmod (\BA^1 )$ on $\Dmod (Z )$, where $\Dmod$ stands for the DG categoryy of D-modules.)

\ssec{The 2-category of correspondences}   \label{ss:corresp}
\sssec{2-categorical conventions}
The word ``2-category"  is always understood in the ``weak" sense (as opposed to the ``strict" one).

Given 2-categories $\CC$ and $\CCD$, the notion of functor $F:\CC\to \CCD$ is also understood in the
``weak" sense.  That is, given 1-morphisms $f:c\to c'$ and $g:c'\to c''$ in $\CC$ one does not require that
$F (gf)=F(g)F(f)$; instead, one specifies 2-isomorphisms 
\begin{equation}  \label{e:2-functor}
F(g)F(f) \to F (gf)
\end{equation}
satisfying the usual coherence condition.

A 1-morphism $f:c_1\to c_2$ is said to be an \emph{isomorphism} if there exists a 1-morphism $g:c_2\to c_1$ such that $g\circ f$ is 2-isomorphic to $\id_{c_1}$ and $f\circ g$ is 2-isomorphic to $\id_{c_2}\,$.

\sssec{The 2-category $\Corr (\CC)$}   \label{sss:Corr (C)}
Let $\CC$ be a category in which finite products and fiber products exist. Then one defines the \emph{2-category of correspondences} $\Corr (\CC)$ as follows. 

$\Corr (\CC)$ has the same objects as $\CC$. For any objects $X_1,X_2$ in $\CC$  the category of 1-morphsisms from $X_1$ to $X_2$ is the category \emph{opposite} to the category of objects of $\CC$ over 
$X_1\times X_2$. (We will usually write a correspondence from $X_1$ to $X_2$ as a diagram 
$X_1\leftarrow X_{1,2}\to X_2$.) As usual, the composition of correspondences $X_1\leftarrow X_{1,2}\to X_2$ and $X_2\leftarrow X_{2,3}\to X_3$ is the diagram $$X_1\leftarrow  X_{1,2}\underset{X_2}\times X_{2,3} \to X_3$$ whose morphisms are defined in the obvious way.

\sssec{The 2-categories $\Corr (\Alg )$ and $\Corr^{open} (\Alg )$}    \label{sss:Corr Alg}
Let $\Alg$ denote the category of algebraic $k$-spaces of finite type.
Applying to it the construction from Subsect.~\ref{sss:Corr (C)}, one gets the 2-category $\Corr (\Alg )$.

Define a 2-subcategory $\Corr^{open} (\Alg )\subset \Corr (\Alg )$ as follows:
it has the same objects and 1-morphisms as $\Corr (\Alg )$, and its 2-morphisms are those 2-morphsisms in 
$\Corr (\Alg )$ which are open embeddings.

\ssec{The category $\P_M$}  \label{ss:P_M}
Let $M$ be a monoid. Assume that $M$ has a zero, i.e., an element $0\in M$ such that $0\cdot m=m\cdot 0=0$ for all $m\in M$ (clearly 0 is unique if it exists).

Then we define a category $\P_M$ as follows. It has two objects, denoted by $\bb$ and $\bs$ (``b" stands for ``big", ``s" stands for ``small"). The monoid $\End (\bb)$ equals $M$. The monoid $\End (\bs)$ has only one element (namely, $\id_{\bs}$).
There is a unique morphism $\bs\to\bb$, denoted by $\alpha^+$, and there is a unique morphism $\bb\to\bs$, denoted by $\alpha^-$. Finally, the composition
\[
\bb\overset{\alpha^-}\longrightarrow\bs\overset{\alpha^+}\longrightarrow\bb
\]
equals $0\in\End (\bb)$.

\ssec{A preliminary (and unsatisfactory) statement}   \label{ss:unsatisfactory}
The construction of Subsect.~\ref{ss:P_M} applies to the monoid $M=\BA^1 (S)$, where $S$ is any scheme and $\BA^1$ is the multiplicative monoid. In particular, we have the 2-category  $\P_{\BA^1 (k)}\,$. 

\begin{prop}   \label{p:unsatisfactory}
Let $Z$ be an affine $k$-scheme of finite type equipped with a $\BG_m$-action. Then there is a unique functor
\begin{equation}   \label{e:unsatisfactory}
\P_{\BA^1 (k)}\to\Corr^{open} (\Alg )
\end{equation}
such that

(i) the object $\bb\in\P_{\BA^1 (k)}$ goes to $Z\in\Corr^{open} (\Alg )$, and
the object $\bs\in\P_{\BA^1 (k)}$ goes to $Z^0\in\Corr^{open} (\Alg )$;

(ii) the morphisms $\alpha^+: \bs\to\bb$ and $\alpha^-:\bb\to\bs$ go to the correspondences
\begin{equation}   \label{e:+ and - correspondences}
Z^0\overset{q^+}\longleftarrow Z^+\overset{p^+}\longrightarrow Z \mbox{ and }
Z\overset{p^-}\longleftarrow Z^-\overset{q^-}\longrightarrow Z^0,
\end{equation}
respectively;

(iii) the action of $\BG_m (k)=\Aut (\bb )$ on $Z$ comes from the original action of $\BG_m$ on $Z$.

Moreover, for any $t\in\BA^1 (k)$ the morphism $t:\bb\to\bb$ goes to the correspondence
\[
Z\leftarrow \wt{Z}_t\to Z,
\]
given by the morphism $\wt{p}_t:\wt{Z}_t\to Z\times Z$ from Subsect.~\ref{sss:tilde p}.
\end{prop}

The proof is easy and left to the reader.
Let us only explain the reason for the affineness assumption in the above proposition. Note that in 
$\BA^1 (k)$ one has the relation $\alpha^-\circ\alpha^+=\id_{\bs}\,$. So we need a similar relation between the
correspondences \eqref{e:+ and - correspondences}. It is provided by the morphism 
\begin{equation}  \label{e:open emb}
Z^0\to Z^+\underset{Z}\times Z^-
\end{equation}
 from  \propref{p:Cartesian}, which is an \emph{isomorphism if $Z$ is affine} 
(see \remref{r:affine case}).

\ssec{What has to be changed}   \label{ss:to fix}
\sssec{Lax functors instead of true functors}
If $Z$ is not assumed affine then the map  \eqref{e:open emb} is an \emph{open embedding} but not necessarily an isomorphism. So we cannot expect to get a true functor $\P_{\BA^1 (k)}\to\Corr^{open} (\Alg )$ but only a \emph{lax functor} (this notion is recalled in Subsect.~\ref{ss:lax} below). 

\sssec{Pre-sheafification is necessary}   \label{sss:why fibered_cat}
In \propref{p:unsatisfactory} the correspondences $\wt{Z}_t$ are treated as if they exist separately for each 
$t\in\BA^1 (k)$. But in fact, they form a (non-flat) family $\wt{Z}$ over~$\BA^1$. To fix this, we will use the language of fibered categories. E.g., instead of the category $\P_{\BA^1 (k)}$ we will work with the fibered category formed by categories $\P_{\BA^1 (S)}$, where $S$ is an arbitrary $k$-scheme. We use the word ``pre-sheafification" for this procedure.

%

\ssec{Lax functors and lax actions}   \label{ss:lax}

\sssec{Lax functors}     \label{sss:lax}
 Let $\CC$ and $\CCD$ be 2-categories. In addition to the usual notion of functor 
$\CC\to \CCD$, there is a notion of  \emph{lax functor} $F:\CC\to \CCD$. The word ``lax" means that the 2-mor\-phisms~\eqref{e:2-functor} are \emph{not required} to be isomorphisms.

We also include the following two conditions in the definition of lax functor:

(i) if either $f$ or $g$ is a 1-isomorphism then \eqref{e:2-functor} is a 2-isomorphism\footnote{Following N.~Rozenblyum's advice, we include this condition in order to ensure that an equivalence $\CC'\iso\CC$ induces an equivalence between the 2-categories of lax functors $\CC\to \CCD$ and $\CC'\to \CCD$.};

(ii) if a 1-morphism $f$ is an isomorphism\footnote{By property (i), it suffices to require this if $f$ is an identity isomorphism.} then so is $F(f)$.

Combining (i) and (ii), we get a canonical 2-isomorphism $F(\id_c )\iso\id_{F(c)}$ for any $c\in \CC_1\, $.
(Thus a more appropriate name for what we call `` lax functor" would be ``unital lax functor".)

\sssec{Lax actions of monoids}  \label{sss:lax_monoids}
If $\CM$ is a monoidal category let $B\CM$ denote the 2-category with a single object $c$ such that the monoidal category of 1-endomorphisms of $c$ is $\CM$. 

Now let $\CCD$ be a 2-category and $d\in\CCD$. By a \emph{lax action} of $\CM$ on $d$ we mean a lax functor $F:B\CM\to\CCD$ such that $F(c)=d$ (as before, $c$ is the unique object of $B\CM$).

Any monoid can be considered as a (discrete) monoidal category. Thus we have the notion of lax action of a monoid on an object of a 2-category.

\ssec{Some fibered categories}  \label{ss:Some fibered}
The reason why we need fibered categories was explained in Subsect.~\ref{sss:why fibered_cat}.

Unless stated otherwise, ``fibered category" will mean ``category fibered over the category of $k$-schemes".
A functor between fibered categories is tacitly assumed to be over the category of $k$-schemes. The natural class of such functors is formed by \emph{Cartesian} functors (i.e., those taking Cartesian morphisms to Cartesian morphisms).

Let $\P_{\BA^1}$ denote the fibered category whose fiber over a $k$-scheme $S$ is $\P_{\BA^1 (S)}\,$.

Let $B\BG_m$ (resp. $B\BA^1$) denote the fibered category whose fiber over a $k$-scheme $S$ is the category with a single object whose monoid of endomorphisms is $\BG_m (S)$ (resp. $\BA^1 (S)$. Note that the monoid of endomorphsms of the object
$\bb\in\P_{\BA^1 (S)}$ equals $\BA^1 (S)$, so we have $$B\BG_m\subset B\BA^1\subset\P_{\BA^1}\,.$$

\begin{rem}   \label{r:enriched}
Any category $\CA$ enriched over $k$-schemes defines a fibered category (its fiber over $S$ is obtained
from $\CA$ by replacing the scheme of morphisms between any two objects of $\CA$ with the set of $S$-points of this scheme) . The fibered categories $\P_{\BA^1}$, $B\BG_m$, and $B\BA^1$ are of this type; the fibered categories defined below are not.
\end{rem}

Let $\Sp$ (resp. $\AAlg$) denote the fibered category whose fiber over a $k$-scheme $S$ is the category of spaces\footnote{Recall that  by a $k$-space we mean a contravariant functor $F$ from the category of 
$k$-schemes to that of sets which is a sheaf for the fpqc topology.} (resp. algebraic spaces of finite presentation) over $S$.

Let $\CORR (\AAlg )$ denote the fibered 2-category\footnote{Let $\CC$ be a 2-category and $\CCD$ a 1-category. A functor $F:\CC\to\CCD$ is said to be a \emph{fibration} if for any $c\in\CC$ and $d\in \CCD$ any morphism $f:d\to F(c)$ can be lifted to a 1-morphism $\tilde f:c'\to c$ which is \emph{Cartesian} in the following sense: for any $c''\in\CC$ and any $g:F(c'')\to d$, composition with $\tilde f$ defines an equivalence 
$\Mor_g (c'',c'))\to \Mor_{g\circ f} (c'',c))$, where $\Mor_g (c'',c')$ is the category of 1-morphisms over $g$.} 
whose fiber over a $k$-scheme $S$ is the 2-category of correspondences (see Subsect.~\ref{sss:Corr (C)}) corresponding to the category of algebraic spaces of finite presentation over $S$. Similarly, we have the fibered 2-categories $\CORR^{open} (\AAlg )$,
$\CORR^{open} (\Sp )$, and $\CORR^{open} (\Sp )$.

\ssec{What we want to construct}    \label{ss:What we want}

Let $Z$ be an algebraic $k$-space  of finite type equipped with a $\BG_m$-action. Such $Z$ is the same as a Cartesian functor
\begin{equation}  \label{e:BG_m_to}
B\BG_m\to\AAlg.
\end{equation}
Recall that $B\BG_m\subset\P_{\BA^1}$ (see Subsect.~\ref{ss:Some fibered}).

In Subsections~\ref{ss:without openness}-\ref{ss:openness}  we will define a canonical 
extension\footnote{The extension constructed below is, in some sense, universal (see \propref{p:Phi_Z} for the precise statement).} 
of  \eqref{e:BG_m_to} to a Cartesian lax functor 
\begin{equation}  \label{e:PBA_to}
\P_{\BA^1}\to\CORR^{open} (\AAlg )
\end{equation}
The restriction of \eqref{e:PBA_to} to $B\BA^1\subset \P_{\BA^1}$ is the ``lax action of $\BA^1$ on $Z$ by correspondences" mentioned in Subsect.~\ref{ss:the goal}.

%
%

\ssec{Twisted arrow category}   \label{ss:Tw}
In this subsection we recall a standard categorical construction, which will be used to construct \eqref{e:PBA_to}.


\sssec{Definition of $\Tw (\CC )$}
Given any category $\CC$ one defines its \emph{twisted arrow 
category}\footnote{Following the convention of \cite[Construction 3.3.5]{Lur3}, 
we use the notation $\Tw (\CC )$ for the category traditionally denoted by $\Tw (\CC )^{op}$. This allows to simplify the formulations in Subsect.~\ref{sss:Tw-Corr}.},
which will be denoted by $\Tw (\CC )$. The objects of $\Tw (\CC )$ are arrows $c_1\to c_2$ in $\CC$. Given two such arrows $f:c_1\to c_2$ and $f':c'_1\to c'_2\,$, a $\Tw (\CC )$-morphism from $f$ to $f'$ is defined to be a commutative square
\[
\CD
c_1 @>{u}>>  c'_1 \\
@V{f}VV   @VV{f'}V   \\ 
c_2@<{v}<<  c'_2
\endCD
\]
or equivalently, a factorization of $f$ as $v\circ f'\circ u$. Composition of $\Tw (\CC )$-morphisms corresponds to juxtaposition of squares:
\[
\CD
c_1 @>>>  c'_1 @>>>  c''_1\\
@V{f}VV   @VV{f'}V   @VV{f''}V \\ 
c_2@<<<  c'_2 @<<<  c''_2
\endCD
\]

\begin{rem}      \label{r:groupoids}
One gets a canonical functor $\Tw (\CC )\to\CC$ by associating to an arrow $f:c_1\to c_2$ the object $c_1$. If 
$\CC$ is a groupoid this functor is an equivalence.
\end{rem}

\begin{rem}  
Any morphism $f:c_1\to c_2$ in $\CC$ defines a diagram
\begin{equation}  \label{e:diagr in Tw}
\id_{c_1}\leftarrow f\to \id_{c_2}
\end{equation}
in $\Tw (\CC )$: namely,  the morphism $f\to\id_{c_1}$ (resp. $f\to \id_{c_2}$) corresponds to the factorization of $f$ as $f\circ \id_{c_1}\circ\id_{c_1}$ (resp. as $\id_{c_2}\circ\id_{c_2}\circ f$).
\end{rem}

\begin{rem}
(ii) Any composable morphisms $f:c_1\to c_2$ and $g:c_2\to c_3$ in $\CC$ define a commutative diagram
\begin{equation} \label{e:square in Tw}
\CD
g\circ f @>>>  g \\
@VVV   @VVV   \\
f@>>>   \id_{c_2}
\endCD
\end{equation}
in $\Tw (\CC )$; in this diagram the morphisms  $f\to \id_{c_2}\leftarrow g$ are as in the previous remark, and the morphism $g\circ f\to f$ (resp. $g\circ f\to g$) corresponds to the factorization of $g\circ f$ as
$g\circ f\circ\id_{c_1}$ (resp. as $\id_{c_2}\circ g\circ f$).

\end{rem}

\sssec{``Adjunction" between $\Tw$ and $\Corr$}  \label{sss:Tw-Corr}
Let $\CC$ and $\CCD$ be categories. Assume that $\CCD$ has finite products and fiber products, so the category $\Corr(\CCD )$ from Subsect.~\ref{sss:Corr (C)} is defined.   Then \emph{the groupoid\footnote{
   In other words, we remove non-invertible morphisms between lax functors $\CC\to \Corr(\CCD )$ (after which 2-morphisms are automatically invertible) and similarly, remove non-invertible  morphisms between functors
  $\Tw(\CC)\to\CCD$. Replacing the categories of functors by the corresponding groupoids is necessary already if $\CC$ is the point category (i.e., the category with one object and one morphism).} of  lax functors
   $\CC\to \Corr(\CCD )$
is canonically equivalent to the groupoid of functors
  $\Tw(\CC)\to\CCD$. } (I learned this fact from N.~Rozenblyum.)

Let us explain the construction in one direction (the only one we need). Given a functor $F: \Tw(\CC)\to\CCD$,  one constructs the following lax functor $\CC\to \Corr(\CCD )$. To an object $c\in \CC$ one associates the object $F(\id_c)\in\CCD$. To a morphism $f:c_1\to c_2$ in $\CC$ one associates the correspondence 
$F(\id_{c_1})\leftarrow F(f) \to F(\id_{c_2})$
obtained by applying $F$ to diagram \eqref{e:diagr in Tw}. Finally, to composable morphisms $f:c_1\to c_2$ and $g:c_2\to c_3$ in $\CC$, one associates the morphism $F(g\circ f)\to F(g)\underset{F(c_2)}\times F(f)$ obtained by applying $F$ to the commutative diagram \eqref{e:square in Tw}.

\ssec{The Cartesian lax functor $\P_{\BA^1}\to\CORR (\AAlg )$}   \label{ss:without openness}
Let $\TTw (\P_{\BA^1})$ denote the fibered category whose fiber over a $k$-scheme $S$ equals 
$\Tw (\P_{\BA^1(S)})$.\footnote{The fibered category $\P_{\BA^1}$ comes from a category enriched over $k$-schemes, see \remref{r:enriched}. Accordingly, the fibered category $\TTw (\P_{\BA^1})$ comes from a category \emph{internal to} the category of $k$-schemes (i.e., \emph{both} objects and morphisms form $k$-schemes rather than abstract sets).} Similarly, one defines  the fibered category $\TTw (B\BG_m )$.

Given an algebraic $k$-space of finite type $Z$ equipped with a $\BG_m$-action, we want to
construct a Cartesian lax functor 
\begin{equation}   \label{e:2PBA_to}
\Theta_Z :\P_{\BA^1}\to\CORR (\AAlg ).
\end{equation}
By Subsect.~\ref{sss:Tw-Corr}, this amounts to constructing a Cartesian functor 
$\Phi_Z:\TTw (\P_{\BA^1})\to\AAlg$. 

Recall that $\P_{\BA^1}\supset B\BG_m\,$. Our $Z$ gives a Cartesian functor 
$B\BG_m\to\AAlg$, which is the same as a Cartesian functor $\Psi_Z: \TTw (B\BG_m )\to\AAlg$ 
(see \remref{r:groupoids}).
We want $\Phi_Z$ to be an extension of $\Psi_Z\,$.

The definition of $\Phi_Z$ is given in part (i) of the next proposition.

\begin{prop}   \label{p:Phi_Z}
(i) The category of Cartesian functors $\Phi:\TTw (\P_{\BA^1})\to\Sp$ equipped with a morphism
$\Phi |_{\TTw (B\BG_m )}\to\Psi_Z$ has a final object, denoted by $\Phi_Z\,$.

(ii) The morphism $(\Phi_Z) |_{\TTw (B\BG_m )}\to\Psi_Z$ is an isomorphism.

(iii) The Cartesian functor $\Phi_Z:\TTw (\P_{\BA^1})\to\Sp$ factors though $\AAlg\subset\Sp$.

(iv) Let $t\in\BA^1(S)$. Consider $t$ as an object of $\Tw (\P_{\BA^1(S)})$. Then $\Phi_Z$ takes  
$t$ to $\wt{Z}_t:=\wt{Z}\underset{\BA^1}\times S$.

(v) $\Phi_Z$ takes the object $\alpha^{\pm}$ of $\Tw (\P_{\BA^1(S)})$ to $Z^{\pm}\times S$.

(vi) $\Phi_Z$ takes the object $\id_{\bs}$ of $\Tw (\P_{\BA^1(S)})$ to $Z^0\times S$ (recall that $\bs$ denotes the ``small" object of $\P_{\BA^1(S)})$.
\end{prop}

\begin{rem}   \label{r:similar to induced}
Statements (i)-(ii) are quite similar to the following well known facts:
if $\CA$ and $\CB$ are categories and $F:\CA\to\CB$ is a functor then the ``restriction" functor
\[
F^*:\{\mbox{Functors } \CB\to \mbox{Sets}\}\to \{\mbox{Functors } \CA\to \mbox{Sets}\}
\]
admits a right adjoint called \emph{right Kan extension along $F$} or the 
\emph{coinduction functor}\footnote{Here is a well known example of a coinduction functor. Suppose that each of the categories $\CA$ and $\CB$ has one object, so $F$ corresponds to a homomorphism of monoids $M\to M'$. Then coinduction is the functor $\{M\mbox{-sets}\}\to\{M'\mbox{-sets}\}$ that takes an $M$-set $Y$ to the $M'$-set $\Maps^M(M',Y)$.}; moreover, if $F$ is fully faithful then applying first coinduction and then restriction one gets the identity. 
\end{rem}

\begin{rem}  
In the situation of statements (i)-(ii) the role of $F$ is played by the functor 
$\TTw (B\BG_m )\to\TTw (\P_{\BA^1})$, which is clearly fully faithful.
\end{rem}


\begin{proof}
Because of the above remarks, instead of giving a detailed proof of statements (i)-(ii), we will only explain the construction of the functor $\Phi_Z$ (which is quite similar to the construction of the coinduced functor in the situation of \remref{r:similar to induced}).
For any object $f\in\Tw (\P_{\BA^1(S)})$ 
the $S$-space $\Phi_Z (f)$ is defined as follows. 
Let  $\bX_f$ denote the $S$-space of $\TTw (\P_{\BA^1})$-morphisms\footnote{By this we mean the functor that to a
scheme $S'$ equipped with a morphism $\varphi :S'\to S$ associates the set of $\Tw (\P_{\BA^1(S')})$-morphisms $\varphi^*(f)\to\id_{\bb}\,$.} $f\to\id_{\bb}$ (where $\bb\in\P_{\BA^1(S)})$ is the ``big" object). The group $\BG_m$ acts on $\bX_f$ (because it acts on 
$\id_{\bb}\in \Tw (\P_{\BA^1})$), and the $S$-space $\Phi_Z (f)$ is 
defined to be the functor
\begin{equation}   \label{e:Phi_Z (f)}
S'\mapsto\gMaps (\bX_f\underset{S}\times S',Z)
\end{equation}
on the category of $S$-schemes.


Now let us prove statement (iv). Let $f:\bb\to\bb$ be the endomorphism corresponding to $t\in\BA^1 (S)$.
By the definition of twisted arrow category, a $\Tw (\P_{\BA^1(S)})$-morphism $f\to\id_{\bb}$ is the same as a factorization $f=f_2\circ f_1\,$, $f_i\in\End (\bb )$, or equivalently, a factorization
\begin{equation}   \label{e:again hyperbolas}
t=\tau_1\tau_2\,, \quad\quad\tau_i\in\BA^1 (S).
\end{equation}
So the space $\bX_f$ defined above identifies with the fiber product $\BX\underset{\BA^1}\times S$ with respect to $t:S\to\BA^1$, which was used in 
Subsect.~\ref{ss:tilde Z}. Statement (iv) follows.

Statements (v)-(vi) are proved similarly to (iv). Statement (iii) follows from (iv)-(vi).
\end{proof}


Thus we have constructed  the Cartesian lax functor $\Theta_Z :\P_{\BA^1}\to\CORR (\AAlg )$. It takes
the objects $\bb,\bs\in\P_{\BA^1 (k)}$ to $Z,Z^0\in\Corr^{open} (\Alg )$, respectively.
The morphisms $\alpha^+: \bs\to\bb$ and $\alpha^-:\bb\to\bs$ from $\P_{\BA^1 (k)}$ 
go to the correspondences
\begin{equation}   
Z^0\overset{q^+}\longleftarrow Z^+\overset{p^+}\longrightarrow Z \mbox{ and }
Z\overset{p^-}\longleftarrow Z^-\overset{q^-}\longrightarrow Z^0,
\end{equation}
respectively.
For any $t\in\BA^1 (S)$ the morphism $t:\bb\to\bb$ in $\P_{\BA^1 (S)}$ goes to the correspondence
obtained from the morphism $\wt{p}:\wt{Z}\to\BA^1\times Z\times Z$ by base change with respect to
$t:S\to\BA^1$.

\ssec{Openness}  \label{ss:openness}
The structure of lax functor on
$$\Theta_Z :\P_{\BA^1}\to\CORR(\AAlg )$$
is given by 2-morphisms
\[
\gamma_{g,f}: \Theta_Z (g)\circ \Theta_Z (f)\to  \Theta_Z (g\circ f),
\]
where $f$ and $g$ are composable morphisms in $\P_{\BA^1 (S)}\,$. E.g., if $f$ and $g$ are endomorphisms of $\bb\in\P_{\BA^1 (S)}$ then $\gamma_{g,f}$ comes from the morphism \eqref{e:open emb for n=2} constructed in Subsect.~\ref{sss:relation anti-action} using the ``anti-action".

\begin{prop}   \label{p:openness}
(i) The Cartesian lax functor $\Theta_Z :\P_{\BA^1}\to\CORR(\AAlg )$ constructed in Subsect.~\ref{ss:without openness} factors through $\CORR^{open}(\AAlg )$.

(ii) The morphism $\gamma_{\alpha^+,\alpha^-}$ is an isomorphism.

(iii) If every $\BG_m$-equivariant map $\BP^1\otimes_k\bar k\to Z\otimes_k\bar k$ is constant  (in particular, if $Z$ is affine) then the lax functor $\Theta_Z$ is a true functor.
\end{prop}

\begin{proof}
Statement (ii) is a reformulation of \propref{p:tilde Z_0}.

By \propref{p:Cartesian}, $\gamma_{\alpha^-,\alpha^+}$ is an open embedding. In the situation of (iii) it is an isomorphism by formula \eqref{e:fiberprod}.

By \propref{p:2open embeddings}, the morphisms 
$$\gamma_{t,\alpha^+} \mbox{ and }\gamma_{\alpha^-,t}\, , \quad\quad t\in\BA^1 (S)$$
are open embeddings. In the situation of (iii) they are isomorphisms by  \remref{r:gamma is an iso}.

By \propref{p:2open embedding}, the morphisms 
$$\gamma_{t_1,t_2}\, , \quad\quad t_1,t_2\in\BA^1 (S)$$
are open embeddings, and  in the situation of (iii)  they are isomorphisms.
\end{proof}


\ssec{Reinterpretation of the lax functor $\Theta_Z\,$}  \label{ss:Reinterpretation}

\sssec{Generalities}   \label{sss:pattern}
Suppose that a monoidal category $\CM$ acts on a 2-category $\CC$. If $M$ is a monoid in $\CM$ and $c$ is an object of $\CC$ then there is a notion of action (resp. lax action) of $M$ on~$c$.

A monoid in $\CM$ is the same as an $\CM$-enriched category with one object. If $\CA$ is an \emph{arbitrary} $\CM$-enriched category there is a similar notion of  functor (resp. lax functor) $F:\CA\to\CC$. This means that we have a map $F:\Ob\CA\to\Ob\CC$, for each $a_1,a_2\in\CA$ a 1-morphism 
\begin{equation}   \label{e:enriched functor}
\Hom_\CA(a_1,a_2)\otimes F(a_1) \to F(a_2)
\end{equation}
in $\CC$, and for each $a_1,a_2,a_3\in\CA$ a 2-isomorphism (resp. 2-morphism) between the two 1-morpisms
$$\Hom_\CA(a_1,a_2)\otimes\Hom_\CA(a_2,a_3)\otimes F(a_1) \to F(a_3)$$
obtained from \eqref{e:enriched functor}. The compatibility conditions are just as in the non-enriched case.

The notion of (lax) functor from  an $\CM$-enriched category $\CA$ to an $\CM$-enriched category $\CC$ exists even if $\CM$  is a monoidal 
$\infty$-category, $\CA$ is an $\infty$-category, and $\CC$ is an $(\infty ,2)$-category.\footnote{Recall that an $(n,m)$-category is an $n$-category such that all its $r$-morphisms are invertible for $r>m$.}


\sssec{Example}   \label{sss:examplepattern}
Take $\CM=\Alg$, $\CC=\Corr^{open} (\Alg )$ and define the monoidal structure on $\CM$ and the action of $\CM$ on $\CC$ via Cartesian product. Then we get the notion of lax functor from an $\Alg$-enriched category to $\Corr^{open} (\Alg )$.

\sssec{Reinterpretaion of $\Theta_Z\,$}  
From now on we consider $\P_{\BA^1}$ as a category enriched over $k$-schemes (and therefore over $\Alg$) rather than as a category fibered over $k$-schemes. Yoneda's lemma allows to reinterpret the lax functor $\Theta_Z: \P_{\BA^1}\to\CORR^{open}(\AAlg )$ constructed in Subsections~\ref{ss:without openness}-\ref{ss:openness} as a lax functor $\Theta_Z: \P_{\BA^1}\to\Corr^{open}(\Alg )$ in the sense of Subsectios~\ref{sss:pattern}-\ref{sss:examplepattern}. From now on we will use this reinterpretation.

\sssec{Passing to quotients by $\BG_m\,$}  \label{sss:passing to quot}
Let $\St$ denote the $(2,1)$-category formed by those algebraic $k$-stacks of finite type whose field-valued points have affine automorphism groups\footnote{The condition on automorphism groups ensures that the theory of D-modules on such stacks is ``nice" (assuming that $k$ has characteristic 0). This will be important in Subsect.~\ref{sss:Dmod}.}.  
One defines the $(3,2)$-category $\Corr^{open}  (\St )$ similarly to $\Corr^{open}  (\Alg )$ (see Subsect.~\ref{sss:Corr Alg}). 
Define the monoidal structure on $\St$ and the action of $\St$ on $\Corr^{open} (\St )$ via Cartesian product. Then by Subsect.~\ref{sss:pattern}, we get the notion of lax action of an algebraic monoid on an object of $\Corr^{open} (\St )$ and a notion of lax functor from a $\St$-enriched category to $\Corr^{open} (\St )$.

Note that $\P_{\BA^1}$ can be considered as a category enriched over the category of $\BG_m$-schemes 
(this is a shorthand for ``schemes with $\BG_m$-action"). Now let $\P_{\BA^1}/B\BG_m$ denote the $\St$-enriched category obtained from 
$\P_{\BA^1}$ by replacing each $\BG_m$-scheme of morphisms with its quotient by $\BG_m\,$.\footnote{Motivation of the notation:  enrichment over $\BG_m$-schemes can be interpreted as an action of the classifying stack $B\BG_m$ on $\P_{\BA^1}$.} Thus $\P_{\BA^1}/B\BG_m$ has two objects $\bs,\bb$, and
\[
\on{Mor} (\bb ,\bb )=\BA^1/\BG_m\, , \quad \on{Mor} (\bs ,\bs )=\on{Mor} (\bs ,\bb )=\on{Mor} (\bb ,\bs )=B\BG_m\, .
\]

The lax functor $\Theta_Z: \P_{\BA^1}\to\Corr^{open}(\Alg )$ induces a lax functor 
\begin{equation}   \label{e:wtTheta}
\wt{\Theta}_Z: \P_{\BA^1}/B\BG_m\to\Corr^{open}(\St ).
\end{equation} 
such that $\wt{\Theta}_Z (\bb ):=Z/\BG_m$ and $\wt{\Theta}_Z (\bs ):=Z^0/\BG_m=Z^0\times B\BG_m\,$. Rather than giving a formal definition\footnote{A formal definition can be given using the following device (N.~Rozenblyum). Let $\CM$ be the category of pairs $(G,X)$, where $X\in\Alg$ and $G$ is a group scheme acting on $X$; a morphism $(G,X)\to (G',X')$ is given by a homomorphism $G\to G'$ and a $G$-equivariant map $X\to X'$. Define a monoidal structure on $\CM$ via Cartesian product. Upgrade $\P_{\BA^1}$ to an $\CM$-enriched category and $\wt{\Theta}_Z$ to a lax functor $\P_{\BA^1}\to\Corr (\CM )$. Finally, use the ``passing to quotient" functor $\CM\to\St$.} of 
$\wt{\Theta}_Z\,$, let us define the lax action of the monoidal stack $\BA^1/\BG_m$ on $\wt{\Theta}_Z (\bb )=Z/\BG_m\in\Corr^{open}(\St )$ (this lax action is a part of $\wt{\Theta}_Z$). Recall that the lax action of the multiplicative monoid $\BA^1$ on $Z\in\Corr^{open}(\Alg)$ is given by the morphism $\wt{p}:\wt{Z}\to \BA^1\times Z\times Z$. The lax action of $\BA^1/\BG_m$ on $Z/\BG_m$ is given by the morphism of stacks 
$$\wt{Z}/(\BG_m\times\BG_m)\to (\BA^1/\BG_m)\times (Z/\BG_m)\times (Z/\BG_m)$$
induced by $\wt{p}$.


\ssec{Using the lax functor  to prove Braden's theorem}   \label{ss:Using lax functor}
The lax functor \eqref{e:wtTheta} is the geometric datum, which is used in \cite{DrGa1} (without being mentioned quite explicitly) to prove a cohomological theorem due to T.~Braden. Here we give a brief outline of the proof.
Let us note that a similar approach can be used to prove the main result of  \cite{DrGa2}.\footnote{In the situation of \cite{DrGa2} the geometric datum is a certain lax functor $\P_{\BA^1}/B\BG_m\to\Corr^{open}(\St' )$ such that $\bb\mapsto\Bun_G\,$. Here $\St'$ is the $(2,1)$-category of algebraic stacks \emph{locally} of finite type over $k$, $G$ is a reductive group, and $\Bun_G$ is the stack of $G$-bundles on a smooth projective curve.} 


To fix the ideas, we consider the de Rham version of Braden's theorem. So the ground field $k$ will be supposed to have characteristic 0, and the role of sheaves will be played by D-modules.

\sssec{Formulation of Braden's theorem}
The theorem (as formulated in \cite[Sect.~3.4]{DrGa1}) says that the functor 
\begin{equation}   \label{e:left_Braden}
(q^-)_\bullet\circ (p^-)^!:\Dmod(Z/\BG_m)\to\Dmod(Z^0/\BG_m)
\end{equation}
 is left adjoint to 
\begin{equation}   \label{e:right_Braden}
(p^+)_\bullet\circ (q^+)^!:\Dmod(Z^0/\BG_m)\to\Dmod(Z/\BG_m).
\end{equation}
Here $Z/\BG_m$ and $Z^0/\BG_m$ denote the quotient \emph{stacks} (the action of $\BG_m$ on $Z^0$ being trivial), $\Dmod$ stands for the DG category of D-modules, $p^\pm$ and $q^\pm$ have the usual meaning (see Subsect.~\ref{ss:the goal}), and $(q^-)_\bullet$ stands for the de Rham direct image with respect to $q^- :Z/\BG_m\to Z^0/\BG_m\,$.

\sssec{The 2-category $\P_{1\to 0}\,$ and adjoint pairs}
Let us describe a general way to construct adjoint pairs of functors.

Recall that to any monoid $M$ having a zero we associated in Subsect.~\ref{ss:P_M} a category $\P_M$ with two objects, denoted by $\bb$ and $\bs$. In particular, we can apply this construction to the monoid $\{ 1, 0\}$ (with the operation of multiplication) and get a category $\P_{\{ 1, 0\}}\,$.

Now let $\P_{1\to 0}$ denote the following (strict) 2-category: $\P_{1\to 0}$ has the same objects and 1-morphisms as $\P_{\{ 1, 0\}}\,$, it has a 2-morphism from the 1-morphism $1:\bb\to\bb$ to the 1-morphism $0:\bb\to\bb$ and no other 2-morphisms except the identities.

The next remark explains how to use the 2-category $\P_{1\to 0}$ to construct adjoint pairs of 1-morphisms in any 2-category.

\begin{rem}   \label{r:lax_to_adj}
Let $\CC$ be a 2-category, and $F:\P_{1\to 0}\to\CC$ a lax functor in the sense of Subsect.~\ref{sss:lax}. In addition, assume that the 2-morphism 
\begin{equation}   \label{e:required_invertible}
F(\alpha^+)F(\alpha^-)\to F(\alpha^+\alpha^-)=F(0)
\end{equation}
(which is a part of the lax functor structure on $F$) is invertible. Then $F(\alpha^-):F(\bb )\to F(\bs )$ is left adjoint to
$F(\alpha^+):F(\bs )\to F(\bb )$; namely, the counit 
$$F(\alpha^-)F(\alpha^+)\to F(\alpha^-\alpha^+)=\id_{F(\bs )}$$ comes from the lax functor structure on $F$, and the unit $\id_{F(\bb )}\to F(\alpha^+)F(\alpha^-)$ is the composition of the morphism
$\id_{F(\bb )}=F(1)\to F(0)$ and the isomorphism $F(0)\iso F(\alpha^+)F(\alpha^-)$ inverse to \eqref{e:required_invertible}. Moreover, a lax functor $F:\P_{1\to 0}\to\CC$ such that \eqref{e:required_invertible} is invertible is \emph{the same} as an adjoint pair of 1-morphisms in $\CC$.
\end{rem}

\begin{rem} \label{r:Nick_lax_to_adj}
N.~Rozenblyum showed that the statements from the previous remark remain valid if $\CC$ is an $(\infty ,2)$-category
 rather than a usual 2-category.

\end{rem}

\sssec{Plan of the proof of Braden's theorem}  \label{sss: functors to construct}

Just as in Subsect.~\ref{sss:passing to quot}, we consider $\P_{\BA^1}$ as a category enriched over $\BG_m$-schemes.
Let $\CA'$ (resp.~$\CA$) denote the ($\infty$,2)-category obtained from $\P_{\BA^1}$ by replacing in $\P_{\BA^1}$ each $\BG_m$-scheme of morphisms $Y$ with the DG category $\Dmod (Y)$ (resp. with  $\Dmod (Y/\BG_m )$). Let $\DGCat_{\on{cont}}$ be the $(\infty ,2)$-category whose objects are cocomplete DG categories (i.e., ones that contain arbitrary direct sums, or equivalently, colimits) and whose 1-morphisms are continuous functors (i.e., exact functors that commute with arbitrary direct sums, or equivalently all colimits). Given an ($\infty$,2)-category $\CCD$, let $\CCD^{naive}$ denote its homotopy 2-category (i.e., $\CCD^{naive}$ has the same objects and 1-morphisms as $\CCD$, and 2-morphisms of $\CCD^{naive}$ are homotopy classes of the 2-morphisms of $\CCD$). 


We are going to construct a functor
\begin{equation}   \label{e:PtoA}
\P_{1\to 0}\to\CA^{naive}
\end{equation}
and a lax functor
\begin{equation}   \label{e:AtoDGCat}
G:\CA\to\DGCat_{\on{cont}}\, .
\end{equation}
Let $F:\P_{1\to 0}\to \DGCat_{\on{cont}}^{naive}$ be the composition of \eqref{e:PtoA} and the functor $$G^{naive}:\CA^{naive}\to\DGCat_{\on{cont}}^{naive}.$$ It will be clear from the construction that the 2-morphism \eqref{e:required_invertible} is invertible. So by Remark~\ref{r:lax_to_adj},
we get an adjoint pair of 1-morphisms in $\DGCat_{\on{cont}}^{naive}$ and then (by the general nonsense from \cite[Sect.~3.2]{Lur1}) in 
$\DGCat_{\on{cont}}$ itself.\footnote{The only reason why we do not construct the adjoint pair directly in $\DGCat_{\on{cont}}$ is that the result mentioned in  \remref{r:Nick_lax_to_adj} is not published.} It will be clear from the construction that these 1-morphisms are the functors~\eqref{e:left_Braden}-\eqref{e:right_Braden}.

\sssec{The functor $G:\CA\to\DGCat_{\on{cont}}$} \label{sss:passing to Dmod}
The lax functor $\wt{\Theta}_Z: \P_{\BA^1}/B\BG_m\to\Corr^{open}(\St )$ from Subsect.~\ref{sss:passing to quot}
induces a lax functor
$G:\CA\to\DGCat_{\on{cont}}$ such that 
$$G(\bs )=\Dmod (Z^0/\BG_m), \quad G(\bb )=\Dmod (Z/\BG_m),$$
$$G(\alpha^-)=(q^-)_\bullet\circ (p^-)^!,\quad G(\alpha^+)=(p^+)_\bullet\circ (q^+)^! .$$
We will explain more details in Subsect.~\ref{ss:D-module details} below.

\sssec{Constructing the functor \eqref{e:2PtoA}} 
Before constructing \eqref{e:PtoA}, let us construct a functor
\begin{equation}   \label{e:2PtoA}
\P_{\{ 1, 0\}}\to\CA\,.
\end{equation}
We define  \eqref{e:2PtoA} to be the composition of the following functors
\begin{equation}   \label{e:three_functors}
\P_{\{ 1, 0\}}\mono\P_{\BA^1 (k)}\to\CA'\to\CA\, .
\end{equation}
The first functor comes from the embedding $\{ 1, 0\}\mono \BA^1 (k)$. The functor 
$\P_{\BA^1 (k)}\to\CA'$ comes from the fact that for any $k$-scheme $Y$ of finite type one has the canonical map
\[
Y(k)\to\Dmod (Y), \quad y\mapsto \delta_y\, ,
\]
where $\delta_y\in\Dmod (Y)$ is the D-module of delta-functions at $y$. The functor $\CA'\to\CA$ comes from the fact that for any $k$-scheme of finite type equipped with a $\BG_m$-action, one has the de Rham direct image functor $\Dmod (Y)\to\Dmod (Y/\BG_m)$.

Now to construct  the functor \eqref{e:PtoA} it remains to prove the following lemma.\footnote{This lemma (and the similar Lemma C.6.1 from \cite{DrGa3}) is closely related to the \emph{specialization map} from \cite[Sect.~4.1]{DrGa1}.}
\begin{lem}
The functor $\P_{\{ 1, 0\}}\to\CA^{naive}$ corresponding to \eqref{e:2PtoA} uniquely extends to a functor $\P_{1\to 0}\to\CA^{naive}\,$.
\end{lem}

\begin{proof}

Let us start with a general remark. For any 2-category $\CC$, a functor $F:\P_{\{ 1, 0\}}\to\CC$ is the same as the following data: objects $F(\bs), F(\bb )\in\CC$, 1-morphisms
$F(\alpha^+): F(\bs )\to F(\bb )$, $F(\alpha^-): F(\bb )\to F(\bs )$, and a 2-isomorphism
\begin{equation} \label{e:potential_counit}
F(\alpha^-)F(\alpha^+)\iso\id_{F(\bs )}.
\end{equation}
Such a functor extends to a functor $\P_{1\to 0}\to\CC$ if and only if \eqref{e:potential_counit} is a counit of an adjunction between $F(\alpha^-)$ and $F(\alpha^+)$; in this case the extension is unique (because an adjunction is uniquely determined by its counit).  

Now let us construct the unit of the adjunction in the situation where $\CC=\CA^{naive}$ and $F$ is the functor  $\P_{\{ 1, 0\}}\to\CA^{naive}$ corresponding to \eqref{e:2PtoA}. Unraveling the definitions, we see that the problem is to construct in $\Dmod (\BA^1/\BG_m)$ a morphism $f:1_\bullet (k)\to 0_\bullet (k)$ with a certain property (here 0 and 1 denote the corresponding morphisms $\Spec k\to\BA^1/\BG_m$). The property is as follows: let $\pi :\BA^1/\BG_m\to (\Spec k)/\BG_m$ be the projection, then we want $\pi_\bullet (f)$ to be (homotopic to) the identity.

To construct $f$, recall that the functor $\pi_\bullet :\Dmod (\BA^1/\BG_m)\to\Dmod ((\Spec k)/\BG_m)$ is canonically isomorphic\footnote{More precisely, the morphism 
$\pi_\bullet\to\pi_\bullet\circ i_\bullet\circ i^*=i^*$ is an isomorphism.} to $i^*$, where $i:(\Spec k)/\BG_m=\{ 0\}/\BG_m\mono\BA^1/\BG_m$ is the closed embedding. In particular, we have a canonical morphism 
\[
1_\bullet (k)\to i_\bullet \circ i^*\circ  1_\bullet (k)= i_\bullet\circ \pi_\bullet\circ 1_\bullet (k)=0_\bullet (k) .
\]
This is $f$.
\end{proof}


\ssec{Passing to D-modules}   \label{ss:D-module details}
Here we will explain the details of the construction of the functor $G:\CA\to\DGCat_{\on{cont}}\, $, which were promised in Subsect.~\ref{sss:passing to Dmod}.

\sssec{A general construction}   \label{sss:2pattern}
Suppose that we have a monoidal category $\CM$ acting on a 2-category $\CC$, a monoidal functor $\CM_0\to\CM$, and a functor 
$\CC_0\to\CC$ compatible with $\CM_0$-actions. Then any lax functor $G_0:\CA_0\to\CC_0$ induces a lax functor $G:\CA\to\CC$, where $\CA$ is obtained from $\CA_0$ by change of enrichment via the monoidal functor $\CM_0\to\CM$.

This remains true even if $\CM_0\, ,\CM$ are monoidal $\infty$-categories and $\CC_0\, ,\CC$ are $(\infty ,2)$-categories.

\sssec{The geometric data}
We will apply the construction of Subsect.~\ref{sss:2pattern} as follows.

We take $\CM_0$ to be the $(2,1)$-category $\St$ defined in Subsect.~\ref{sss:pattern}. We take $\CC_0$ to be the $(3,2)$-category 
$\Corr^{open} (\St )$. The monoidal structure on $\CM_0$ and the action of $\CM_0$ on $\CC_0$ are given by Cartesian product.

We take $\CA_0=\P_{\BA^1}\,$, and we take $G_0$ to be the lax functor $\wt{\Theta}_Z: \P_{\BA^1}\to\Corr^{open}(\St )$ (see Subsect.~\ref{sss:passing to quot}).

%
%

\sssec{Passing to D-modules} \label{sss:Dmod}

We take $\CC$ to be the $(\infty,2)$-category $\DGCat_{\on{cont}}$ defined in Subsect.~\ref{ss:Using lax functor}. 
We take $\CM$ to be the $(\infty,1)$-category obtained from $\DGCat_{\on{cont}}$ by removing all non-invertible 2-morphisms.
The monoidal structure on $\CM$ and the action of $\CM$ on $\CC$ are given by Lurie's tensor product.

Finally, one uses the functors 
$$\Dmod :\St\to\DGCat_{\on{cont}}\,, \;\quad \Dmod  :\Corr^{open} (\St )\to\DGCat_{\on{cont}}$$ constructed in \cite{RG}. Namely, to $\CY\in\St$ one associates the DG category of (complexes of) D-modules on $\CY$, denoted by $\Dmod (\CY )$. To a correspondence
\[
\CY_1\overset{p_1}\longleftarrow \CY\overset{p_2}\longrightarrow \CY_2
\]
one associates the functor $(p_2)_{\bullet}\circ p_1^!:\Dmod (\CY_1)\to\Dmod (\CY_2)$. To an open embedding $j:\CY'\mono\CY$ one associates the natural transformation $(p_2)_{\bullet}\circ p_1^!\to (p_2\circ j)_{\bullet}\circ (p_1\circ j)^!$ that comes from the adjunction between $j_*$ and $j^*=j^!$. The required compatibilities are proved  in \cite{RG}.


\end{document}